\documentclass[10pt]{amsart}
\theoremstyle{plain}
\usepackage{amsmath, amssymb, amscd,graphicx}
\usepackage{hyperref}
\newtheorem{lemma}{Lemma}[section]
\newtheorem{prop}[lemma]{Proposition}
\newtheorem{theorem}[lemma]{Theorem}
\newtheorem{cor}[lemma]{Corollary}
\newtheorem{definition}[lemma] {Definition}

\newtheorem{conj}[lemma]{Conjecture}

\newtheorem*{theorem*}{Theorem}
\newtheorem*{conj*}{Conjecture}
\newtheorem{question}{Question}

\usepackage{tikz}
\usetikzlibrary{matrix,arrows}

\tikzset{cdlabel/.style={above,sloped,
    execute at begin node=$\scriptstyle,execute at end node=$}}
\tikzset{algarrow/.style={->, thick}}
\tikzset{blgarrow/.style={->, thick}}
\tikzset{clgarrow/.style={->, thick}}
\tikzset{tensoralgarrow/.style={double, double equal sign distance, -implies}}
\tikzset{tensorblgarrow/.style={double, double equal sign distance, -implies}}
\tikzset{tensorclgarrow/.style={double, double equal sign distance, -implies}}
\tikzset{tensorelgarrow/.style={double, double equal sign distance, -implies}}
\tikzset{modarrow/.style={->, dashed}}
\tikzset{Amodar/.style={->, dashed}}
\tikzset{Dmodar/.style={->, dashed}}
\tikzset{DAmodar/.style={->, dashed}}

\newcommand{\lab}[1]{{\footnotesize$#1$}}

\newcommand{\scon}{\mathchoice
  {\mbox{\smaller$\#$}}
  {\mbox{\smaller$\#$}}
  {\mbox{\larger[-4]$\#$}}
  {\mbox{\larger[-5]$\#$}}
}

\newcommand{\si}{\mathchoice
  {\mbox{\smaller$\infty$}}
  {\mbox{\larger[-2]$\infty\mkern-.6mu$}}
  {\raisebox{.4pt}{\larger[-6]$\infty$}}
  {\raisebox{.4pt}{\larger[-7]$\infty$}}
}

\newcommand{\sge}{\mathchoice
  {\mbox{\smaller$\ge$}}
  {\mbox{\smaller$\ge$}}
  {\mbox{\larger[-6]$\ge$}}
  {\mbox{\larger[-7]$\ge$}}
}

\newcommand{\sg}{\mathchoice
  {\mbox{\smaller$>$}}
  {\mbox{\smaller$>$}}
  {\mbox{\larger[-6]$>$}}
  {\mbox{\larger[-7]$>$}}
}

\newcommand{\lzero}{\mathchoice
  {\mbox{\smaller$0$}}
  {\mbox{\smaller$0$}}
  {\raisebox{-.7pt}{\larger[-3]$0$}}
  {\raisebox{-.7pt}{\larger[-4]$0$}}
}

\newcommand{\gezero}{\mathchoice
  {\mbox{\smaller$0$}}
  {\mbox{\smaller$0$}}
  {\raisebox{-.8pt}{\larger[-3]$0$}}
  {\raisebox{-.8pt}{\larger[-4]$0$}}
}

\renewcommand{\P}{{{\mathbb P}}}
\newcommand{\Q}{{{\mathbb Q}}}
\newcommand{\R}{{{\mathbb R}}}
\newcommand{\Z}{{{\mathbb Z}}}
\newcommand{\FF}{{{\mathbb F}_p}}

\newcommand{\Qbar}{\overline{\Q}}
\newcommand{\Aa}{\mathbf{A}}
\newcommand{\Bb}{\mathbf{B}}
\newcommand{\Ff}{\mathbb{F}}

\newcommand{\aA}{\mathcal{A}}
\newcommand{\lL}{\mathcal{L}}

\newcommand{\zZ}{\mathcal{Z}}

\renewcommand{\mod}{{{\mathrm{mod\;}}}}

\DeclareMathOperator{\grm}{gr}
\newcommand{\coker}{{{\text {coker} \ }}}
\renewcommand{\ker}{{{\text {ker} \ }}}

\newcommand{\im}{{{\text {im} \ }}}

\newcommand{\Tors}{{{\text {Tors}}}}

\newcommand{\Span}{{{\text {Span} \ }}}
\newcommand{\spinc}{{{\mathrm{Spin}^c}}}
\newcommand{\spi}{\mathfrak{s}}
\newcommand{\gr}{{{\mathrm{gr} }}}

\renewcommand{\epsilon}{{\varepsilon}}

\newcommand{\hfhat}{{{\widehat{HF}}}}

\newcommand{\hfk}{\widehat{HFK}}
\newcommand{\cfk}{\widehat{CFK}}
\newcommand{\cfd}{\widehat{CFD}}
\newcommand{\bfx}{\mathbf{x}}
\newcommand{\bfy}{\mathbf{y}}

\newcommand{\Deltas}{\overline{\Delta}}

\newcommand{\black}{{{S_{\textsc{black}}}}}
\newcommand{\Ybar}{\overline{Y}}
\newcommand{\spibar}{\overline{\spi}}
\newcommand{\tspi}{\widetilde{\spi}}
\newcommand{\taubar}{\overline{\tau}}
\newcommand{\tc}{{{\tau^{\mathrm{c\mkern-1.5mu}}}}}
\newcommand{\dt}{{{\mathcal{D}^{\tau\!}}}}
\newcommand{\dtge}{{{\mathcal{D}_{\mkern-11.2mu\phantom{a}^{\sge \mkern.3mu\gezero}\!}^{\tau}}}}
\newcommand{\dtgz}{{{\mathcal{D}_{\mkern-11.2mu\phantom{a}^{\sg \mkern.3mu\lzero}\!}^{\tau}}}}

\title{Floer Simple Manifolds and L-Space Intervals}
\author{Jacob Rasmussen}
\address{Department of Pure Mathematics and Mathematical Statistics, University of Cambridge, UK}
\email{J.Rasmussen@dpmms.cam.ac.uk}
\thanks{JR was partially supported by EPSRC grant EP/M000648/1}

\author{Sarah Dean Rasmussen}
\address{Department of Pure Mathematics and Mathematical Statistics, University of Cambridge, UK}
\email{S.Rasmussen@dpmms.cam.ac.uk}
\thanks{SDR was supported by EPSRC grant EP/M000648/1}

\begin{document}
\begin{abstract}
An oriented three-manifold with torus boundary admits either
no L-space Dehn filling, a unique L-space filling, or an interval of L-space fillings.
In the latter case, which we call ``Floer simple,'' we construct an invariant 
which computes the interval of L-space filling slopes from the Turaev torsion
and a given slope from the interval's interior.  As applications, we give a new proof
of the classification of
Seifert fibered L-spaces over $S^2\!$,
and prove a special case of 
 a conjecture of Boyer and Clay \cite{BoyerClay}
about L-spaces formed by gluing three-manifolds along a torus.
\end{abstract}
\maketitle

\section{Introduction}

An oriented rational homology 3-sphere \(Y\) is called an L-space if the Heegaard Floer homology 
\(\hfhat(Y)\) satisfies \(\hfhat(Y,\spi) \simeq \Z\) for each \(\spinc\) structure \(\spi\) on \(Y\). Recent interest in the topological meaning of this condition has been stirred by a conjecture of Boyer, Gordon, and Watson \cite{BGW}, which states that a prime oriented three-manifold \(Y\) is an L-space if and only if \(\pi_1(Y)\) is non left-orderable. Subsequently, Boyer and Clay \cite{BoyerClay} studied a relative version of this problem for manifolds with toroidal boundary.

In this paper, we study the set of L-space fillings of a connected manifold \(Y\) with a single torus boundary component. If \(Y\) is such a manifold, we let 
$$ Sl(Y) = \{ \alpha \in H_1(\partial Y) \, | \, \alpha \ \text{is primitive} \}/ \pm 1$$
be the set of slopes on \(\partial Y\). \(Sl(Y)\) is a one-dimensional projective space defined over the rational numbers. If we fix a basis \( \langle \mu, \lambda \rangle \) for \(H_1(Y)\), we can identify \(Sl(Y)\) with \(\Qbar: = \Q \cup \{ \infty \}\) via the map \(a \mu + b \lambda \mapsto a/b\). We denote by \(Y(\alpha)\) the closed manifold obtained by Dehn filling \(Y\) with slope \(\alpha\), and let \(K_\alpha \subset Y(\alpha)\) be the core of the filling solid torus. 

\begin{definition} If \(Y\) is a compact connected oriented three-manifold with torus boundary, 
$$\lL(Y) = \{ \alpha \in Sl(Y) \, | \, Y(\alpha) \ \text{is an L-space}\}$$
 is the set of {\em L-space filling slopes} of \(Y\). 
\end{definition}

For the set \(\lL(Y)\) to be nonempty, we must have \(b_1(Y)=1\), which implies that \(Y\) is a rational homology \(S^1 \times D^2\). In this paper, we will restrict our attention to manifolds with multiple L-space fillings: that is, for which \(|\lL(Y)| >1\). Such manifolds can be easily characterized in terms of their Floer homology. Recall that a knot \(K\) in a rational homology sphere \(\overline{Y}\) is {\em Floer simple} \cite{HeddenLens} if the knot Floer homology \(\widehat{HFK}(K) \simeq \Z^{|H_1(Y)|}\). Equivalently, \(K\) is Floer simple if  \(\overline{Y}\) is an L-space and the spectral sequence from \(\widehat{HFK}(K)\) to \(\hfhat(\overline{Y})\) degenerates. 

\begin{definition} A compact oriented three-manifold \(Y\)  with torus boundary is {\em Floer simple} if it has some Dehn filling \(Y(\alpha)\) whose core \(K_\alpha\) is a Floer simple knot in \(Y(\alpha)\). 
\end{definition}

Then we have 
\begin{prop}
\label{Prop:MultiL}
\(|\lL(Y)|>1\) if and only if $Y\!$ is Floer simple. 
\end{prop}

If \(K_\alpha \subset Y(\alpha)\) is Floer simple, then the Floer homology of any surgery on \(K_\alpha\) can be determined from \(\hfk(K_\alpha)\) using the Ozsv{\'a}th-Szab{\'o} mapping cone. The knot Floer homology, in turn, is determined by the Turaev torsion \(\tau(Y)\) via the relation 
$$\chi(\hfk(K_\alpha)) \sim (1-[\alpha]) \tau(Y)$$
established in Proposition~\ref{Prop:Chi}.
It follows that if \(Y\) is Floer simple, then the Floer homology of any Dehn filling of \(Y\) can be determined from the Turaev torsion together with a single \(\alpha \in \lL(Y)\). In particular, we can determine \(\lL(Y)\) from  this data, as described below.

Write \(H_1(Y) = \Z \oplus T\), where \(T\) is a torsion group, and let \(\phi:H_1(Y) \to \Z\) be the projection. Properly normalized,  \(\tau(Y)\) can be written as a sum 
$$\tau(Y) = \sum_{\substack{h \in H_1(Y) \\ \phi(h) \geq 0}} a_h [h],$$
where \(a_h = 1\) for all but finitely many \(h\in H_1(Y)\) with \(\phi(h)>0\), and \(a_0 \neq 0 \). For example, if \(H_1(Y) = \Z\), then
$$\tau(Y) = \frac{\Delta(Y)}{1-t} \in \Z[[t]],$$
where the Alexander polynomial \(\Delta(Y)\) is normalized to be an element of \(\Z[t]\) and we expand the denominator as a Laurent series in positive powers of \(t\). 

\begin{prop}
\label{Prop:TorsionCoeff}
When \(Y\) is Floer simple, every coefficient \(a_h\) of \(\tau(Y)\) is either \(0\) or \(1\). 
\end{prop}

Let \(S[\tau(Y)] = \{h \in H_1(Y) \, | \, a_h\neq 0\}\) 
denote the {\em support} of \(\tau(Y)\), and let 
\(\iota:H_1(\partial Y) \to H_1(Y)\) be the map induced by inclusion. 
\begin{definition}
If \(Y\) is a Floer simple manifold, we define 
$$ \dt(Y) = \{x-y\,|\, x \notin S[\tau(Y)], y \in S[\tau(Y)],\phi(x)\geq \phi(y)\} 
\cap \mathrm{im} \, \iota \subset H_1(Y),$$
and write 
\(\dtgz(Y)\) for the subset of \(\dt(Y)\) consisting of those elements with  \(\phi(h)>0\). 
\end{definition}

Let \([l] \in Sl(Y)\) be the homological longitude ({\it i.e.} \(l\) is a primitive element of \(H_1(Y)\) such that \(\iota(l)\) is torsion.) We can now state our first main theorem:  

\begin{theorem}
\label{Thm:LY}
If \(Y\) is  Floer simple, then either
$\dtgz(Y) = \emptyset$ and \(\lL(Y) = Sl(Y) \setminus [l]\),
or $\dtgz(Y) \neq \emptyset$ and 
\(\lL(Y)\) is a closed  interval whose endpoints are consecutive elements of 
$\iota^{-1}(\dtgz(Y))$.
\end{theorem}

Given \(\tau(Y)\) and a Floer simple filling slope \(\alpha\) for \(Y\), it is thus straightforward to determine \(\lL(Y)\): the torsion determines the set $\dt(Y)$, and \(\lL(Y)\) is the smallest interval with endpoints in $\iota^{-1}(\dtgz(Y))$ which contains \(\alpha\) in its interior.

\subsection{Splicing}
Theorem~\ref{Thm:LY} can be used to address a problem raised by Boyer and Clay in \cite{BoyerClay}. Suppose that \(Y_1\) and \(Y_2\) are rational homology solid tori, and that \(\varphi:\partial Y_1 \to \partial Y_2\) is an orientation reversing diffeomorphism. The manifold \(Y_\varphi = Y_1 \cup_\varphi Y_2\) is said to obtained by {splicing} \(Y_1\) and \(Y_2\) together by \(\varphi\).

In \cite{BoyerClay}, Boyer and Clay studied how the presence of structure \((*)\) on Dehn fillings of the pieces \(Y_1\) and \(Y_2\) relates to the presence of structure \((*)\) on the splice $Y_\varphi$, where structure \((*)\) could be one of three things: 1)  a coorientable taut foliation; 2) a left-ordering on \(\pi_1(Y_\varphi)\); or 3) a nontrivial class in 
$HF^{red}(Y_\varphi)$ (as $HF^{red}$ vanishes on, and only on, L-spaces). When \(Y_1\) and \(Y_2\) are graph manifolds, they obtained very strong results in cases 1) and 2), in addition to less complete results in the third case. The analogy with the first two cases suggests the following conjecture, which is implicit in the work of Boyer and Clay and stated explicitly in certain cases by Hanselman \cite{Hanselman}.

\begin{conj}
\label{Conj:Splice}
Suppose that \(Y_1\) and \(Y_2\) as above are boundary incompressible, and let $\lL^{\circ}_i$ be the interior of $\lL(Y_i) \subset Sl(Y_i)$. Then \(Y_\varphi\) is an L-space if and only if \(\varphi_*(\lL^{\circ}_1) \cup \lL^{\circ}_2 = Sl(Y_2)\). 
\end{conj}
In particular, the conjecture says that in order for \(Y_\varphi\) to be an L-space, both \(Y_1\) and \(Y_2\) must be Floer simple.  Our second main result is

\begin{theorem} 
\label{Thm:Splice}
Suppose that \(Y_1\) and \(Y_2\) as above are Floer simple and have \(\dt \neq \emptyset\), and that \(\varphi_*(\lL^{\circ}_1) \cup \lL^{\circ}_2 \neq \emptyset\). Then \(Y_\varphi\) is an L-space if and only if \(\varphi_*(\lL^{\circ}_1) \cup \lL^{\circ}_2  = Sl(Y_2)\). 
\end{theorem}
Hanselman and Watson \cite{HanWat} have proved a  similar theorem using bordered Floer homology.  
The restriction that \(\varphi_*(\lL^{\circ}_1) \cap \lL^{\circ}_2 \neq \emptyset\) represents a limitation of our approach, rather than anything intrinsic to the problem. To be specific, Theorem~\ref{Thm:Splice} is proved by writing \(Y_\varphi\) as surgery on a connected sum of Floer simple knots. When  \(\varphi_*(\lL^{\circ}_1) \cap \lL^{\circ}_2 = \emptyset\), we have no convenient way of representing the splice as surgery on a knot in an L-space. In contrast, Hanselman and Watson's approach does not require this hypothesis, but does need a condition on the bordered Floer homology, which they call {\em simple loop type}. In a subsequent joint paper \cite{HRRW}, it is shown that the conditions of being Floer simple and being simple loop type  are equivalent thus enabling us to remove the hypothesis that  \(\varphi_*(\lL^{\circ}_1) \cap \lL^{\circ}_2 \neq \emptyset\). The proof of this fact relies on Proposition~\ref{Prop:FloerSimpleCFD} of the current paper, where
 we explicitly compute the bordered Floer homology \(\cfd(Y,\mu,\lambda)\) of a Floer simple manifold \(Y\) for an appropriate choice of \(\mu, \lambda \in H_1(\partial Y)\) parametrizing \(\partial Y\). 
 
We briefly discuss those aspects of Conjecture~\ref{Conj:Splice} which are not covered by Theorem~\ref{Thm:Splice} and its generalizations. As stated, the conjecture implies that a Floer simple manifold \(Y\) with \(\dt (Y)= \emptyset\) is boundary compressible. This is easily seen to be the case when \(H_1(Y) \simeq \Z\), or more generally, when \(Y\) is semi-primitive ({\it c.f.} Proposition~\ref{Prop:FST} below), but in general we have very little idea how to address this question. (Indeed, this seems like the weakest point of the conjecture.) The other situation which is not addressed  by Theorem~\ref{Thm:Splice} is the case where one or both of \(Y_1\) and \(Y_2\) is not Floer simple. 
 It seems plausible that bordered Floer homology could be used to prove the conjecture when  \(|\lL(Y_1)| = 1\) and \(|\lL(Y_2)|>1\), or when
 \(|\lL(Y_1)| = |\lL(Y_2)|=1\). In contrast, the case where one or both of the \(Y_i\) has {\em no} L-space fillings seems considerably more difficult to address with current technology. 

\subsection{Floer homology solid tori}
The class of Floer simple manifolds with \(\dtgz = \emptyset\) is of special interest. If \(Y\) is a rational homology
\(S^1 \times D^2\), we say that \(Y\) is {\em semi-primitive} if the torsion subgroup of \(Y\) is contained in the image of \(\iota\), and that \(Y\) has genus \(0\) if \(H_2(Y,\partial Y)\) is generated by a surface of genus \(0\). 

\begin{prop}
\label{Prop:FST}
If \(Y\) is semi-primitive, the following conditions are equivalent:
\begin{enumerate}
\item \(Y\) is Floer simple and \(\dtgz(Y) = \emptyset\).
\item \(Y\) is Floer simple and has genus \(0\). 
\item \(Y\) has genus \(0\) and has an L-space filling. 
\end{enumerate}
\end{prop}
For example, if 
 \(K\subset S^1 \times S^2\) has a lens space surgery, then the complement of \(K\) satisfies the conditions of the proposition. Such knots have been studied by Berge \cite{BergeTorus}, Gabai \cite{GabaiST}, Cebanu \cite{Cebanu}, and 
Buck, Baker and Leucona \cite{BBL}. Other examples of such manifolds are discussed in section~\ref{Subsec:FSTExamples}. 

 The conditions of Proposition~\ref{Prop:FST} are closely related to Watson's notion of a {\em Floer homology solid torus}. Suppose that \(Y\) is a rational homology \(S^1 \times D^2\) with homological longitude \(l\), and that 
 \(m \in H_1(\partial Y)\) satisfies \(m \cdot l =1\). 
\begin{definition} \cite{WatsonFST}
\(Y\) is a {\em Floer homology solid torus} if \(\widehat{CFD}(Y,m,l)\simeq \widehat{CFD}(Y,m+l, l)\).  
\end{definition}

\begin{prop}
\label{Prop:Fake=Floer}
If \(Y\) satisfies the conditions of Proposition~\ref{Prop:FST}, then it is a Floer homology solid torus. 
\end{prop}

Manifolds with \(\dtgz(Y) = \emptyset\) play an important role in the  notion of  NLS detection introduced by Boyer and Clay in \cite{BoyerClay}. If \(Y\) is a rational homology \(S^1\times D^2\) and \(\alpha \in Sl(Y)\), \(\alpha\) is said to be {\em strongly NLS detected} if \(Y(\alpha)\) is not an L-space; \(\alpha\) is {\em NLS detected} if certain splicings of \(Y\) with a family of Floer homology solid tori are not L-spaces. (For the precise definition, see section~\ref{SubSec:NLS}). By Theorem~\ref{Thm:LY}, the set of strongly NLS detected slopes is either a single point, an open interval in \(Sl(Y)\), or all of \(Sl(Y)\). By combining Theorem~\ref{Thm:Splice} with some direct geometric computation, we can show

\begin{cor}
\label{Cor:NLS}
If \(Y\) is a rational homology \(S^1\times D^2\),  the set of NLS detected slopes in \(Sl(Y)\) is the closure of the set of strongly NLS detected slopes.  \end{cor}

\subsection{Seifert fibred spaces}

One of the key motivating examples for the conjecture of \cite{BGW} is the class of Seifert-fibred spaces. Indeed, building on work of
Ozsv{\'a}th, Szab{\'o}, Mati{\'c}, Naimi, Jankins, Neumann, Eisenbud, and Hirsch
\cite{OSGen, LiscaMatic, Naimi, JankinsNeumann, EHN},
Lisca and Stipsicz proved
\begin{theorem}
\label{Thm:LS}
\cite{LSIII} A Seifert fibred space over \(S^2\) is an L-space if and only it does not admit a coorientable taut foliation. 
\end{theorem}
In combination with a result of Boyer, Rolfsen, and Wiest \cite{BRW},
this also implies that a Seifert-fibred space over \(S^2\) has non left-orderable \(\pi_1\) if and only if it is an L-space. 
The set of Seifert fibred spaces over \(S^2\) which admit a coorientable taut foliation was explicitly described by Jankins and Neumann \cite{JankinsNeumann} and Naimi \cite{Naimi},
building on a result of Eisenbud, Hirsch, and Neumann \cite{EHN}.

Any Seifert-fibred space over \(S^2\) can be obtained by Dehn filling a Seifert fibred space over \(D^2\). It follows easily from work of Ozsv{\'a}th and Szab{\'o} \cite{OSSF} that any Seifert fibred space over \(D^2\) is Floer simple, so we can compute the set of L-space filling slopes using Theorem~\ref{Thm:LY}. The resulting description of the set of Seifert fibred spaces which are not L-spaces agrees with the Jankins-Neumann set, thus giving a new direct proof of Theorem~\ref{Thm:LS}. 
\subsection{Discussion}
We conclude with some questions about about Floer simple manifolds and their relation to the conjecture of Boyer, Gordon, and Watson. First, we recall the statement of the conjecture.

\begin{conj}
\label{Conj:BGW}
\cite{BGW} If \(Y\) is a oriented, closed, prime three-manifold, then \(Y\) is an L-space if and only if \(\pi_1(Y)\) is non left-orderable. 
\end{conj}

A potentially more tractable subset of this problem, raised by Boyer and Clay \cite{BoyerClay} is:

\begin{question} Suppose \(Y\) is Floer simple. Is \(\pi_1(Y(\alpha))\) non left-orderable equivalent to \(\alpha\) being an element of \(\lL(Y)\)? 
\end{question}

The characterization of \(\lL(Y)\) given in Theorem~\ref{Thm:LY} should make it possible to conduct more detailed tests of Conjecture~\ref{Conj:BGW}. Since there is already considerable experimental evidence in support of the conjecture, we should also consider what circumstances might explain a positive answer to Question 1. One possible explanation is that the condition of being Floer simple is correlated with some strong geometrical property, which in turn can be related to orderings of \(\pi_1\). 

\begin{question} Is there a geometric characterization of Floer simple manifolds which can be stated without reference to Floer homology?
\end{question}

More generally, we think that  Floer simple manifolds are a natural class of manifolds  whose geometrical properties should be investigated for their own sake. Some evidence in support of this idea is provided by the  frequency of Floer simple manifolds among geometrically simple 3-manifolds (as measured by the SnapPea census).
Proposition~\ref{Prop:MultiL} may lead readers familiar with the example of L-space knots in \(S^3\) to suspect that the class of Floer simple manifolds is relatively small, but this is not the case. Of the 59,068 rational homology \(S^1\times D^2\)'s in the SnapPy census of manifolds triangulated by at most 9 ideal tetrahedra, 
nearly \(20\%\) have multiple finite fillings, and are thus certifiably Floer simple. Moreover, more than two-thirds of the remaining manifolds have Turaev torsion compatible with their being Floer simple. It seems likely that many of these manifolds are Floer simple as well. (The authors thank Tom Brown for sharing these statistics with them.) For those who like other geometries, we note that every Seifert fibred rational homology \(S^1\times D^2\) is Floer simple. 

It would be interesting to know what happens to the density of Floer simple manifolds as the complexity increases. Perhaps the most basic question we could ask along these lines is 

\begin{question}
Are there infinitely many irreducible Floer simple manifolds with the same Turaev torsion? 
 \end{question}

\subsection{Organization}
The remainder of the paper is organized as follows. In section~\ref{Sec:HFK}, we review some facts about knot Floer homology and the Ozsv{\'a}th-Szab{\'o} mapping cone. These are used in section~\ref{Sec:FloerSimple} to prove Proposition~\ref{Prop:MultiL} and to give a characterization of when a given surgery on a Floer simple knot produces an L-space. In this section, we also explain how to compute the bordered Floer homology of a Floer simple manifold.
Theorem~\ref{Thm:LY} is proved in In Section~\ref{Sec:LY}. In  Section~\ref{s: Seifert fibered L-spaces}  we apply Theorem~\ref{Thm:LY} to Seifert fibred spaces, thus giving a new proof of Theorem~\ref{Thm:LS}. The proof of Theorem~\ref{Thm:Splice} is given in Section~\ref{s: L-space if and only if intervals cover.}. Finally, in Section~\ref{Sec:dtgz}, we discuss manifolds with \(\dtgz = \emptyset\). 

\vskip0.05in
\noindent{\bf Acknowledgements:} The authors  would like to thank Steve Boyer, Tom Brown, Adam Clay, Tom Gillespie, Jonathan Hanselman, Robert Lipshitz, Saul Schleimer, Faramarz Vafaee, and Liam Watson for helpful conversations. We also thank the organizers of the 9th William Rowan Hamilton conference in Dublin, which helped to get this project started.

\section{Knot Floer homology and the Ozsv{\'a}th-Szab{\'o} mapping cone}
\label{Sec:HFK}
In this section, we briefly recall some facts about knot Floer homology 
\cite{OSHFK1, jakethesis, OSInt}
which will be used in what follows. First, let us fix some notation. Throughout this section, we assume that  \(K \subset \Ybar\) is an oriented knot in a rational homology sphere. We let \(Y = \overline{Y} \setminus \nu(K)\) be its complement, and denote by \(\mu \in H_1(\partial Y)\) the class of its meridian.
Furthermore, we let \(T \subset H_1(Y)\) be the torsion subgroup, and  denote by  \(\phi:H_1(Y) \to \Z\) the projection from \(H_1(Y)\) to \(H_1(Y)/T\simeq \Z\), where the isomorphism is chosen so that \(\phi(\mu) > 0\). 

\subsection{Knot Floer homology}
 The knot Floer homology \(\hfk(K)\) is a finitely generated abelian group with an absolute \(\Z/2\) grading. 
 It decomposes as a direct sum \(\hfk(K) = \oplus \hfk(K,\spi)\), where \(\spi\) runs over the set \( {\mathrm{Spin}^c}(Y, \partial Y)\) of relative \({\mathrm{Spin}^c}\) structures on \((Y, \partial Y)\). \({\mathrm{Spin}^c}(Y,\partial Y)\) is an affine copy of \(H_1(Y)\) ({\em aka} \(H_1(Y)\) torsor); it has a free transitive action of \(H_1(Y)\). The group \(\hfk(K,\spi)\) is trivial for all but finitely many \(\spi \in {\mathrm{Spin}^c}(Y, \partial Y)\).
 
 Given \(\spi \in {\mathrm{Spin}^c}(Y, \partial Y)\), we consider the formal sum
$$\chi_\spi(\hfk(K)) :=  \sum_{h \in H_1(Y)} \chi(\hfk(K,\spi+h)) [h],$$
where \(\chi(\hfk(K,\spi))\) is defined using the absolute \(\Z/2\) grading. We view \(\chi_\spi(\hfk(K))\) as an element of the group ring \(\Z[H_1(Y)]\); it is known as the {\em graded Euler characteristic of \(\hfk(K)\)}. 
Clearly $$ \chi_{\spi'}(\hfk(K)) = [\spi-\spi'] \chi_\spi(\hfk(K)).$$
From now on, we will drop \(\spi\) from the notation and view \(\chi(\hfk(K))\) as an element of \(\Z[H_1(Y)]\), well defined up to global multiplication by elements of \(H_1(Y)\). We write \(x \sim y\) if \(x,y \in \Z[H_1(Y)]\) satisfy \(x = [h]y\) for some \(h \in H_1(Y)\).

For knots in \(S^3\), it is well-known that \(\chi(\hfk(K))\) is the Alexander polynomial of \(K\). More generally, we have 

\begin{prop}
\label{Prop:Chi}
$ \chi (\hfk(K)) \sim (1 - [\mu]) \tau(Y),$ where \(\tau(Y)\) is the Turaev torsion of \(Y\). 
\end{prop}
\begin{proof}
 \(\widehat{HFK}(K)\) can be identified with the sutured Floer homology \(SFH(Y, \gamma_\mu)\) \cite{Juhasz1}, where the suture  \(\gamma_\mu\)  consists of two parallel copies of \(\mu\).  The Euler characteristic of the sutured Floer homology can be described as an appropriately formulated torsion \cite{FJR}. When \(\partial Y\) is toroidal, this torsion can be expressed in terms of the Turaev torsion, as in  Lemma 6.3 of \cite{FJR}. (This lemma was stated for links in \(S^3\), but the proof carries through unchanged.)
\end{proof}

{\it A priori},   \(\tau(Y) \) is an element of the field \(Q(H_1(Y))\)  obtained by inverting all elements of \(\Z[H_1(Y)]\) which are not zero divisors. Choose any primitive \(\mu \in H_1(\partial Y)\) with \(\phi(\mu) \neq 0\); then \(1- [\mu]\) will not be a zero divisor in \(\Z[H_1(Y)]\). It follows from the proposition that   \(\tau(Y) \in \Z[H_1(Y)][(1-[\mu])^{-1}] \subset Q(H_1(Y))\).

Writing \(  (1-[\mu])^{-1}  = \sum_{i = 0} ^ \infty\,  [\mu]^i\)
allows us to embed
$\Z[H_1(Y)][(1-[\mu])^{-1}]$ in the Novikov ring
$$ \Lambda_\phi[H_1(Y)] = \Big\{ \sum_{h \in H_1(Y)} a_h [h] \ \Big| \ \#\{h \, | \, a_h \neq 0, \phi(h)<k\}<\infty \ \text{for all} \ k\Big\}.$$
We will view \(\tau(Y)\) as an element of \(\Lambda_\phi[H_1(Y)]\). By choosing  a  splitting \(H_1(Y) \simeq \Z \oplus T\), we can identify \(\Lambda_\phi[H_1(Y)] \) with the Laurent series ring \(\Z[t^{-1},t]] \otimes \Z[T]\),
which we shall later sometimes call the ``Laurent series group ring.''
  
 As an element of  the Novikov ring, \(\tau(Y)\) is  well-defined up to multiplication by elements of \(H_1(Y)\). We  shall always  normalize so that \(\tau(Y)\) has the form 
 $ \tau(Y) = \sum_h a_h [h]$, where \(a_h = 0 \) for all \(h\) with \(\phi(h) <0\), and 
 \(a_0 \neq 0 \).
 
 If \(H_1(Y) = \Z\),  it is well-known that  $ \tau(Y) \sim \Delta(Y)/(1-t)$, where \(\Delta(Y)\) is the Alexander polynomial of \(Y\). 
 More generally, if 
  \(\Phi: \Lambda_\phi[H_1(Y)] \to \Z[t^{-1},t]]\) is the map induced by the projection \(\phi:H_1(Y) \to \Z\), we define  
$$ \taubar(Y)= \Phi(\tau(Y)) \quad \text{and} \quad  \Deltas(Y) = (1-t) \taubar(Y).$$
 Note that in general, \(\Deltas(Y) \neq \Delta(Y)\); an interesting example to consider is the connected sum 
 \(Y = Z\mkern1.5mu\#\mkern1.5mu (S^1\mkern-2mu \times\mkern-2mu D^2)\), where \(b_1(Z) = 0\). This manifold 
 has \(\Delta(Y) = 0\), but \(\Deltas(Y) = |H_1(Z)|\). 
 
 If \(K\) is a knot in \(S^3\), it is well known that  \(\deg \Delta(t) \leq 2g(K)\), and \(\Delta(K)|_{t=1} = 1\). The following result is a simultaneous generalization of these two facts. 

\begin{prop}[\cite{Turaev} Lemma II.4.5.1 and Theorem II.4.2.1]
\label{Prop:AlexG} 
If \(\| Y\|\) is the Thurston norm of a generator of \(H_2(Y, \partial Y)\) and 
 \(\tau(Y)\) is normalized as above, then 
\(a_h = 1\) for all \(h\in H_1(Y)\) with \(\phi(h)> \|Y\| \). 
\end{prop}

More generally, it is known that \(\hfk(K)\) determines both the Thurston norm of \(Y\) and whether it is fibred \cite{YiNiThurstonNorm, YiNiFibred, Juhasz}. Since the knot Floer homology of a Floer simple knot is determined by its Euler characteristic, we have

\begin{cor}
\label{Cor:TNorm}
If \(Y\) is boundary incompressible and Floer simple,  \(\|Y\| = \deg \Deltas(Y) - 1\). If \(Y\) is also irreducible, 
then \(Y\) fibres over \(S^1\) if and only if \(\Deltas(Y)\) is monic. 
\end{cor}

\subsection{Differentials}
The knot Floer homology of \(K\) can be used to compute the Floer  homology of surgeries on \(K\). Before we explain how to do this, we must understand the relation between \(\hfk(K)\) and \(\hfhat(\Ybar)\). 

We begin by discussing \({\mathrm{Spin}^c}\) structures. There are maps \(i_v,i_h:{\mathrm{Spin}^c}(Y,\partial Y) \to   {\mathrm{Spin}^c}(\Ybar)\) which respect the action of \(H_1(Y)\), in the sense that \(i_v(\spi+a) = i_v(\spi) + i_*(a)\) and \(i_h(\spi+a) = i_h(\spi) + i_*(a)\)
 where \(i_*:H_1(Y) \to H_1(\Ybar)\) is the map induced by inclusion. Moreover, \(i_v(\spi) - i_h(\spi) = i_*(\lambda)\), where \(\lambda\) is a longitude of \(K\). We define an equivalence relation on \({\mathrm{Spin}^c}(Y,\partial Y)\) by declaring \(\spi_1 \sim \spi_2\) if \(i_v(\spi_1) = i_v(\spi_2)\). It is easy to see that this is the same as requiring that \(i_h(\spi_1) = i_h(\spi_2)\), and that the equivalence classes are orbits of \({\mathrm{Spin}^c}(Y,\partial Y)\) under the action of \(\mu\). 

Let \(\tspi\) be an equivalence class in \({\mathrm{Spin}^c}(Y,\partial Y)\). 
After we choose some auxiliary data (a doubly pointed Heegaard diagram for \(K\)), Heegaard Floer homology constructs for us a graded group 
$$\cfk(K,\tspi) = \bigoplus_{\spi\in \tspi} \cfk(K,\spi)$$
together with maps \(d_0,d_v,d_h:\cfk(K,\tspi) \to \cfk(K,\tspi)\), which are filtered with respect to the \({\mathrm{Spin}^c}\) grading in the following sense: if \(x \in \cfk(Y,\spi)\), then 
$d_0x \in \cfk(Y,\spi)$, $d_vx \in \oplus_{k<0} \cfk(Y,\spi + k \mu)$ and $d_hx \in \oplus_{k>0} \cfk(Y,\spi + k \mu).$ These differentials satisfy the relations \(d_0^2=(d_0+d_v)^2 = (d_0+d_h)^2 = 0\). Furthermore, we have
\begin{align*} H(\cfk(K,\spi),d_0) & = \hfk(K,\spi), \\
H(\cfk(K,\tspi),d_0+d_v) & = \hfk(\Ybar, i_v(\spi)), \\
H(\cfk(K,\tspi),d_0+d_h) & = \hfk(\Ybar, i_h(\spi)).
\end{align*}

The \({\mathrm{Spin}^c}\) grading provides a natural filtration on the latter two complexes, in the sense that 
\(\oplus_{k<n} \cfk(K,\spi+k \mu)\) is a subcomplex of \((\cfk(K,\tspi),d_0+d_v)\)
and \(\oplus_{k>n} \cfk(K,\spi+k \mu)\) is a subcomplex of \((\cfk(K,\tspi),d_0+d_h)\).
These filtrations give rise to spectral sequences whose \(E_1\) term is \(\hfk(K,\tspi)\). We denote by 
\(\tilde{d}_v, \tilde{d}_h\) the induced differentials on the \(E_1\) term, so that {\it e.g.} \(\cfk(K,d_0+d_v)\) is homotopy equivalent to \(\hfk(K,\tilde{d}_v)\). (Note that these are not the same as the \(d_1\) differentials in the spectral sequence.)

\begin{definition} For each \(\spi \in {\mathrm{Spin}^c}(Y)\), the {\em bent complex} is
$A_{K,\spi} = (\cfk(K,\tspi),d_{\spi})$, where for \(x \in \cfk(K,\spi+k \mu)\), 
$$d_\spi(x) = \begin{cases} d_0(x)+ d_v(x) \quad & k<0 \\ d_0(x)+d_v(x)+d_h(x) \quad & k=0 \\ d_0(x)+d_h(x) \quad & k>0 \end{cases}.
$$
\end{definition}
The bent complexes measure the 
Heegaard Floer homology 
of large integer surgery on \(K\): \(H(A_{K,\spi}) \simeq \hfhat(Y(N \mu + \lambda), i_n(\spi))\) for sufficiently large \(N\) and an appropriately chosen \({\mathrm{Spin}^c}\) structure \(i_N(\spi)\) on the filling. 

The existence of the  \({\mathrm{Spin}^c}\) filtration means there are chain maps
\begin{align*}
\pi_v&: A_{K,\spi} \to  (\cfk(K,\tspi),d_0+d_v) \\
\pi_h&: A_{K,\spi} \to  (\cfk(K,\tspi),d_0+d_h)
\end{align*}
given by 
$$ \pi_v(x) = \begin{cases} 0 \quad & k> 0 \\ x \quad & k\leq 0 \end{cases} \quad \quad 
 \pi_h(x) = \begin{cases} x \quad & k \geq 0 \\ 0 \quad & k  < 0 \end{cases} 
 $$
 for \(x \in \cfk(\spi+k\mu)\). 
\subsection{The Ozsv{\'a}th-Szab{\'o} mapping cone}
Let \(\lambda\) be a longitude for \(K\), so that \(\mu \cdot \lambda = 1\). 
The mapping cone of Ozsv{\'a}th and Szab{\'o}  \cite{OSInt} relates the Heegaard Floer homology of the filling \(Y(\lambda)\) to the knot Floer homology of \(K\). We recall its construction here. 

Since \(i_h(\spi-\lambda) = i_v(\spi)\), we have
$$H(\cfk(K,\spi - \lambda),d_0+d_h)\simeq \hfhat(Y,i_v(\spi)) \simeq H(\cfk(K,\spi),d_0+d_v).$$ 
This isomorphism is realized by a chain homotopy equivalence
$$j: (\cfk(K,\spi - \lambda),d_0+d_h) \to (\cfk(K,\spi),d_0+d_v).$$
(The map on homology induced by \(j\) is the canonical isomorphism of \cite{JT}, although we will not use this fact here.)

For \(\spi \in {\mathrm{Spin}^c}(Y,\partial Y)\), let \(B_{K,\spi} = (\cfk(K,\tspi),d_0+d_v)\). We form two chain complexes 

$$\mathbb{A}(K) = \bigoplus_{\spi \in {\mathrm{Spin}^c}(Y)} A_{K,\spi} \quad \text{and} \quad  \mathbb{B}(K) = \bigoplus_{\spi \in {\mathrm{Spin}^c}(Y)} B_{K,\spi}.$$
There is a chain map \(f_\lambda :\mathbb{A}(K) \to \mathbb{B}(K)\) given by  \(f =\pi_v + j \circ \pi_h\). (So if \(x \in A_{K,\spi}\), \(f_\lambda(x)\) is a sum of terms in \(B_{K,\spi}\) and \(B_{K,\spi+\lambda}\).) Let \(\mathbb{X}_\lambda(K)\) be the mapping cone of \(f_\lambda\). In \cite{OSInt} , Ozsv{\'a}th and Szab{\'o} prove  

\begin{theorem}
\label{Thm:MappingCone}
 \cite{OSInt}
\(\hfhat(Y(\lambda)) \simeq H_*(\mathbb{X}_\lambda(K))\).
\end{theorem}

We make some remarks on the construction. First, it is easy to see that the complex \(\mathbb{X}_\lambda(K)\) decomposes as a direct sum of complexes whose underlying groups are of the form
$$ \mathbb{X}_\lambda(K,\spi) = \bigoplus_{n \in \Z} A_{K,\spi + n \lambda} \oplus \bigoplus_{n \in \Z} B_{K,\spi + n \lambda}.$$
The summands are on one to one correspondence with elements of the quotient \(H_1(Y)/\langle \lambda \rangle \simeq H_1(Y(\lambda))\). The resulting decomposition on homology corresponds to the decomposition of \(\hfhat(Y(\lambda))\) by \({\mathrm{Spin}^c}\) structures. 

Second, if \(\FF\) is  the field of order \(p\), where \(p\) is a prime, then we can form the complex 
\(\mathbb{X}_\lambda(K;\FF) =  \mathbb{X}_\lambda(K)\otimes \FF\). It follows from the universal coefficient theorem that 
\(\hfhat(Y(\lambda);\FF) \simeq H_*(\mathbb{X}_\lambda(K;\FF))\).

Finally, it is often convenient to work with the homology of the complexes
 \(A_{K,\spi}\) and \(B_{K,\spi}\), rather than the complexes themselves. We can do this if we use field coefficients. Specifically, fix a field \(\mathbb{F}_p\), and let
 \(\Aa_{K,\spi} = H(A_{K,\spi} \otimes \mathbb{F}_p)\), \(\Aa(K) = \oplus \Aa_{K,\spi}\), \(\Bb_{K,\spi} = H(B_{K,\spi}\otimes\mathbb{F}_p)\), \(\Bb(K) = \oplus \Bb_{K,\spi}\). 
 Similarly, let \(v:\Aa_{K,\spi} \to \Bb_{K,\spi}\) be the map induced by \(\pi_v\), and \(h:\Aa_{K,\spi} \to \Bb_{K,\spi+\lambda}\) be the map induced by \(j \circ \pi_h\). Finally, let \(C_\lambda(K;\mathbb{F}_p)\) be the chain complex whose underlying group is \(\Aa(K) \oplus \Bb(K)\), with differential given by \(dx = v(x)+h(x)\) for \(x \in \Aa(K)\), \(dy = 0\) for \(y \in \Bb(K)\). 
 
 \begin{cor}
\( \hfhat(Y(\lambda);\mathbb{F}_p) \simeq H(C_\lambda(K;\mathbb{F}_p))\).
 \end{cor}
 
 \begin{proof}
 The short exact sequence $$0 \to \mathbb{B}(K)\otimes \mathbb{F}_p \to \mathbb{X}_\lambda(K;\mathbb{F}_p) \to \mathbb{A}(K)\otimes \mathbb{F}_p \to 0$$ gives a long exact sequence 
$$ \to \Bb(K) \to \hfhat(Y(\lambda);\mathbb{F}_p) \to \Aa(K) \to \Bb(K) \to $$
whose boundary map is given by \(v+h\). An exact sequence over a field splits, so we get the statement of the corollary.  
 \end{proof}

\subsection{Splicing and surgery} Suppose \(Y_1\) and \(Y_2\) are rational homology solid tori, and that \(\varphi:\partial Y_1 \to \partial Y_2\) is an orientation reversing diffeomorphism. The manifold \(Y_\varphi = Y_1 \cup_\varphi Y_2\) is obtained by {splicing} \(Y_1\) and \(Y_2\) together along \(\varphi\).  Choose a slope \(\mu_1 \in Sl(\partial Y_1)\), and let \(\mu_2 = \varphi_*(\mu)\) be its image in \(Sl(\partial Y_2)\). Let \(\Ybar_i = Y_i(\mu_i)\)  be the corresponding Dehn fillings, and let  \(K_i = K_{\mu_i}\)  be their cores. 

\begin{lemma}
\label{Lem:SpliceSurgery}
\(Y_\varphi\) can be obtained by integral surgery on \(K_1 \# K_2 \subset \Ybar_1 \# \Ybar_2\). 
\end{lemma}

This is well-known, but an understanding of the proof will be useful in what follows, so we sketch it here.
\begin{proof}
Let \(Y'\) be the complement of \(K_1\#K_2\). \(Y'\) is obtained by identifying an annulus \(\nu(\mu_1) \subset \partial Y_1\) with its image \(\nu(\mu_2)=\varphi(\nu(\mu_1)) \subset \partial Y_2\). (Throughout the proof, we use the same symbol to denote both a slope on the torus and a simple closed curve representing it.) Equivalently, \(Y'\) can be obtained by starting with the disjoint union of \(Y_1,Y_2\) and \(S^1 \times I \times I\) and identifying \(S^1 \times I \times 0\) with \(\nu(\mu_1)\) and \(S^1 \times I \times 1\) with \(\nu(\mu_2)\).  
In this model, \(\partial Y'\) is a union of four annuli:  \(\partial Y_1 -\nu(\mu_1)\), \(S^1 \times 0 \times I\), \(\partial Y_2 -\nu(\mu_2)\), and \(S^1 \times 1 \times I\). The meridian \(\mu\) of \(K_1 \# K_2\) is homotopic to both \(\mu_1\) and \(\mu_2\) (and to the core of each of the four annuli.) 

Let \(\lambda_1\) be a longitude for \(\mu_1\), so that \(\lambda_2=-\varphi(\lambda_1)\) is a longitude for \(\mu_2\). We may assume that \(\lambda_1 \cap \nu(\mu_1) = p \times I \subset S^1 \times I \simeq \nu(\mu_1)\), and similarly for \(\lambda _2\). Let \(\lambda_1'\) be the arc obtained by intersecting \(\lambda_1\) with \(\partial Y_1 - \nu(\mu_1)\), and similarly for \(\lambda_2'\). The union of the arcs \(\lambda_1', p \times 0 \times I, \lambda_2',\) and \(p \times  1\times I\) is a longitude \(\lambda\) for \(K_1\#K_2\). Attaching a 2-handle along \(\lambda\) is the same as attaching  \(I \times I \times I\) to \(Y'\), where the top and bottom edges \(I \times 1/2 \times 1\) and  \(I \times 1/2 \times 0\) are identified with \(\lambda_1'\) and \(\lambda_2'\), and the sides \(1 \times 1/2 \times I\) and \(0 \times 1/2 \times I\) are identified with the other arcs in \(\lambda\).  
The resulting manifold can be obtained by starting with \(Y_1,Y_2\) and \(\Sigma \times I\), where \(\Sigma\) is a regular neighborhood of the \(1\)-skeleton in \(T^2\) and identifying \(\Sigma \times 0\) with a  tubular neighborhood of \( \mu_1 \cup \lambda_1 \subset \partial Y_1\) and \(\Sigma \times 1\) with its image under \(\varphi\). Finally, filling in the spherical boundary component with \(B^3\) gives \(Y_1 \cup (T^2 \times I) \cup Y_2 = Y_\varphi\). 
\end{proof} 

From the proof, we see that \(H_1(Y') \simeq H_1(Y_1)\oplus H_1(Y_2)/R\), where \(R\) is the subgroup generated by \((\mu_1, \mu_2)\), and that under this isomorphism,  \(\lambda = (\lambda_1, \varphi_*(\lambda_1)) = (\lambda_1,-\lambda_2)\). 
\vskip0.05in
We make two remarks on the utility of this construction. First, it is quite flexible, in the sense that the choice of {\em any} meridian \(\mu_1\in Sl(\partial Y_1)\) gives a different way of realizing the spliced manifold as a surgery. This flexibility will be useful to us in what follows. 

Second, rational surgery on a knot \(K \subset \Ybar\)  amounts to splicing \(Y\) with \(S^1 \times D^2\). Suppose \(\langle \mu, \lambda\rangle\) is our usual basis for \(H_1(\partial Y)\), and that \(\langle m, l\rangle\) is the standard basis for \(H_1(\partial S^1 \times D^2)\) (so \(l = [\partial D^2]\)). If we glue \(\partial Y\) to \(\partial(S^1\times D^2)\) in such a way that \([\partial D^2]\) is identified with \(\alpha = p \mu + q \lambda \in H_1(\partial Y)\), then it is easy to see that 
\(\mu\) is identified with \(-q m + p^* l\), where \(pp^* \equiv 1\,  \mod q\). Applying the lemma, we see that \(Y(\alpha) \) is obtained by integer surgery on a knot \(K' = K \# K_{-q/p} \subset \Ybar \# L(q,-p^*) = \Ybar \# L(q,-p)\). 

The knot \(K_{-q/p}\) is the unique knot in \(L(q,-p)\) whose complement is \(S^1 \times D^2\). (In the notation of \cite{LSpaceSurgeries}, it is the simple knot \(K(q,-p,1)\)). It is Floer simple, with Euler characteristic
$$ \chi(\hfk(K(q,-p,1)) \sim \frac{t^q-1}{t-1}.$$

To use Lemma~\ref{Lem:SpliceSurgery} to compute the Floer homology of a splice, we need to know how the knot Floer homology behaves under connected sum. 
\begin{lemma} \cite{OSRat}
\(\hfk(K_1 \# K_2) \simeq \hfk(K_1) \otimes \hfk(K_2) \).
\end{lemma}
The isomorphism is well-behaved with respect to \({\mathrm{Spin}^c}\) structures, in the sense that 
$$ \chi(\hfk(K_1\#K_2)) \sim \chi(\hfk(K_1)) \chi(\hfk(K_2)).$$ It is also respects the differentials, in the sense that \(\cfk(K_1\# K_2,d_0+d_v)\) is homotopy equivalent to 
\(\cfk(K_1,d_0+d_v)\otimes \cfk(K_2,d_0+d_v)\), and similarly for \(d_h\). 

In \cite{OSRat}, Ozsv{\'a}th and Szab{\'o} combined the observations above with their mapping cone for integer surgeries to express the Floer homology of any rational surgery as a mapping cone.

\section{Floer Simple Manifolds}
\label{Sec:FloerSimple}

In this section we use Ozsv{\'a}th and Szab{\'o}'s mapping cone formula to prove Proposition~\ref{Prop:MultiL} and to derive some basic facts about Floer simple manifolds. For the most part, these are straightforward extensions of results in \cite{OSLens},\cite{LSpaceSurgeries}, and \cite{BoyerCebanu}. We conclude by explaining how to compute the bordered Floer homology of a Floer simple manifold \(Y\) in terms of \(\tau(Y)\) and a Floer simple filling slope \(\alpha\). 
Our notation and  assumptions are the same as in section~\ref{Sec:HFK}.

\subsection{Proof of Proposition~\ref{Prop:MultiL}}
Suppose that \(K \subset \Ybar\) is a knot in an \(L\)--space, and that some nontrivial  surgery on \(Y\) is also an L-space. 

\begin{definition}
\label{Def:Chains}
We say that \(\hfk(K,\tspi)\) is a {\em positive chain} if it is  generated by elements \(x_1,\ldots, x_{n}, y_1, \ldots, y_{n-1}\) and the induced differentials \(\tilde{d}_h\), \(\tilde{d}_v\) satisfy \(\tilde{d}_v(y_i)  = \pm x_{i+1}\),
 \(\tilde{d}_h(y_i) = \pm x_{i+1}\), and \(\tilde{d}_v(x_i) =\tilde{d}_h(x_i) = 0 \) for all \(i\). More generally, we say that \(\hfk(K)\) {\em consists of positive chains} if  \(\widehat{CFK}(K, \tspi)\) is a positive chain for each \(\spi \in {\mathrm{Spin}^c}(Y)\), and that \(\hfk(K)\) consists of {\em coherent chains} if either \(\hfk(K)\) or \(\hfk(-K)\) consists of positive chains, where \(-K \subset -\Ybar\) is the mirror knot. 
\end{definition}
Note that all the \(x_i\)'s in the definition must have the same relative \(\Z/2\) grading, which is opposite that of the \(y_i\)'s. Since there are more \(x_i\)'s than \(y_i\)'s, the \(x_i\) contribute to \(\chi(\hfk(K))\) with positive sign, while the \(y_i\)'s contribute with negative sign. 

Ozsv{\'a}th and Szab{\'o} proved in \cite{OSLens} that if \(K \subset S^3\) has an L-space surgery with positive slope, then \(\hfk(K)\) is a positive chain. The following generalization is an easy consequence of a result of  Boileau, Boyer, Cebanu, and Walsh:

\begin{lemma}  Suppose that \(K \subset \Ybar\) is a knot in an \(L\)--space, and that some
 surgery on \(K\) is also an \(L\)--space. Then \(\hfk(K)\) consists of coherent chains.  
\end{lemma}

\begin{proof} 
A surgery on \(K\) is {\em positive} if the corresponding surgery cobordism is positive definite. 
Suppose that some positive integral surgery on \(K\) is an L-space. 
By Lemma 6.7 of \cite{BoyerCebanu}, the bent group \(A_{K,\spi} \simeq \Z\) for all \(\spi \in {\mathrm{Spin}^c}(Y, \partial Y)\). The proof of Theorem 1.2 of \cite{OSLens} carries over unchanged to show that 
\(\hfk(K, \tspi)\) is a positive chain.

Next, suppose that \(Y'\) is obtained  by negative integral surgery on \(K \subset \overline{Y}\), and that \(Y'\) is an L-space.  By reversing the orientation of the surgery cobordism, we see that \(-Y'\) is obtained by positive surgery on \(-K \subset - \overline{Y}\).  
\(-Y'\) is also an L-space, so \(\hfk(-K)\) consists of positive chains, and \(\hfk(K)\) consists of negative ones.

Finally, suppose that an L-space \(Y'\) is obtained by fractional surgery on \(K\). Then \(Y'\) is obtained by integral surgery on a knot of the form \(K \# K_{-q/p} \subset \Ybar \# L(q,-p)\), so \(\hfk(K \# K_{-q/p}) \simeq \hfk(K) \otimes \hfk(K_{-q/p})\) is composed of coherent chains. Since \(K_{-q/p}\) is Floer simple, it is easy to see that this occurs if and only if \(\hfk(K)\) is composed of coherent chains. 
\end{proof}

\begin{lemma}
If \(\hfk(K)\) consists of coherent chains, then 
$\displaystyle \tau(Y) = \sum_{h \in S[\tau(Y)]}[h].$
\end{lemma}
\begin{proof}
We have
\begin{equation*}
\tau(Y)   \sim \frac{\chi(\hfk(K))}{(1-[\mu])} 
 = \left( \sum_{\spibar \in M}  \chi(\hfk(K,\tspi)) \spibar \right) \left( \sum_{i=0}^\infty [\mu]^i\right)
\end{equation*}
where \(M \subset {\mathrm{Spin}^c}(Y,\partial Y)\) is a set of coset representatives for the action of \(\langle \mu \rangle\) and 
 $$\chi(\hfk(K,\tspi)) = \sum_{j \in \Z} \chi(\hfk(K,\spibar+j \mu))[\mu]^j.$$
The hypothesis that \(\hfk(K)\) consists of coherent chains implies that the nonzero coefficients of \(\chi(\hfk(K,\tspi))\) alternate between \(+1\) and \(-1\), and that the outermost coefficients are \(+1\). It follows that the  coefficients of the product
$\chi(\hfk(K,\tspi))  \left( \sum_{i=0}^\infty [\mu]^i\right)$
are all either \(0\) or \(+1\), and hence that all the coefficients of \(\tau(Y)\) are either \(0\) or \(1\) as well. 
\end{proof}

\begin{cor}
\label{Cor:LargeFS}
Suppose \(\hfk(K)\) is composed of coherent chains, and that \(\phi(\mu)>\|Y\|\). Then \(K\) is Floer simple.
\end{cor}

\begin{proof}
By hypothesis, \(\hfk(K)\) is composed of coherent chains, so to prove that \(K\) is Floer simple, it suffices to show that every monomial in \(\chi(\hfk(K))\) appears with a positive coefficient.  As usual, we normalize \(\tau(Y) = \sum_h a_h [h]\) so that \(a_h = 0 \) whenever \(\phi(h)<0\), and \(a_0\neq 0 \). We have 
$\chi(\hfk(K)) \sim (1-[\mu]) \tau(Y),$ so  the coefficient of \([h]\) in \(\chi(\hfk(K))\) is \(a_h - a_{h-\mu}\). Both terms in this difference are either \(0\) or \(1\). If \(\phi(\mu) > \phi(h)\), then \(a_{h-\mu}=0\), while if \(\phi(h)\geq \phi(\mu) > \|Y\|\), then \(a_h = 1\) by Proposition~\ref{Prop:AlexG}. In either case, we see that the coefficient of \([h]\) in \(\chi(\hfk(K))\) is either \(0\) or \(1\). 
\end{proof}

\begin{lemma}
\label{Lem:Interval}
If \(\hfk(Y)\) is composed of positive chains, there is an  interval in \(Sl(Y)\) whose left endpoint is \(\mu\) and which is contained in \(\lL(Y)\).
\end{lemma}

\begin{proof}
Since \(\hfk(K)\) is composed of positive chains, the homology of each of its bent complexes is \(\Z\). Since the homology of the bent complexes computes  \(\hfhat(Y(N \mu + \lambda))\) for some \(N \gg 0\), we see that \(N \mu + \lambda \in \lL(Y)\). Since \(\mu \cdot (N \mu + \lambda) = 1\), Proposition 17 of \cite{BGW} shows that the entire interval \([\mu, N \mu + \lambda]\) is contained in  \(\lL(Y)\).  
\end{proof}

By considering mirrors, we see that if \(\hfk(K)\) is composed of negative chains, then \(\mu\) is the right endpoint of a closed interval in \(\lL(Y)\). It follows that if \(K\) is Floer simple, then it is an interior point of an interval in \(\lL(Y)\). Conversely, if \(\hfk(K)\) is composed of negative chains but is not Floer simple, then some bent group of \(K\) has rank \(>1\). This implies that \(Y(N \mu + \lambda)\) is not an L-space for \(N\gg 0\). 
Thus if \( \hfk(K)\) is composed of coherent chains but is not Floer simple,  \(\mu\) is in not in the interior of \(\lL(Y)\). 

\begin{proof}
(Of Proposition~\ref{Prop:MultiL})
If \(Y\) is Floer simple, then it has some filling \(Y(\alpha)\) for which \(K_\alpha\) is Floer simple. As we observed above, \(\alpha\) is  contained in the interior of an interval in \(\lL(Y)\), so clearly \(|\lL(Y)|>1\). Conversely, if \(\lL(Y)>1\), then \(\hfk(K)\) is composed of coherent chains, so \(\lL(Y)\) contains an interval. Now any interval in \(Sl(Y)\) contains elements \(\alpha\) with \(\phi(\alpha)\) arbitrarily large. (To see this, identify \(Sl(Y)\) with \(\Qbar\) using the canonical meridian and longitude. If \(\alpha \mapsto a/b\) under this identification, then \(\phi(\alpha) = k a\) for some fixed \(k>0\).) By Corollary~\ref{Cor:LargeFS}, \(K_\alpha \subset Y(\alpha)\) is Floer simple, so \(Y\) is Floer simple.  
\end{proof}

\subsection{Surgery on Floer simple knots}
We now suppose that \(K \subset \Ybar\) is Floer simple. We give a graphical criterion for determining whether a given integer surgery on \(K\) is an L-space. To do so, we consider the set \(\black= S[\hfk(K)] \subset {\mathrm{Spin}^c}(Y, \partial Y)\). Since \(K\) is Floer simple, \(\black\) is a set of coset representatives for the action of the subgroup \(\langle \mu \rangle \subset H_1(Y)\). In other words, every \(\spi \in {\mathrm{Spin}^c}(Y, \partial Y)\) can be written in a unique way as \(\spibar + n \mu\), where \(\spibar \in \black\) and \(n \in \Z\). We color \(\spi\) {black} if \(n=0\), {red} if \(n>0\), and blue if \(n<0\).

 Now suppose we do surgery along \(K\) with slope \(\lambda\), where \(\mu \cdot \lambda = 1\). We divide \({\mathrm{Spin}^c}(Y,\partial Y)\) into cosets for the action of \( \langle \lambda \rangle \). Each coset \(L\) is an affine copy of \(\Z\), so it has a natural ordering. Each element of \(L\) is colored either black,  red, or blue; elements which are sufficiently negative are all colored blue, and elements which are sufficiently positive are all colored red.   We say \(L\) is {\em properly colored} if no red element of \(L\) appears before a blue element. 

\begin{prop}
\label{Prop:RedBlueBlack}
\(Y(\lambda)\) is an L-space if and only if every coset for the action of \(\langle \lambda \rangle\) is properly colored. 
\end{prop}

\begin{figure}
\begin{center}
\begin{tikzpicture}[scale=0.6]
\node at (0,0) (a1) {${ \bullet}$} ;
\node at (2,0) (a2) {$ \bullet$} ;
\node at (4,0) (a3) {${ \bullet}$} ;
\node at (6,0) (a4) {${ \bullet}$} ;
\node at (8,0) (a5) {${ \bullet}$} ;
\node at (10,0) (a6) {${ \bullet}$} ;
\node at (12,0) (a7) {${ \bullet}$} ;
\node at (14,0) (a8) {${ \bullet}$} ;
\node at (16,0) (a9) {${ \bullet}$} ;
\node at (18,0) (a10) {${ \bullet}$} ;
\node at (20,0) (a11) {${ \bullet}$} ;

\node[blue] at (0,2) (b1) {${ \star}$} ;
\node[blue]  at (2,2) (b2) {${ \star}$} ;
\node at (4,2) (b3) {${ \bullet}$} ;
\node[blue]  at (6,2) (b4) {${ \star}$} ;
\node[red] at (8,2) (b5) {${ \circ}$} ;
\node at (10,2) (b6) {${ \bullet}$} ;
\node[blue]  at (12,2) (b7) {${ \star}$} ;
\node[red] at (14,2) (b8) {${ \circ}$} ;
\node at (16,2) (b9) {${ \bullet}$} ;
\node[red] at (18,2) (b10) {${ \circ}$} ;
\node[red] at (20,2) (b11) {${ \circ}$} ;

\draw[->] (b1) [bend left=0] to node[above,sloped] {}  (a2)  ;
\draw[->] (b2) [bend left=0] to node[above,sloped] {}  (a3)  ;
\draw[->] (b3) [bend left=0] to node[above,sloped] {}  (a4)  ;
\draw[->] (b4) [bend left=0] to node[above,sloped] {}  (a5)  ;
\draw[->] (b6) [bend left=0] to node[above,sloped] {}  (a7)  ;
\draw[->] (b7) [bend left=0] to node[above,sloped] {}  (a8)  ;
\draw[->] (b9) [bend left=0] to node[above,sloped] {}  (a10)  ;
\draw[->] (b3) [bend left=0] to node[above,sloped] {}  (a3)  ;
\draw[->] (b5) [bend left=0] to node[above,sloped] {}  (a5)  ;
\draw[->] (b6) [bend left=0] to node[above,sloped] {}  (a6)  ;
\draw[->] (b8) [bend left=0] to node[above,sloped] {}  (a8)  ;
\draw[->] (b9) [bend left=0] to node[above,sloped] {}  (a9)  ;
\draw[->] (b10) [bend left=0] to node[above,sloped] {}  (a10)  ;
\draw[->] (b11) [bend left=0] to node[above,sloped] {}  (a11)  ;
\end{tikzpicture}
\end{center}
\caption{\label{Fig:Chains} Part of a  typical complex \(C_L\). Blue dots are shown by stars; red dots by hollow circles. Summands of each of the possible forms are visible.}
\end{figure}
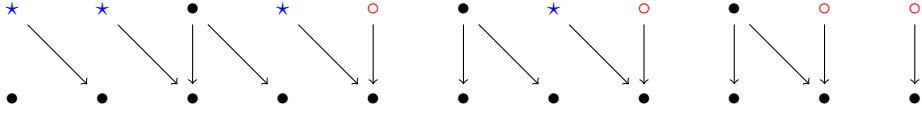

\begin{proof} The argument is the same as the proof of Lemma 4.8 in \cite{LSpaceSurgeries}; we sketch it briefly here. 
We fix a prime \(p\) and use the  mapping cone to compute 
\(\hfhat(Y(\lambda); \FF)\). 
The mapping cone \(C_\lambda(K)\) decomposes as a direct sum of chain complexes \(C_L\), one for each coset \(L\). Since \(K\) is Floer simple, the bent groups \(\Aa_{K,\spibar+n \lambda}\) appearing in one summand are all isomorphic to \(\FF\), as are the groups \(\Bb_{K,\spibar+ n \lambda}\). Let \(h_\spi,v_\spi\) be the restriction of the maps \(h,v\) to \(\Aa_{K,\spi}\). If \(\spi\) is colored red, the map \(v_\spi\) is an isomorphism and \(h_\spi=0\); if \(\spi\) is colored blue, the map \(h_\spi\) is an isomorphism and \(v_\spi=0\); and if \(\spi\) is colored black, both \(h_\spi\) and \(v_\spi\) are isomorphisms. 

The complex \(C_L\) takes the form shown in Figure~\ref{Fig:Chains}, where each colored dot in the top row represents \(\Aa_{K,\spibar+n \lambda} \simeq \FF\), each dot in the bottom row represents  \(\Bb_{K,\spibar+n \lambda} \simeq \FF\), and the arrows represent nonzero differentials. The chain of differentials breaks each time we encounter a red or blue dot, thus decomposing \(C_L\) into smaller summands. Summands corresponding to intervals in \(L\) whose endpoints are both red or both  blue are acyclic; summands whose left endpoint is blue and whose right endpoint is red have homology in even \(\Z/2\) homological degree, and 
summands whose left endpoint is red and whose right endpoint is blue have homology in odd \(\Z/2\) homological degree. 

It follows that \(\hfhat(Y(\lambda),\spibar) \simeq \FF\) if and only if \(L\) is properly colored, and hence that 
 \(Y(\lambda)\) is an \(\FF\) L-space if and only if every coset is properly colored. Finally, the statement of the proposition follows from the fact that  \(Y(\lambda)\) is an
  L-space if and only it is an \(\FF\) L-space for every prime \(p\). 
\end{proof}

\subsection{Bordered Floer homology of Floer simple manifolds}

In this section, we show that the bordered Floer homology \cite{LOT} of a Floer simple manifold \(Y\) is determined by the Turaev torsion of \(Y\) together with a slope in the interior of \(\lL(Y)\). We very briefly review some facts about bordered Floer homology; for more details see \cite{LOT, LOT2}. 

A bordered three-manifold is an oriented three-manifold \(Y\) equipped with a {\it parametrization} (that is, a minimal handle decomposition) of its boundary. We will restrict our attention to the case where \(\partial Y = T^2\), in which case a parametrization is specified  by a choice of two simple closed curves \(\mu,\lambda \in H_1(\partial Y)\) which satisfy \(\mu \cdot \lambda = 1\). 

The type \(D\) bordered Floer homology \(\cfd(Y,\mu,\lambda)\) is a differential graded module over a certain \(\Ff_2\)--algebra \(\aA(\zZ)\) associated to the torus. \(\aA(\zZ)\) is generated by elements \(\rho_1,\rho_2,\rho_3,\rho_{12},\rho_{23}\) and \(\rho_{123}\) corresponding to certain arcs on the boundary of the \(0\)-handle in the handle decomposition of \(\partial Y\), together with a pair of idempotents \(\iota_0, \iota_1\).  Following Chapter 11 of \cite{LOT}, we can think of the module structure as being specified by a pair of vector spaces \(V^0\), \(V^1\) over the field of two elements \(\Ff_2\), together with linear maps 
\begin{align*}
D_1,D_3,D_{123}&: V^0 \to V^1  \quad \quad &D_2: V^1 \to V^0 \\
D_{12}&:V^0 \to V^0  \quad \quad &D_{23}:V^1 \to V^1
\end{align*}
where \(\cfd(Y,\mu,\lambda) = \aA(\zZ) \otimes_{\Ff_2} (V^0 \oplus V^1)\) and for \(x \in V^0 \oplus V^1\), the differential is given by 
$\partial x= \sum \rho_I D_I(x).$

In writing the above, we have assumed that \(\cfd(Y,\mu,\lambda)\) has been reduced with respect to all provincial differentials, so that 
\begin{equation*}
V^0 \simeq SFH(Y, \gamma_\mu) \simeq \hfk(K_\mu) \quad V^1 \simeq \hfk(Y, \gamma_\lambda) \simeq HFK(K_\lambda)
\end{equation*}
where the suture \(\gamma_\mu\) is two parallel copies of \(\mu\), and similarly for \(\gamma_\lambda\). 

Petkova \cite{Ina} showed that the algebra \(\aA(\zZ)\) can be given an absolute \(\Z/2\) grading, and that \(\cfd(Y,\mu, \lambda)\) can be given a \(\Z/2\) grading compatible with it. Petkova's  grading depends on some auxiliary choices, but we can make some statements which are independent of these choices.

\begin{lemma}
\label{Lem:CFDMaslov}
The maps \(D_{12}\) and \( D_{23}\) preserve the homological \(\Z/2\) grading. If \(D_1\) has parity
\(i\) with respect to the \(\Z/2\) grading, then \(D_2\), \(D_3\) and \(D_{123}\) have parity \(1+i\), \(i\) and \(1+i\), respectively. \end{lemma}

\begin{proof} We first consider the absolute grading on \(\aA(\zZ)\). By definition, algebra generators corresponding to arcs joining two ends of the same \(\alpha\) arc have grading \(1\). (See definition 11 of \cite{Ina} and the equations just preceding it.) In our case, this says that \(\grm \rho_{12}\equiv \grm \rho_{23} \equiv 1\). From the relations \(\rho_{1} \cdot \rho_{23} = \rho_{123}\), \(\rho_1 \cdot \rho_2 = \rho_{12}\), and \(\rho_{2} \cdot \rho_3 = \rho_{23}\), we see that 
\(\grm \rho_{123}\equiv  \grm \rho_1 + 1\), \(\grm \rho_2 \equiv \grm \rho_1 + 1\), and \(\grm \rho_3 \equiv \grm \rho_2 + 1 \equiv \grm \rho_1\). 
The statement now follows from the fact that \(\grm \partial \bfx \equiv \grm \bfx + 1\). 
\end{proof}

We will also need to know how the \(D_I\)'s behave with respect to the \({\mathrm{Spin}^c}\) grading. Let us write
 \(V^0_\spi : = \hfk(K_\mu,\spi)\), so we have a decomposition \(V^0 \simeq \oplus_\spi V^0_\spi\), and similarly for \(V^1\), where the indexing sets in the sums are  \({\mathrm{Spin}^c}(Y,\gamma_\mu)\) and \({\mathrm{Spin}^c}(Y, \gamma_\lambda)\),  as defined in \cite{Juhasz1}. Elements of \({\mathrm{Spin}^c}(Y,\gamma_\mu)\) are represented by homology classes of nonvanishing vector fields on \(Y\) with fixed behavior on \(\partial Y\). (Recall that two nonvanishing vector fields are said to be homologous if they are homotopic on the complement of a ball in \(Y\).)  The sets 
  \({\mathrm{Spin}^c}(Y,\gamma_\mu)\) and \({\mathrm{Spin}^c}(Y,\gamma_\lambda)\) are in bijection, but not canonically so, since the boundary conditions are different. 

\begin{lemma} 
\label{Lem:CFDAlex}
There is a bijection 
\(j:{\mathrm{Spin}^c}(Y,\gamma_\mu) \to {\mathrm{Spin}^c}(Y, \gamma_\lambda)\) which respects the action of \(H_1(Y)\) and for which 
\begin{align*} D_1&:V^0_\spi \to V^1_{j(\spi)} & D_2&:V^1_{j(\spi)} \to V^0_{\spi-\lambda} & D_3&:V^0_\spi \to V^1_{j(\spi)+\lambda + \mu} \\
D_{12}&:V^0_\spi \to V^0_{\spi - \lambda} & D_{23}&:V^1_{j(\spi)} \to V^1_{j(\spi) +\mu} & D_{123}&:V^0_{\spi} \to V^1_{j(\spi) + \mu} 
\end{align*}
\end{lemma}
This is essentially Lemma 11.42 of \cite{LOT}, but stated so as to clarify the dependence on \(\mu\) and \(\lambda\). 

\begin{proof}
 Huang and Ramos \cite{HuangRamos} have constructed  a grading \(\gr\) on \(\cfd(Y,\mu,\lambda)\). This grading lives in  a set \(S(\mathcal{H})\) of homotopy classes of nonvanishing vector fields on \(Y\) which satisfy certain boundary conditions. To be specific, for each elementary idempotent \(\iota\) in the algeba \(\aA(\zZ)\), there is an associated vector field \(v_\iota\) on \(\partial Y\), and if \(v \in S(\mathcal{H})\), then \(v|_{\partial Y}\) should be equal to \(v_\iota\) for some elementary idempotent \(\iota\). 
 
 Similarly, Huang and Ramos consider the set \(G(\zZ)\) of homotopy classes of nonvanishing vector fields on \(\partial Y \times [0,1]\), subject to the constraint that \(v|_{\partial Y \times 0} = v_\iota\) and 
 \(v|_{\partial Y \times 1} = v_{\iota'}\) for some elementary idempotents \(\iota\) and \(\iota'\). They show that 
 \(G(\zZ)\) forms a groupoid under concatenation, and that it acts on the grading set \(S(\mathcal{H})\), again by concatenation. In section 2.3 of \cite{HuangRamos}, they construct explicit vector fields \(v_I\) on \(\partial Y \times [0,1]\) associated to each \(\rho_I\); the grading of \(\rho_I x\) is the vector field \(v_I \cdot \grm x\), where \(\cdot\) denotes the action by concatenation.

  The grading of \cite{HuangRamos} contains the  \({\mathrm{Spin}^c}\) grading, in the sense that if \(\mathbf{x}\) is a generator of \(\cfd(Y, \mu, \lambda)\), then its \({\mathrm{Spin}^c}\) grading is \(\spi(\mathbf{x}) = p(\grm \bfx)\), where \(p\) is the forgetful map which takes a homotopy class of vector fields to its homology class.  
By Theorem 1.3 of \cite{HuangRamos}, if \(\bfx \in \cfd(Y,\mu,\lambda)\),  \(\grm \partial \bfx = \lambda ^{-1} \cdot \grm \bfx\), where  \(\lambda\) is a vector field on \(\partial Y \times [0,1]\) which is supported in a ball. It follows that
 \(\spi(\partial \bfx) = \spi(\bfx)\), and hence that \(p(v_I)\cdot \spi(D_I \bfx)\) = \(\spi(\bfx)\). 
 
 If \(\spi \in {\mathrm{Spin}^c}(Y, \gamma_\mu)\), we define \(j (\spi) = p(v_1^{-1}) \cdot \spi\). By construction, 
 \(D_1:V^0_\spi \to V^1_{j(\spi)}\). The fact that \(G(\zZ)\) is a groupoid implies that \(j\) is a bijection; \(j\) is equivariant with respect to the action of \(H_1(Y)\) since we can arrange this action to take place in the interior of \(Y\), away from the region in which the concatenation takes place. Similarly, we see that 
 $$\spi(D_3 \bfx) = p(v_3^{-1}) \cdot \spi(\bfx) = p(v_3^{-1} \cdot v_1) \cdot j(\spi(\bfx)).$$
 
 The set of homology classes of nonvanishing vector fields on \(\partial Y \times [0,1]\) which restrict to \(v_{\iota_0}\) on one end  and \(v_{\iota_1}\) on the other is an affine copy of \(H_1(\partial Y \times [0,1]) \simeq H_1(\partial Y)\). Thus if \(I_1\) is the idempotent of the groupoid \(G(\zZ)\) corresponding to the idempotent \(\iota_1\), we must have
 \(p(v_3^{-1}\cdot v_1) = p(I_1) + \alpha\), for some \(\alpha \in H_1(\partial Y)\). It follows that 
 \(\spi(D_3(\bfx)) = j(\spi(\bfx))+\alpha\) for some universal element \(\alpha \in H_1(\partial Y)\) which does not depend on \(Y\) or \(\bfx\). Comparing with Lemma 11.42 of \cite{LOT}, we see that \(\alpha = \mu + \lambda\). Thus \(D_3:V^0_\spi \to V^1_{j(\spi)+  \lambda+\mu}\) as desired. The arguments for the other \(D_I\)'s are very similar. 
\end{proof}

\begin{figure}
\begin{center}
  \begin{tikzpicture}[y=54pt,x=1in]
    \node at (0,2) (p) {${\mathbf p}$} ;
    \node at (2,2) (q) {${\mathbf q}$} ;
    \node at (1,0) (r) {${\mathbf r}$} ;
    \node at (1,1.25) (label) {$\tau_m$};
    \draw[->] (p) to node[above,sloped] {\lab{\rho_1\otimes\rho_1+\rho_{123}\otimes\rho_{123}+\rho_3\otimes (\rho_{3},\rho_{23})}}  (q)  ;
    \draw[->] (p) [bend left=15] to node[above,sloped] {\lab{\rho_{123}\otimes\rho_{12}+
\rho_3\otimes (\rho_{3},\rho_{2})}} (r) ;
    \draw[->] (q) [bend left=15] to node[below,sloped] {\lab{\rho_{23}\otimes \rho_{2}}} (r) ;
    \draw[->] (q) [loop] to node[above] {\lab{\rho_{23}\otimes\rho_{23}}}
                 (q) ;
    \draw[->] (r) [bend left=15] to node[below,sloped] {\lab{\rho_2\otimes 1}} (p) ;
    \draw[->] (r) [bend left=15] to node[below,sloped] {\lab{1\otimes \rho_3}} (q) ;
  \end{tikzpicture}
  \begin{tikzpicture}[y=54pt,x=1in]
    \node at (0,2) (p) {${\mathbf q}$} ;
    \node at (2,2) (q) {${\mathbf p}$} ;
    \node at (1,0) (r) {${\mathbf s}$} ;
    \node at (1,1.25) (label) {$\tau_l^{-1}$};
    \draw[->] (p) [bend left=15] to node[above,sloped] {\lab{\rho_2\otimes 1}}  (r)  ;
    \draw[->] (q) [loop] to node[above] {\lab{\rho_{12}\otimes\rho_{12}}}
                (q) ;
    \draw[->] (q) [bend left=15] to node[below,sloped]  {\lab{1\otimes \rho_1}} (r) ;
    \draw[->] (q)  to node[above,sloped] {\lab{\rho_3\otimes\rho_3+\rho_{123}\otimes\rho_{123}+\rho_1\otimes(\rho_{12},\rho_1)}} (p) ;
    \draw[->] (r) [bend left=15] to node[below,sloped] {\lab{\rho_1\otimes (\rho_2,\rho_1)+\rho_{123}\otimes\rho_{23}}} (p) ;
    \draw[->] (r) [bend left=15] to node[above,sloped] {\lab{\rho_{12}\otimes\rho_2}} (q) ;
  \end{tikzpicture}
\end{center}
\caption{\label{Fig:Bimodules} Change of framing bimodules for the torus, taken from figure A.3 of \cite{LOT}.}
\end{figure}
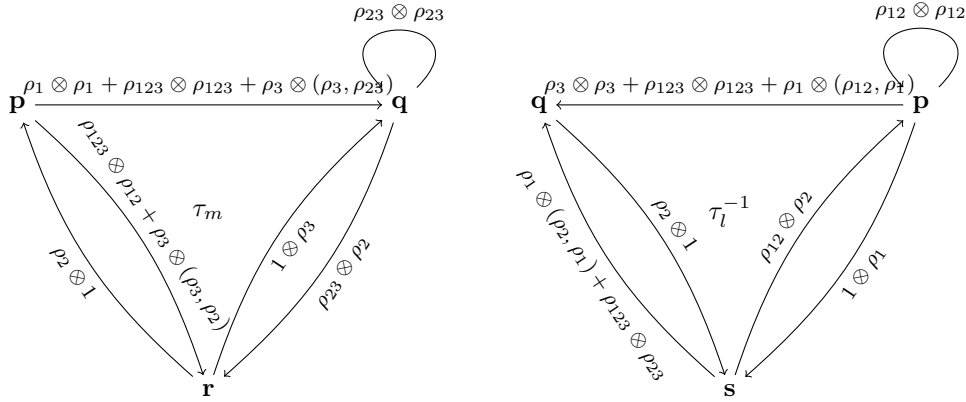

\begin{prop}
\label{Prop:FloerSimpleCFD}
Suppose that \(Y\) is Floer simple, that \(\alpha \in Sl(Y)\) is a Floer simple filling slope, and that \(\mu, \lambda \in H_1(\partial Y)\) satisify \(\mu \cdot \lambda = 1\). Then \(\cfd(Y, \mu, \lambda)\) is determined by 
\(\alpha\) and \(\tau(Y)\). 
\end{prop}

\begin{proof}
It suffices to show that \(\cfd(Y,\mu,\lambda)\) is determined for one particular choice of \(\mu\) and \(\lambda\), since the invariant of any other choice can then be determined using the change of basis bimodules in \cite{LOT2}. 

 We  choose \(\mu\) to be a slope in the interior of \(\lL(Y)\) such that \(\phi(\mu)>\|Y\|\),  and take \(\lambda = \lambda_0 - N \mu\), where \(\lambda_0\) is some class with \(\mu \cdot \lambda_0 = 1\), and \(N \gg 0\). (We will specify below  how large \(N\) needs to be.)

The knots \(K_\mu\), \(K_\lambda\) are Floer simple, so all the elements of \(V_0\) have the same \(\Z/2\) grading. Similarly, all elements of \(V_1\) have the same \(\Z/2\) grading.   By Lemma~\ref{Lem:CFDMaslov}, either \(D_2=D_{123}=0\) or \(D_1=D_3=0\). To see which of these two options hold, we consider the effect of a Dehn twist along \(\mu\). We have
$$ \cfd(Y, \mu, \lambda+\mu) = \widehat{CFDA}(\tau_\mu) \boxtimes \cfd(Y,\mu,\lambda)$$
where the change of framing bimodule \( \widehat{CFDA}(\tau_\mu)\) is shown in Figure~\ref{Fig:Bimodules}.

Writing \(\cfd(Y,\mu,\lambda+\mu) = W^0\oplus W^1\), we have \(W^1 = \mathbf{r}\boxtimes V^0 \oplus \mathbf{q} \boxtimes V^1\). Denote by \(D:W^1 \to W^1\) the contribution to \(\partial\) coming from provincial differentials; then  we have \(H(W^1,D)=\hfk(K_{\mu+\lambda})\). By choosing \(N\) sufficiently large, we can ensure that  \(\mu+ \lambda = \lambda_0 - (N-1) \mu\) is in the interior of \(\lL(Y)\). It follows that  \(\hfk(K_{\mu+\lambda})\) is Floer simple and has dimension equal to 
$|H_1(Y(\mu+\lambda))|=|H_1(Y_\lambda)| - |H_1(Y_\mu)| = \dim V_1 - \dim V_0.$ Referring to the figure, we see that the only contribution to  the provincial differential \(D\) comes from the arrow labeled \(1 \otimes \rho_3\). Thus the map \(\rho_3:V^0 \to V^1\) is an injection. Similarly, by considering 
$$\cfd(Y,\mu+\lambda, \lambda) = \widehat{CFDA}(\tau_\lambda^{-1}) \boxtimes \cfd(Y, \mu, \lambda)$$ we deduce that the map \(D_1:V_0 \to V_1\) is injective. Since \(D_1\) and \(D_3\) are nontrivial, we must have 
\(D_2=D_{123}=0\). 

Let \(\spi_{max} \in S[\hfk(K_\mu)]\) be {maximal}, in the sense that if \(\spi_{\max} + \alpha \in S[\hfk(K_\mu)]\) (where \(\alpha \in H_1(Y)\)), then \(\phi(\alpha) \leq 0\). 

\begin{lemma} \(j(\spi_{\max})\) is maximal in \(S[\hfk(K_\lambda)]\). 
\end{lemma}

\begin{proof}
It is well known \cite{YiNiThurstonNorm} that \(\hfk\) detects the Thurston norm, in the sense that if 
\(K \subset Y(\alpha)\), then 
$$ \max \{\phi(\spi-\spi') \, | \, \spi, \spi' \in S[\hfk(K)]\}  = \|Y\| + |\phi(\alpha))|.$$
Choose nonzero elements \(\bfx \in V^0_{\spi_{max}}\), \(\bfy \in V^0_{\spi_{min}}\), where 
\(\phi(\spi_{max} - \spi_{min}) = \|Y\| + \phi(\mu)\). Since \(D_1\) and \(D_3\) are injective, 
\(j(\spi_{max})\) and \(j(\spi_{min}) + \lambda + \mu\) are both in \(S[\hfk(K_\lambda)]\). 
We compute 
\begin{align*}
\phi(j(\spi_{max}) - (j(\spi_{min})+\mu+\lambda)) & = \|Y\| + |\phi(\lambda)| \\
& =\max \{\phi(\spi-\spi') \, | \, \spi, \spi' \in S[\hfk(K_\lambda)]\} .
\end{align*}
It follows that \(j(\spi_{max})\) must be maximal and \(j(\spi_{min}+\mu+\lambda)\) must be minimal. 
\end{proof}

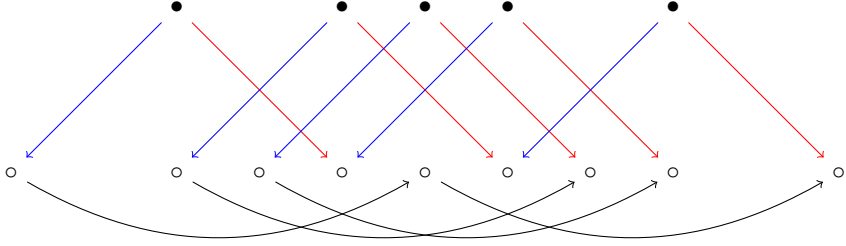
\begin{figure}
\begin{center}
\begin{tikzpicture}[scale=2.2]
\node at (0,0) (a1) {${\circ}$} ;
\node at (1,0) (a2) {${\circ}$} ;
\node at (1.5,0) (a3) {${\circ}$} ;
\node at (2,0) (a4) {${\circ}$} ;
\node at (2.5,0) (a5) {${\circ}$} ;
\node at (3,0) (a6) {${\circ}$} ;
\node at (3.5,0) (a7) {${\circ}$} ;
\node at (4,0) (a8) {${\circ}$} ;
\node at (5,0) (a9) {${\circ}$} ;

\node at (1,1) (b1) {${\bullet}$} ;

\node at (2,1) (b2) {${\bullet}$} ;
\node at (2.5,1) (b3) {${\bullet}$} ;
\node at (3,1) (b4) {${\bullet}$} ;

\node at (4,1) (b5) {${\bullet}$} ;
\draw[->,red] (b5) [bend left=0] to node[above,sloped] {}  (a9)  ;
\draw[->,red] (b4) [bend left=0] to node[above,sloped] {}  (a8)  ;
\draw[->,red] (b3) [bend left=0] to node[above,sloped] {}  (a7)  ;
\draw[->,red] (b2) [bend left=0] to node[above,sloped] {}  (a6)  ;
\draw[->,red] (b1) [bend left=0] to node[above,sloped] {}  (a4)  ;

\draw[->,blue] (b5) [bend left=0] to node[above,sloped] {}  (a6)  ;
\draw[->,blue] (b4) [bend left=0] to node[above,sloped] {}  (a4)  ;
\draw[->,blue] (b3) [bend left=0] to node[above,sloped] {}  (a3)  ;
\draw[->,blue] (b2) [bend left=0] to node[above,sloped] {}  (a2)  ;
\draw[->,blue] (b1) [bend left=0] to node[above,sloped] {}  (a1)  ;

\draw[->] (a1) [bend right=30] to node[above,sloped] {}  (a5)  ;
\draw[->] (a2) [bend right=30] to node[above,sloped] {}  (a7)  ;
\draw[->] (a3) [bend right=30] to node[above,sloped] {}  (a8)  ;
\draw[->] (a5) [bend right=30] to node[above,sloped] {}  (a9)  ;
\end{tikzpicture}
\end{center}
\caption{\label{Fig:CFD} Generators of \(\cfd(Y,5m-l,-9m+ 2l)\), where \(Y\) is the complement of the negative trefoil in \(S^3\). Dots in the top row represent generators of \(V_0\), dots in the bottom row generators of \(V_1\). The  horizontal position of each generator indicates  its \({\mathrm{Spin}^c}\) grading.  Potential components of the differential are shown by arrows: red (sloping right) for \(D_1\), blue (sloping left) for \(D_3\), and  black (the arcs) for \(D_{23}\).}
\end{figure}
We represent \(\cfd(Y, \mu, \lambda)\) by a directed graph like that shown in Figure~\ref{Fig:CFD}, with a vertex for each generator and an edge  for each  potential component of the differential; that is, for each pair of generators \(\bfx,\bfy\) whose \(\Z/2\) and \({\mathrm{Spin}^c}\) gradings are compatible with having \(D_I \bfx = \bfy\) for some \(D_I\), we draw an edge from \(\bfx\) to \(\bfy\) and label it with \(D_I\). 

\begin{lemma} 
\label{Lem:2Arrows}
Each vertex  of the graph associated to \(\cfd(Y,\mu, \lambda)\) has valence two.
\end{lemma}

\begin{proof}
First suppose that \(\bfx\) is a generator of \(V^0\). We have already seen that \(D_1\) and \(D_3\) are both injective, so \(\bfx\) is the starting point of one arrow labeled with \(D_1\) and one arrow labeled with \(D_3\). \(D_2=D_{123}=0\), so the only other possible arrows adjacent to \(\bfx\) are labeled by \(D_{12}\). Now \(D_{12}\) shifts the \({\mathrm{Spin}^c}\) grading by \(-\lambda\), and \(\phi(-\lambda) = N \phi(\mu) - \phi(\lambda_0)\). Let \(S =S[\hfk(K_\mu)]\). We choose \(N\) sufficiently large that 
$|\phi(\lambda)| > \max _{\spi \in S} \phi(\spi) - \min _{\spi \in S} \phi(\spi)$; then \(D_{12}\) vanishes for grading reasons. 

Next, if \(\bfx\) is a generator of \(V_1\), it can be a terminal point of an arrow labeled \(D_1\) or \(D_3\), and either an initial or a terminal point of a arrow labeled \(D_{23}\). We claim that \(\bfx\) is a terminal point of an arrow of type \(D_1\) if and only if it is not an initial point of an arrow of type \(D_{23}\).  To see this, consider \(\spi \in {\mathrm{Spin}^c}(Y, \gamma_\mu)\). We say \(\spi\) is {\em occupied} if \(\spi \in S[\hfk(K_\mu)]\), and {\em unoccupied} otherwise; similarly for \(j(\spi) \in {\mathrm{Spin}^c}(Y,\gamma_\lambda)\), but with \(K_\lambda\) in place of \(K_\mu\). The claim is equivalent to saying that if \(j(\spi)\) is occupied, then exactly one of \(\spi\) and \(j(\spi) + \mu\) is occupied. 

Write \(j(\spi) = j(\spi_{max}) - \alpha\) for \(\alpha \in H_1(Y)\). We consider the situation case by case, depending on the value of \(\phi(\alpha)\). 
\begin{enumerate}
\item \(\phi(\alpha) < 0 \). In this case \(j(\spi)\) is unoccupied, and there is nothing to check.
\item \(0 \leq \phi(\alpha) < \phi(\mu)\). In this region, \(\chi(\hfk(K_\mu))\) and \(\chi(\hfk(K_\lambda))\) are both given by  \(\tau(Y)\), so \(\spi\) is occupied if and only if \(j(\spi)\) is occupied. \(\phi(-\alpha + \mu)>0\), so \(j(\spi)+\mu\) is unoccupied. 
\item \(\phi(\mu) \leq \phi(\alpha) \leq \phi(\mu)+\|Y \| \). In this region \(j(\spi)\) is always occupied (see the argument for region 4) below), while \(\spi\) is occupied if and only if \(\spi+ \mu\) is not occupied. \(j(\spi) + \mu\) is in region 2), so \(\spi\) is occupied if and only if \(j(\spi)+\mu\) is not occupied. 
\item \(\phi(\mu)+\|Y\| < \phi(\alpha)<|\phi(\lambda)|\). In this region \(\spi\) is unoccupied, while \(\chi(\hfk(K_\lambda))\) is given by \(\tau(Y)\). By Proposition~\ref{Prop:AlexG},  both \(j(\spi)\) and \(j(\spi)+\mu\) are always occupied. 
\item \(| \phi(\lambda)|<\phi(\alpha)<|\phi(\lambda)|+\|Y\|\). In this region, \(\spi\) is unoccupied. Since \(\phi(\mu)>|Y|\),  \(j(\spi)+\mu\) is in region 4) and is always occupied. 
\item \(|\phi(\lambda)|+\|Y\| \leq \alpha\). In this region, \(j(\spi)\) is unoccupied.
\end{enumerate}

This proves the claim. A very similar argument shows that \(\bfx\) is a terminal point of an arrow of type \(D_3\) if and only if it is not the terminal point of an arrow of type \(D_{23}\). The statement of the lemma follows. 
\end{proof}

Since \(K_\mu\) and \(K_\lambda\) are Floer simple,  each arrow in the diagram corresponds to a map 
\(\Ff_2 \to \Ff_2\). To determine the corresponding component of the differential, it suffices to know whether or not this map is \(0\). We will show that every map corresponding to an arrow in the diagram is nonzero, thus completing the proof of Proposition~\ref{Prop:FloerSimpleCFD}. The maps \(D_1\) and \(D_3\) are injective, so any arrow labeled by \(D_1\) or \(D_3\) is nonzero. For the arrows labeled by \(D_{23}\), we argue as in the proof of Theorem 11.36 in \cite{LOT}. 
By Proposition 11.30 of \cite{LOT}, there are maps \(D_{012}, D_{01},D_0,D_{230},\) and \(D_{301}\) satisfying 
\begin{align*}
D_3 \circ D_{012}+D_{23}\circ D_{01} + D_{123}\circ D_0&  = 1_{V_1} \\
D_1\circ D_{230}+D_{01}\circ D_{23}+D_{301}\circ D_2 &= 1_{V_1}
\end{align*}
Since \(D_2=D_{123}=0\), it follows that if \(\bfx\) is not in the image of \(D_{3}\), it must be in the image of \(D_{23}\), and if \(\bfx\) is not in the image of \(D_1\), \(D_{23}(\bfx) \neq 0 \). Comparing with the proof of Lemma~\ref{Lem:2Arrows}, we see that every arrow in the diagram must correspond to a nonzero map. 
\end{proof}

\section{Intervals of L-space filling slopes}
\label{Sec:LY}

Now that the ``proper coloring'' condition of
Proposition~\ref{Prop:RedBlueBlack} is in place,
we are equipped to tackle the problem
of describing L-space intervals in terms of $\dt(Y)$
and a slope from the interior of the L-space interval.
We begin by establishing some conventions.

\smallskip
\subsection{Conventions for slopes and homology}
\label{ss: slope and basis conventions}
If $Y$ is a compact oriented three-manifold with torus boundary, then 
a {\it{slope}} of $Y$ is a nonseparating, oriented, simple closed curve in $\partial Y$.
Such objects correspond bijectively to primitive elements of $H_1(\partial Y)/\{\pm 1\}$,
or equivalently, to elements of  $\P(H_1(\partial Y))$.
Any choice of basis $(m,l)$ for $H_1(\partial Y)$ specifies
homogeneous coordinates $n m + n' l \mapsto  [n\!: n']$ on
$\P(H_1 (\partial Y))$,
to which we usually refer in terms of the affinization
\begin{align}
\label{eq: def of pi for thm 1}
H_1(\partial Y) \setminus \{0\} 
&\to \Q \cup \{\infty\},
\\
n m + n' l 
&\mapsto  n/n'.
\nonumber
\end{align}

Let  $\iota : H_1(\partial Y)  \to H_1(Y)$ be the map induced by inclusion. 
We fix a basis $(m,l)$ for $H_1(\partial  Y)$
such that \(l\) is a generator of \(\ker \iota\) and \(m \cdot l = 1\). 
The generator \(l\) is the {\em homological longitude} of \(Y\); it is well defined up to sign. 
In contrast, the choice of \(m\) is only well defined up to the addition of a multiple of \(l\). 
Consequently,  the numerator of $\pi(nm + n'l) = n/n'$ is
canonical (up to sign), but the denominator depends on the choice of $m$.

To Dehn fill $Y$ along a slope
$\mu = nm + n'l \in H_1(\partial Y)$, one attaches a 2-handle along
the simple closed curve associated to $\mu$, and then fills in the remaining $S^2$ boundary with a 3-ball.
The resulting manifold, which we denote by $Y(\mu)$ or $Y(n/n')$,
has homology $H_1(Y(\mu)) = H_1(Y)/(\iota(\mu))$,
which has order $|n|$ if $H_1(Y)$ is torsion free.

Any non-zero Dehn filling $Y(\mu_{\textsc{l}})$ produces a knot
$K_{\mu_{\textsc{l}}} := \mathrm{core}(Y(\mu_{\textsc{l}}) \setminus Y)
\subset Y(\mu_{\textsc{l}})$, on which one can now perform Dehn surgery.
 Whereas our conventionial 
choice of basis for Dehn {\it{filling}} slopes involves a canonical (up to sign)
{\it{longitude}} $l$, with $m$ (satisfying $m \cdot l = 1$)
only determined up to addition of copies of $l$,
the conventional basis for Dehn {\it{surgery}} involves a canonical
{\it{meridian}}, namely $\mu_{\textsc{l}}$, for the knot
$K_{\mu_{\textsc{l}}} \subset Y(\mu_{\textsc{l}})$,
with the longitude $\lambda_{\textsc{l}} \in H_1(\partial Y)$
(satisfying $\mu_{\textsc{l}} \cdot \lambda_{\textsc{l}} = 1$)
only determined up to the addition of copies of $\mu_{\textsc{l}}$.

Thus, for an arbitrary slope, say
\begin{equation}
\mu = n m + n' l = \alpha \mu_{\textsc{l}} + \beta \lambda_{\textsc{l}}
\in H_1(\partial Y),
\end{equation}
we could describe the Dehn filling $Y(\mu)$
as the $n/n'$-filling of $Y$ (with respect to the basis $(m, l)$),
or as the $\alpha/\beta$-surgery along the knot $K_{\mu_{\textsc{l}}}$
(with respect to the basis 
$(\mu_{\textsc{l}},  \lambda_{\textsc{l}})$).  Note that
each of these conventional descriptions involves
either a denominator or a numerator which is non-canonical.
To dodge this problem, we can instead
divide the canonical numerator of $n/n'$
by the canonical denominator of $\alpha/\beta$
to obtain $n/\beta$, with
\begin{equation}
\label{eq: beta n def}
\;\;\;\;n:= \mu \cdot l,
\;\;\;\;\;\;\;\;
\beta
:=
\mu_{\textsc{l}} \cdot \mu = pn'\!-\!qn\; (\mathrm{where}\; \mu_{\textsc{l}} = pm + ql),
\end{equation}
and with $|n| = |H_1(Y(\mu))|$ when $H_1(Y)$ is torsion free.
Note that $n/\beta$ is not a slope in the conventional sense,
since $\mu = n (\mu_{\textsc{l}}/p) + \beta (l/p)$,
with $\mu_{\textsc{l}}/p, l/p \notin H_1(\partial Y; \Z)$,
and the projective linear map
$\P(H_1(\partial Y)) \to \P(\Z^2)$, $[n: n'] \mapsto [n: \beta]$
is not surjective, having determinant $p$.
Still, since this map is injective, it is sufficient for
cataloguing slopes.  In fact, the reciprocal $\beta/n$
is more convenient for this purpose.  Given an initial filling
$Y(\mu_{\textsc{l}})$ on which we wish to perform surgery,
we call $(\mu_{\textsc{l}}\cdot \mu)/(\mu \cdot l) = \beta/n$
the {\it{surgery}} $\mu_{\textsc{l}}$-{\it{label}} (or just 
{\it{surgery label}})
of $\mu$.  Since
\begin{equation}
\label{eq: pq beta n identity}
\frac{n}{n'} = \frac{p}{q + \beta/n},
\end{equation}
the surgery $\mu_{\textsc{l}}$-label of $\mu$ quantifies the deviation of
the Dehn filling slope of $\mu$ from that of $\mu_{\textsc{l}}$,
with a surgery label of $\beta/n=0$ labeling the original slope $\mu_{\textsc{l}}$.

We also need conventions for $H_1(Y)$, relative to the map
$\iota: H_1(\partial Y) \to H_1(Y)$,
restricting to the case of $b_1(Y)=1$.
The Universal Coefficients Theorem implies
$\coker \iota \cong H^1(Y) \cong \Tors(H_1(Y))$.
Thus, setting 
$T: = \Tors(H_1(Y))$ and $T^{\partial} := \left< \iota(l) \right> = T \cap \iota(H_1(\partial Y))$,
we have
$\coker \iota  
\,=\,  H_1(Y) /\left( \left<\iota(m)\right> \oplus T^{\partial} \right) \cong T$, which implies
\begin{equation}
\left(H_1(Y)/T\right) / \iota(m) \cong T^{\partial} \cong \Z/g,
\end{equation}
where $g := |T^{\partial}|$.
In other words, any generator  $\bar{m}$ for $H_1(Y)/T$
will satisfy $\iota(m) \in \pm g\bar{m} + T$.  We shall always choose
$\bar{m}$ so that $\iota(m) \in +g\bar{m} + T$.

\subsection{Conventions for Turaev torsion and $\boldsymbol{\dt(Y)}$}
\label{ss: torsion conventions}

Recall our definition for $\dt(Y) \subset H_1(Y)$
as the finite set
\begin{equation}
\label{eq: first def of dt}
\dt(Y)
:= \{x-y|\, x \notin S[\tau(Y)], y \in S[\tau(Y)]\} 
 \cap \iota(m \Z_{\ge 0} + l \Z),
\end{equation}
where $\tau(Y)$ is the Turaev torsion of $Y$, 
which we always normalize so that 
\begin{equation}
\label{eq: torsion support 0 convention}
0 \in S[\tau(Y)],\;\;\;\;
\tau(Y) \in \Z[[t]][T],
\end{equation}
with $t:= [\bar{m}]$ for any generator $\bar{m}$
of $H_1(Y)/T \cong \Z$ satisfying
$\iota(m) \in \bar{m} \Z_{>0} + T$.

When $Y$ is Floer simple, we can also define the
{\it{torsion complement}},
\begin{equation}
\label{eq: torsion complement definition}
\tc(Y) 
\;:=\; \frac{1}{1-t} \sum_{h\in T}[h]
\;\;-\;\; \tau(Y),
\end{equation}
with the Floer simplicity of $Y$ guaranteeing that
\begin{equation}
S[\tc(Y)]  \;=\; \bar{m}\Z_{\ge 0} \oplus T \; \setminus\; S[\tau(Y)],
\end{equation}
so that $\dt(Y)$ admits the alternative definition
\begin{equation}
\dt(Y)
:= (S[\tc(Y)] - S[\tau(Y)])
 \cap \iota(m \Z_{\ge 0} + l \Z).
\end{equation}
We shall often want to restrict our attention to the non-torsion elements of $\dt(Y)$,
\begin{equation}
\dtgz(Y)
:= (S[\tc(Y)] - S[\tau(Y)])
 \cap \iota(m \Z_{> 0} + l \Z)\;\;=\;\; \dt(Y) \setminus T,
\end{equation}
When we wish to emphasize our inclusion of the torsion elements of $\dt(Y)$,
we shall write $\dtge(Y)$ for $\dt(Y)$.

Although we shall not need the following fact until the
proof of 
Theorem \ref{thm: torus gluing and complementary intervals for L-spaces}
in Section~\ref{s: L-space if and only if intervals cover.},
we lastly remark that
the complement of $\dt(Y)$ is a semigroup. 
\begin{prop}
\label{prop: The complement of D is additively closed.}
If $Y$ is Floer-simple, then the complement
$\Gamma(Y):=  \iota(m \Z_{\ge 0} + l \Z) \setminus \dt(Y)$
is closed under addition.
\end{prop}
\begin{proof}
Suppose there exist $x, y \in \Gamma(Y)$ with
$x+ y \in \dt(Y)$.  
Since $x+ y \in \dt(Y)$, we know there exists 
$z \in S[\tau(Y)]$ for which $x+y + z \in S[\tc(Y)]$.
If $z+x \in S[\tc(Y)]$ then
$x = (x+ z) - z \in \dt(Y)$, a contradiction.
On the other hand, if
$s+x \in S[\tau(Y)]$,
then $y = (x+y+ z) - (x+z) \in \dt(Y)$, another contradiction.
Thus $x+y \notin \dt(Y)$.
\end{proof}
{{\noindent{In the case that $Y$ Floer simple is the complement of the
link of a complex planar singularity, $\Gamma(Y)$ coincides
with the semigroup associated to the Newton-Puiseux expansion.}}}

\subsection{Notation: Truncation and remainders}

Lastly, we need some basic arithmetic notation.
Henceforth in this paper,
we use the conventional truncations $\lfloor \cdot \rfloor , \lceil \cdot \rceil: \Q \rightarrow \Z$,
\begin{equation}
\lfloor r\rfloor := \max\{z \in \Z  \,|\, z \le r \},\;\;\;
\lceil r \rceil := \min\{z \in \Z \,|\, z \ge r \},
\end{equation}
and the less conventional notation
$[\cdot]_p : \Z \rightarrow \{0, \dots, |p|-1\}$ to select a representative modulo $p$, 
by projecting an integer to $\Z/|p|\Z$
and then selecting its preimage in $\{0, \dots, |p|-1\} \subset \Z$.
In terms of our truncation notation,
\begin{equation}
\label{eq: mod floor ceiling identities}
[a]_b = a - \left\lfloor \frac{a}{b} \right\rfloor \!b,\;\;\;
[-a]_b = -a + \left\lceil \frac{a}{b} \right\rceil \!b,\;\; \mathrm{when}\;b >0.
\end{equation}

\subsection{Restating Theorem \ref{Thm:LY} as 
Theorem \ref{thm:  L-space interval in terms of beta/n}}

We are now equipped to re--express Theorem \ref{Thm:LY} in a more practical form,
describing the L-space slope interval in terms of any given slope
from the interior of that interval,
using the ``surgery label'' description of slopes.
Since the interval of L-space surgery labels always excludes $\infty$---its
being the surgery label of the canonical longitude---we can always describe
the interval of L-space surgery labels in terms of its minimum and maximum in $\Q$.

That is, given an L-space slope
$\mu_{\textsc{l}} = pm + ql \in H_1(\partial Y)$
from the interior of the L-space interval,
Theorem \ref{Thm:LY} tells us that a Dehn filling $Y(\mu)$
is an L-space if and only if
\begin{equation}
\label{eq: pi mu between pi delta - and pi delta +}
\pi(\tilde{\boldsymbol{\delta}}_-) 
\le
\pi(\mu)
\le
\pi(\tilde{\boldsymbol{\delta}}_+) 
\;\;\;
\text{for all }
\boldsymbol{\delta} \in \dtgz(Y),
\end{equation}
where $\pi$ denotes the surgery $\mu_{\textsc{l}}$-label,
\begin{equation}
\pi : H_1(\partial Y) \setminus \{0\} \to \Q \cup \{\infty\},\;\;\;\;
\mu \mapsto \pi(\mu):=  (\mu_{\textsc{l}} \cdot \mu)/(\mu \cdot l),
\end{equation}
and where, for each $\boldsymbol{\delta}  \in \dtgz(Y)$,
the lifts 
$\tilde{\boldsymbol{\delta}}_-, \tilde{\boldsymbol{\delta}}_+ \in
\iota^{-1}(\boldsymbol{\delta})$, with 
$\pi(\tilde{\boldsymbol{\delta}}_-) < \pi(\tilde{\boldsymbol{\delta}}_+)$,
are the two lifts of $\boldsymbol{\delta}$ closest to $\mu_{\textsc{l}}$
with respect to $\pi$, assuming $\dtgz(Y)$ nonempty.

Since $\dtgz(Y) \subset \iota(H_1(\partial Y))$,
we can express any $\boldsymbol{\delta} \in \dtgz(Y)$
as $\boldsymbol{\delta} = \delta\iota(m) + \gamma\iota(l)$.
Any lift $\tilde{\boldsymbol{\delta}} \in \iota^{-1}(\boldsymbol{\delta})$
of $\boldsymbol{\delta}$
then takes the form $\tilde{\boldsymbol{\delta}} = \delta m + \tilde{\gamma}l$,
satisfying
$\pi(\tilde{\boldsymbol{\delta}}) = 
(\mu_{\textsc{l}} \cdot \tilde{\boldsymbol{\delta}})/\delta
 = (p\tilde{\gamma} - q\delta)/\delta$,
for some 
$\tilde{\gamma} \equiv \gamma\, (\mod g)$.
In other words, we have
\begin{equation}
\label{eq: surgery labels of all the lifts of delta}
\left\{ \pi(\tilde{\boldsymbol{\delta}}) 
\mkern-2mu\left|\mkern2mu
\iota(\tilde{\boldsymbol{\delta}}) = {\boldsymbol{\delta}}\mkern-1.5mu
\right.\right\} 
\,=\; \frac{[p\gamma - q\delta]_{pg} + pg\Z}{\delta}.
\end{equation}
Since $\pi(\mu_{\textsc{l}}) = 0$, the fact that
$\mu_{\textsc{l}}$ lies in the interior of the L-space interval
implies that $0$ fails to belong to the
above set, {\it{i.e.}}, that $[p\gamma - q\delta]_{pg} \neq 0$, for all
$\boldsymbol{\delta} \in \dtgz(Y)$.
The lifts
$\tilde{\boldsymbol{\delta}}_+$ and $\tilde{\boldsymbol{\delta}}_-$
then evidently satisfy
$\pi(\tilde{\boldsymbol{\delta}}_+) = [p\gamma - q\delta]_{pg} /\delta$
and
$\pi(\tilde{\boldsymbol{\delta}}_-) = ([p\gamma - q\delta]_{pg} - pg) /\delta$
for all $\boldsymbol{\delta} \in \dtgz$,
and we can rewrite 
Theorem \ref{Thm:LY} as follows.

\begin{theorem}
\label{thm:  L-space interval in terms of beta/n}
Suppose $Y$ is Floer simple, with $\dtgz(Y) \neq \emptyset$.
If $\mu_{\textsc{l}} = pm + ql \in H_1(\partial Y)$ is an
L-space slope for $Y$, satisfying
$b_+^{\boldsymbol{\delta}} 
:= [p\gamma - q\delta]_{pg} \neq 0$
for all 
$\boldsymbol{\delta} = \delta \iota(m) + \gamma\iota(l) \in \dtgz(Y)$,
then the Dehn filling $Y(\mu)$ is an L-space if and only if
\begin{equation}
\label{eq: thm ineqality for beta/n}
\frac{b_-^{\boldsymbol{\delta}}}{\delta}
\le
\frac{\mu_{\textsc{l}} \cdot \mu}{\mu \cdot l}
\le
\frac{b_+^{\boldsymbol{\delta}}}{\delta}
\;\;\;
\text{for all }
\boldsymbol{\delta} = \delta\iota(m) + \gamma\iota(l) \in \dtgz(Y),
\end{equation}
where $b_-^{\boldsymbol{\delta}}\mkern-1mu := 
b_+^{\boldsymbol{\delta}}\mkern-1mu - \mkern2mu pg$,
and where we call $(\mu_{\textsc{l}} \cdot \mu)/(\mu \cdot l)$ the
surgery $\mu_{\textsc{l}}$-label for $\mu$.
If $\dtgz(Y) = \emptyset$, then 
$Y(\mu)$ is an L-space if and only if $\mu \notin \left< l \right>$,
{\it{i.e.}}, when $\mu$ has finite surgery label.
\end{theorem}
{\noindent{\bf{Remark.}} It is often more natural
to state the above result exclusively in terms of $\dt(Y)$.  That is,
if $Y$ is Floer simple and $\mu_{\textsc{l}}$ is an interior L-space
slope, then the L-space interval $\mathcal{L}(Y)$ is the smallest interval
containing $\mu_{\textsc{l}}$ with endpoints in $\iota^{-1}(\dt(Y))$.
This interval is open if its endpoints are equal, and closed otherwise.}
\vspace{.1cm}

Of course, one could
express the above criterion in any other basis.
To characterize L-space slopes in terms of conventional surgery coefficients,
for surgery along the knot core 
$K_{\mu_{\textsc{l}}} \subset Y(\mu_{\textsc{l}}) \setminus Y$
associated to a given interior
L-space slope $\mu_{\textsc{l}} = pm + ql$,
one must first choose a longitude, say
$\lambda_{\textsc{l}} := q^* m + p^* l$,
with $\mu_{\textsc{l}} \cdot \lambda_{\textsc{l}} = 1$
implying $pp^* - qq^* =1$.
Next, for each $\boldsymbol{\delta} \in \dtgz(Y)$, we express the lifts 
$\tilde{\boldsymbol{\delta}}_+, 
\tilde{\boldsymbol{\delta}}_- \in \iota^{-1}(\boldsymbol{\delta})$
flanking $\mu_{\textsc{l}}$ as
\begin{equation}
\tilde{\boldsymbol{\delta}}_+ 
= a^{\boldsymbol{\delta}}_+ \mu_{\textsc{l}}
+ b^{\boldsymbol{\delta}}_+ \lambda_{\textsc{l}},
\;\;\;\;\;
\tilde{\boldsymbol{\delta}}_- 
= a^{\boldsymbol{\delta}}_- \mu_{\textsc{l}}
+ b^{\boldsymbol{\delta}}_- \lambda_{\textsc{l}},
\end{equation}
with $b^{\boldsymbol{\delta}}_+$, $b^{\boldsymbol{\delta}}_-$
$a^{\boldsymbol{\delta}}_+$, and $a^{\boldsymbol{\delta}}_-$
satisfying
\begin{equation}
b^{\boldsymbol{\delta}}_+ := [p\gamma - q\delta]_{pg},\;\;\;\,
b^{\boldsymbol{\delta}}_- := b^{\boldsymbol{\delta}}_+ - pg,\;\;\;\,
\delta \;=\; a^{\boldsymbol{\delta}}_+ p + b^{\boldsymbol{\delta}}_+ q^*
= a^{\boldsymbol{\delta}}_- p + b^{\boldsymbol{\delta}}_- q^*.
\end{equation}
When $p >0$,
a straightforward calculation shows that
\begin{equation}
\frac{a^{\boldsymbol{\delta}}_-}{b^{\boldsymbol{\delta}}_-} 
\;<\;
-\frac{q^*\!}{p}
\;<\;
\frac{a^{\boldsymbol{\delta}}_+}{b^{\boldsymbol{\delta}}_+}\;\;\;\;
\text{for all}\;\;\boldsymbol{\delta} \in \dtgz(Y),
\end{equation}
and Theorem \ref{thm:  L-space interval in terms of beta/n}
takes the following form.

\begin{cor}
\label{cor: L-space criterion in surgery coefficients}
Suppose $Y$ is Floer simple, with $\dtgz(Y) \neq \emptyset$.
If $\mu_{\textsc{l}} = pm + ql$ with $p>0$ is an
L-space slope for $Y$, satisfying
$b_+^{\boldsymbol{\delta}} 
:= [p\gamma - q\delta]_{pg} \neq 0$
for all 
$\boldsymbol{\delta} = \delta \iota(m) + \gamma\iota(l) \in \dtgz(Y)$,
then for any longitude $\lambda_{\textsc{l}} = q^* m + p^* l$
(with $\mu_{\textsc{l}} \cdot \lambda_{\textsc{l}} = 1$),
the $\alpha/\beta$ surgery along 
$K_{\mu_{\textsc{l}}} \subset Y\mkern-2mu(\mu_{\textsc{l}})$---or
equivalently, the Dehn filling 
$Y\mkern-2mu(\mu)$ with 
$\mu:= \alpha \mu_{\textsc{l}} + \beta \lambda_{\textsc{l}}$---is an L-space
if and only if
\begin{equation}
\frac{\alpha}{\beta} \le \frac{a^{\boldsymbol{\delta}}_-}{b^{\boldsymbol{\delta}}_-}
\;\;\mathrm{or}\;\;
\frac{a^{\boldsymbol{\delta}}_+}{b^{\boldsymbol{\delta}}_+}
\le
\frac{\alpha}{\beta}
\;\;\;\;
\text{for all}\;\;\boldsymbol{\delta} \in \dtgz(Y),
\end{equation}
where $\iota(a^{\boldsymbol{\delta}}_+ \mu_{\textsc{l}}
+ b^{\boldsymbol{\delta}}_+ \lambda_{\textsc{l}})
= \iota(a^{\boldsymbol{\delta}}_- \mu_{\textsc{l}}
+ b^{\boldsymbol{\delta}}_- \lambda_{\textsc{l}})
= \boldsymbol{\delta}$, with
$b^{\boldsymbol{\delta}}_- := b^{\boldsymbol{\delta}}_+ -pg$,
for each $\boldsymbol{\delta} \in \dtgz(Y)$.
In such case,
the left hand inequality obtains when 
$\beta /n <0$, the right hand when $\beta/n > 0$,
and we regard both inequalities as vacuously true when $\beta/n = 0$,
where 
$n:= \mu \cdot l
= \alpha p + \beta q^*$.
If $\dtgz(Y) = \emptyset$, then $Y(\mu)$ is an L-space if and only if $n \neq 0$. 
\end{cor}

One could also characterize L-space slopes in terms of the Dehn filling basis, $m, l$.
If we take $\mu_{\textsc{l}} = pm + ql$ to be an interior L-space slope
with $p > 0$, then
for any $\boldsymbol{\delta} = \delta \iota(m) + \gamma \iota(l) \in \dtgz(Y)$,
it follows from the two identities in 
(\ref{eq: mod floor ceiling identities}) that
\begin{equation}
[p\gamma-q\delta]_{pg}
= [-q\delta]_p \,+\, p\! \left[\gamma - 
{{\textstyle{\left\lceil \!\frac{q}{p}\delta\! \right\rceil}}}\right]_{{\mkern-2.5mu}g};
\;\;\;\;
-[q\delta - p\gamma]_{pg}
= -[q\delta]_p \,-\, 
p\! \left[{{\textstyle{\left\lfloor \!\frac{q}{p}\delta\! \right\rfloor}}}  
- \gamma\right]_{{\mkern-2.5mu}g},
\end{equation}
from which it follows that the lifts
$\tilde{\boldsymbol{\delta}}_+, 
\tilde{\boldsymbol{\delta}}_- \in \iota^{-1}(\boldsymbol{\delta})$
adjacent to $\mu_{\textsc{l}}$ take the form
\begin{equation}
\label{eq: defs of delta-+ in filling basis}
\tilde{\boldsymbol{\delta}}_+ =
\delta m +
\left(\left\lceil \!\frac{q}{p}\delta\! \right\rceil
+ \left[\gamma - \left\lceil \!\frac{q}{p}\delta\! \right\rceil \right]_g\right)\!l,
\;\;\;\;
\tilde{\boldsymbol{\delta}}_- =
\delta m +
\left(\left\lfloor \!\frac{q}{p}\delta \! \right\rfloor
- \left[\left\lfloor \!\frac{q}{p}\delta \! \right\rfloor - \gamma\right]_g \right)\!l,
\end{equation}
As expected, these are the lifts of $\boldsymbol{\delta}$
with Dehn filling slope closest to $p/q$
(regardless of whether $p > 0$), and
Theorem
\ref{thm:  L-space interval in terms of beta/n}
takes the following form.

\begin{cor}
\label{cor: L-space criterion in Dehn filling basis}
Suppose $Y$ is Floer simple, with $\dtgz(Y) \neq \emptyset$.
If $\mu_{\textsc{l}} = pm + ql$ is an
L-space slope for $Y$, satisfying
$p\gamma - q\delta \not\equiv 0 \,(\mod pg)$
for all 
$\boldsymbol{\delta} = \delta \iota(m) + \gamma\iota(l) \in \dtgz(Y)$,
then $\mu = nm + n'l$ is an L-space slope for $Y$ 
if and only if
$\frac{n}{n'} \in I^{\boldsymbol{\delta}}$ for all
$\boldsymbol{\delta} \in \dtgz(Y)$, where, for each
$\boldsymbol{\delta} \in \dtgz(Y)$,
$I^{\boldsymbol{\delta}}$ is the closed interval
in $\Q \cup \{\infty\}$ which exludes 0 and has endpoints
\begin{equation}
\label{eq: endpoints for dehn filling l space corollary}
\frac{\delta}
{\left\lceil \!\frac{q}{p}\delta\! \right\rceil
+ \left[\gamma - \left\lceil \!\frac{q}{p}\delta\! \right\rceil \right]_g},
\;\;\;\;
\frac{\delta}
{\left\lfloor \!\frac{q}{p}\delta \! \right\rfloor
- \left[\left\lfloor \!\frac{q}{p}\delta \! \right\rfloor - \gamma\right]_g}.
\end{equation}
If $\dtgz(Y) = \emptyset$, then 
$Y(\mu)$ is an L-space if and only if $n \neq 0$. 
\end{cor}

{\noindent{\bf{Example.}}
Suppose $K \subset S^3$ is a Floer simple knot of positive genus $g(K)$,
with Alexander polynomial $\Delta(K)$.
Then $Y := S^3 \setminus \nu(K)$ is Floer simple, and since $K \subset S^3$
Floer simple implies $\deg \Delta(K) = 2g(K)$, the hypothesis $g(K)>0$ implies
$\dtgz(Y) \neq \emptyset$.
Since $H_1(Y)$ is torsion free,
the endpoints of $I^{\boldsymbol{\delta}}$
reduce to 
$\delta/ \mkern-3mu\left\lceil \mkern-1.5mu \frac{q}{p}\mkern-1.5mu \right\rceil$
and
$\delta/ \mkern-3mu\left\lfloor \mkern-1.5mu \frac{q}{p}\mkern-1.5mu \right\rfloor$
for each 
$\boldsymbol{\delta} = \delta \iota(m)  \in \dtgz(Y)$.
We already know that the infinity filling $Y(1m+0l) = S^3$ is an L-space.
Thus
(if necessary replacing $K$ with its mirror and 
using $-\frac{n}{n^{\prime}\!}$ for $\frac{n}{n^{\prime}\!}$ in
(\ref{eq: example l space interval for k in s3})),
we know that $Y(pm + 1l)$ is an L-space for any $p>0$ sufficiently large.
Taking $p > \max_{\boldsymbol{\delta}\in\dtgz(Y)}\mkern-2mu\delta$
then makes the endpoints of each $I^{\boldsymbol{\delta}}$
become
$\delta/ \mkern-3mu\left\lceil \mkern-1.5mu \frac{1}{p}\mkern-1.5mu \right\rceil = \delta$
and
$\delta/ \mkern-3mu\left\lfloor \mkern-1.5mu \frac{1}{p}\mkern-1.5mu \right\rfloor
= +\infty$, and we recover the
well known result that for $n' \neq 0$, $Y(nm + n'l)$ is an 
L-space if and only if
\begin{equation}
\label{eq: example l space interval for k in s3}
\frac{n}{n'} \ge \max_{\delta \iota(m) \in\dtgz(Y)}\delta =
\deg \tc(Y)
= (\deg \Delta(K)) - 1 = 2g(K) - 1.
\end{equation}
}

\subsection{Set-up for proof of Theorem
\ref{thm:  L-space interval in terms of beta/n}}

We begin by making some simplifying assumptions,
without loss of generality.
\begin{prop}
\label{prop: simpliftying assumptions for Theorem 1}
Suppose that $Y$ is Floer simple, 
that $\mu_{\textsc{l}} = pm + ql$ is an L-space slope,
and that we wish to determine if $\mu = nm + n'l$
is an L-space slope for $Y$.
For purposes of proving Theorem
\ref{thm:  L-space interval in terms of beta/n},
we may assume, without loss of generality, that
$p, \beta >0$, $n \neq 0$, $pg > \deg_{[\bar{m}]} \!\tc(Y)$, 
and $\gcd(p, q) = \gcd(pg, \beta) = 1$, 
where $\beta:= \mu_{\textsc{l}} \cdot \mu$ and
$g := \left|\left< \iota(l) \right> \right|$,
with $\iota : H_1(\partial Y ) \to H_1(Y)$
the map induced on homology by inclusion.
\end{prop}
\begin{proof}
Theorem 
\ref{thm:  L-space interval in terms of beta/n}
already correctly characterizes the
cases of $\beta = 0$, corresponding to the Dehn filling $Y(\mu_{\textsc{l}})$,
which we already know to be an L-space, and $n=0$,
for which the filling $Y(l)$ is not a rational homology sphere, hence not an L-space.
Likewise, we know that any L-space slope $\mu_{\textsc{l}} = pm + ql$ must have $p \neq 0$.
Since we are free to replace $\mu_{\textsc{l}}$ with
$-\mu_{\textsc{l}}$ or
$\mu$ with $-\mu$, we may take $p, \beta > 0$ without loss.
Lastly, by Lemma~\ref{Lem:Interval}, we can approximate
$\mu_{\textsc{l}}$ with a primitive 
L-space slope
$\mu'_{\textsc{l}} = p' m + q' l$ (with $q' \neq 0$)
such that $p'g > \deg_{[\bar{m}]} \!\tc(Y)$ and
$\gcd(p'g, \beta') = 1$,
where $\beta' := \mu'_{\textsc{l}} \cdot \mu$.
\end{proof}

We henceforth consider the assumptions of 
Proposition \ref{prop: simpliftying assumptions for Theorem 1} to hold.
Given such initial data, we have a primary tool from
Heegaard Floer homology to determine whether $\mu$ is
an L-space slope for the Floer simple manifold $Y$: namely, 
Proposition~\ref{Prop:RedBlueBlack}.
To exploit this proposition, we must exhibit
$Y(\mu)$ as zero surgery on an L-space,
given the L-space slope $\mu_{\textsc{l}}$ for $Y$.
Fortunately, a standard such construction exists,
whereby we first express $Y(\mu)$ as some $\alpha/\beta$-surgery
on $Y(\mu_{\textsc{l}})$, and then reexpress this as a zero surgery
on a connected summed knot inside $Y(\mu_{\textsc{l}}) \# L(\beta, \alpha^*)$,
for some $\alpha^* \equiv -\alpha^{-1} \, (\mod \beta)$.

\subsection{$\boldsymbol{Y}\mkern-1mu(\boldsymbol{\mu})$ as zero surgery on an L-space}
To describe this construction more explicitly,
we first let $K_{\textsc{u}} \subset S^3$ denote the unknot, and
take $(m_1, l_1)$
and $(m_2, l_2)$ as respective bases
for $H_1(\partial Y)$ and $H_1(\partial (S^3\setminus K_{\textsc{u}}))$,
such that $m_1 \cdot l_1 = m_2 \cdot l_2 = 1$,
with 
$l_1$ generating $\iota_1^{-1}(T_1)$, where $T_1 := \Tors(H_1(Y)$,
and with $l_2$ generating 
$\ker \iota_2$,
where
$\iota_1 : H_1(\partial Y) \to H_1(Y)$
and 
$\iota_2 : H_1(\partial (S^3\setminus K_{\textsc{u}})) \to H_1(S^3\setminus K_{\textsc{u}})$
are the maps induced on homology by inclusion.
Write
\begin{equation}
\mu := n m_1 + n' l_1,\;\;\;\;
\mu_1:= \mu_{\textsc{l}} = pm_1 + ql_1,\;\;\;\;
\mu_2 := \beta m_2 + \alpha^* l_2
\end{equation}
for our test slope $\mu$ and given L-space slope $\mu_1 = \mu_{\textsc{l}}$,
and for $\mu_2$ constructed to produce the desired lens space
$(S^3 \setminus K_{\textsc{u}})(\mu_2) = L(\alpha^*,\beta)$, with 
$\beta := \mu_{\textsc{l}} \cdot \mu$ and 
$\alpha^* \!:=[-n^{-1}p]_{\beta}$.
Setting $q^* \!:= [-q^{-1}]_p$, 
write $\alpha$, $p^*$, and $\beta^*$ for the (integer) solutions
to the respective equations
$n = \alpha p + \beta q^*$, $p p^* - q q^* = 1$,  and 
$\beta  \beta^* - \alpha^* \alpha = 1$, so that
\begin{equation}
\lambda_1 := q^* m_1 + p^* l_1,\;\;\;\; \lambda_2 := \alpha m_2 + \beta^* l_2
\end{equation}
serve as longitudes, 
satisfying $\mu_1 \cdot \lambda_1 = \mu_2 \cdot \lambda_2 =  1$.
Note that this makes $\mu = \alpha \mu_1 + \beta \lambda_1$.

Let $Y_{\!\scon}$ denote the connected sum knot complement
\begin{equation}
Y_{\!\scon} := Y\mkern-2mu(\mu_1) \mkern1mu \# \mkern1mu 
(S^3 \mkern-1.5mu\setminus\mkern-1.5mu K_{\textsc{u}})(\mu_2)\; \setminus\;
K_{\mu_1} \# K_{\mu_2},
\end{equation}
where $K_{\mu_1}  \subset Y\mkern-2mu(\mu_1)$ and 
$K_{\mu_2} \subset (S^3 \setminus K_{\textsc{u}})(\mu_2) = L(\beta, \alpha^*)$
are the knot cores associated to the respective fillings by $\mu_1$ and $\mu_2$.
If we write $\iota : H_1(\partial Y_{\!\scon}) \to H_1(Y_{\!\scon})$,
$f_1 : H_1(Y) \to H_1(Y_{\!\scon})$, and $f_2: H_1(S^3 \setminus K_{\textsc{u}}) \to H_1(Y_{\!\scon})$
for the maps induced on homology by the corresponding inclusions,
then $f_1 \oplus f_2$ descends to the isomorphism,
\begin{equation}
\label{eq: thm 1 f1 f2 iso whole}
\left(H_1(Y) \oplus H_1(S^3 \setminus K_{\textsc{u}})\right) / 
(\iota_1(\mu_1) \sim \iota_2(\mu_2))
\to H_1(Y_{\!\scon}),
\end{equation}
which, since $H_1(S^3 \setminus K_{\textsc{u}})$ is torsion free,
restricts to the isomorphism,
\begin{equation}
\label{eq: thm 1 f1 f2 iso boundary image}
\left(\iota_1(H_1(\partial Y)) \oplus \iota_2(H_1(\partial(S^3 \setminus K_{\textsc{u}})))\right)
/ 
(\iota_1(\mu_1) \sim \iota_2(\mu_2))
\to \iota(H_1(\partial Y_{\!\scon})).
\end{equation}

For the knot $K_{\mu_1} \# K_{\mu_2}$
with meridian $\mu_{\!\scon}$,
we can splice the longitudes 
$\lambda_1$ and $\lambda_2$ together to form a longitude of class
$\lambda_{\!\scon} \in \iota^{-1}(f_1(\lambda_1) + f_2(\lambda_2))
\subset \iota(H_1(\partial Y_{\!\scon}))$.
The Dehn filling $Y_{\!\scon}(\lambda_{\!\scon})$
then has homology elements satisfying
\begin{equation}
f_1\iota_1(\mu_1)
= f_2\iota_2(\mu_2)
= {\textstyle{\frac{\beta}{\alpha}}} f_2\iota_2(\lambda_2)
= -{\textstyle{\frac{\beta}{\alpha}}} f_1\iota_1(\lambda_1),
\end{equation}
implying that 
$f_1\iota_1(\mu) = f_1\iota_1(\alpha \mu_1 + \beta \lambda_1) = 0$ in 
$H_1(Y_{\!\scon}(\lambda_{\!\scon}))$.
Since, in addition, we know that $Y_{\!\scon}$ is homeomorphic to $Y$,
it follows that $Y(\mu) = Y_{\!\scon}(\lambda_{\!\scon})$,
and this is zero surgery on the L-space 
$Y(\mu_1) \# L(\alpha^*, \beta) = Y\mkern-2mu(\mu_{\!\scon})$.

Since $\gcd(pg, \beta) = 1$
and $H_1(S^3 \setminus U)$ is torsion free,
it follows from the isomorphisms 
(\ref{eq: thm 1 f1 f2 iso whole}) and
(\ref{eq: thm 1 f1 f2 iso boundary image}) that
$f_1$ restricts to isomorphisms $T_1 \stackrel{\sim}{\to} T$ and 
$T^{\partial}_1 \stackrel{\sim}{\to} T^{\partial}$,
where $T_1 := \Tors(H_1(Y))$,  $T :=  \Tors(H_1(Y_{\!\scon}))$,
$T_1^{\partial} := T_1 \cap \iota_1(H_1(\partial Y))$, and
$T^{\partial} :=  T \cap \iota(H_1(\partial Y_{\!\scon}))$.
It also follows that we can choose $l \in \iota^{-1}(T^{\partial})$
and $m \in H_1(\partial Y_{\!\scon})$
with $m \cdot l = 1$ such that  
$f_1$ and $f_2$ satisfy
\begin{align}
\label{eq: defs of f1 and f2}
f_1: \iota_1(m_1) \mapsto \beta \iota(m),
\,\;\;\;&\;\;\;
f_2: \iota_2(m_2) \mapsto p \iota(m) + q \xi \iota(l),
    \\ \nonumber
f_1: \iota_1(l_1) \mapsto \beta \xi \iota(l),\;\;\;
\;\;\;&\;\;\;
\hphantom{f_2\iota_2(m_2) = p \iota(m) + q \xi \iota(l),}
\end{align}
on the images of $\iota_1$ and $\iota_2$,
for some $\xi \in \Z/g$,
with $g := \left|T^{\partial}\right| = \left|T^{\partial}_1\right|$.  We then have
\begin{align}
\label{eq:  defs of mu connect and lambda connect}
\iota(\mu_{\scon}) 
&= f_1\iota_1(\mu_1) = f_2\iota_2(\mu_2)
= \beta p \iota(m) + \beta q \xi \iota(l),
       \\
\iota(\lambda_{\scon}) 
&= f_1\iota_1(\lambda_1) +  f_2\iota_2(\lambda_2)
= n \iota(m) + n' \xi \iota(l),
\nonumber
\end{align}
where we used the facts that
\begin{equation}
n = \alpha p + \beta q^*,\;\;\;\; n' = \alpha q + \beta p^*.
\end{equation}
The condition that $\mu_{\scon} \!\cdot\! \lambda_{\scon} = 1$ determines $\xi$,
which we shall not need.

\subsection{Applying the ``coloring condition'' of Proposition~\ref{Prop:RedBlueBlack}}

Since this section uses the Euler characteristic of knot Floer homology,
which we express in terms of the Turaev torsion, regarded as an element
of the Laurent series group ring of homology,
we briefly introduce generators
$\bar{m}$, $\bar{m}_1$, and $\bar{m}_2$ for $H_1(Y_{\!\scon})/T$, $H_1(Y)/T_1$, and 
$H_1(S^3 \setminus K_{\textsc{u}})$, respectively, with signs chosen so that
\begin{equation}
\iota(m) \in +g\bar{m} + T,
\;\;\;\;
\iota(m_1) \in +g\bar{m}_1 + T_1,
\;\;\;\;
\iota(m_2) = \bar{m}_2.
\end{equation}
For notational brevity, we also set
\begin{equation}
t:= [\bar{m}],
\;\;\;\;
t_1:= [\bar{m}_1],
\;\;\;\;
t_2:= [\bar{m}_2],
\end{equation}
where $[\cdot]$ indicates inclusion into the Laurent series group ring
of the relevant homology group.

In order to use Proposition~\ref{Prop:RedBlueBlack}, we need the support of the Euler characteristic of the
 knot Floer homology of $K_{\scon} \subset Y_{\!\scon}(\lambda_{\scon})$.
Since $\widehat{HFK}$ tensors on connected sums,
its Euler characteristic $\chi^{\widehat{\textsc{hfk}}}$
turns tensor product into multiplication, and the support
function $S[\cdot]$ on (Laurent series) group rings converts this
multiplication of polynomials into addition of sets, yielding
\begin{equation}
S\!\left[\chi^{\widehat{\textsc{hfk}}}(Y_{\!\scon}(\lambda_{\scon}), K_{\scon})\right]
= f_1 S\!\left[\chi^{\widehat{\textsc{hfk}}}(Y({\mu_1}), K_{\mu_1})\right]
+ f_2 S\!\left[\chi^{\widehat{\textsc{hfk}}}((S^3 \setminus K_{\textsc{u}})(\mu_2), K_{\mu_2})\right].
\end{equation}
Proposition~\ref{Prop:Chi} tells us that
\begin{align}
S\mkern-3.8mu\left[\chi^{\widehat{\textsc{hfk}}}
(Y\mkern-2mu({\mu_1}\mkern-1mu), K_{\mkern-1mu\mu_1})\mkern-2mu\right]
&\!=
S\!\left[(1 - [\iota_1(\mu_1)]) \cdot
\left((1-t_1)^{-1}{\textstyle{\sum_{h \in T_1}\![h]}} \,-\,\tc(Y) \right) \right]
           \\ \nonumber
&\!=
S\!\left[ \frac{1-t_1^{pg}}{1-t_1}{\textstyle{\sum_{h \in T_1}\![h]}}
- \tc(Y) \;+\;  [\iota_1(\mu_1)]\tc(Y)
\right]
           \\ \nonumber
&\!=
\left(S[\tau(Y)] \cap \mkern-1.5mu\left( \{0, \ldots, pg-1\}\bar{m}_1 + T_1\right) \right)
\; \amalg \; \left(S[\tc(Y)] + \iota_1\mkern-1mu(\mu_1\mkern-.5mu)) \right),
\end{align}
where $\tau(Y) \in \Z[t^{-1}, t]][T] \supset \Z[H_1(Y)]$ is the Turaev torsion,
$\tc(Y)$ is the torsion complement as defined in
(\ref{eq: torsion complement definition}), and 
we used our simplifying assumption that
$\deg_{t_1} \!\tc(Y) < pg$.  Similarly, we have
\begin{align}
S\!\left[\chi^{\widehat{\textsc{hfk}}}((S^3 \setminus K_{\textsc{u}})(\mu_2), K_{\mu_2})\right]
&= S\!\left[(1 - [(\iota_2(\mu_2)])\cdot\tau(S^3 \setminus K_{\textsc{u}})\right]
           \\ \nonumber
&= S[(1-t_2^\beta)/(1-t_2)]
           \\ \nonumber
&= \{0, \ldots, \beta - 1\} \iota_2(m_2).
\end{align}

Thus, if we set
\begin{align}
A_0 &:= f_1\!\left(S[\tau(Y)] \cap \left( \{0, \ldots, pg-1\} \bar{m}_1 + T_1\right) \right) 
\;+\; \{0, \ldots, \beta - 1\} f_2\iota_2(m_2),
           \\ \nonumber
A_1 &:= f_1S[\tc(Y)] + \iota(\mu_{\scon}) 
\;+\; \{0, \ldots, \beta - 1\} f_2\iota_2(m_2),
\end{align}
then in the language of 
Proposition~\ref{Prop:RedBlueBlack}, we have
\begin{align}
S_{\textsc{black}}
\;:=\; S\!\left[\chi^{\widehat{\textsc{hfk}}}(Y_{\!\scon}(\lambda_{\scon}), K_{\scon})\right]
\;=\; A_0 \amalg A_1, \hphantom{1212345}
           \\ \nonumber
S_{\textsc{red}} := S_{\textsc{black}}  + \iota(\mu_{\scon})\Z_{>0},
\;\;\;\;\;\;\;\;\;\;
S_{\textsc{blue}} := S_{\textsc{black}}  - \iota(\mu_{\scon})\Z_{>0}.
\end{align}
Using the fact that $\iota(\mu_{\scon}) = \beta f_2\iota_2(m_2)$, one can easily verify that
\begin{align}
(S_{\textsc{blue}} \mkern-2mu -\mkern-2mu S_{\textsc{red}})\cap (\bar{m}\Z_{> 0}\!+\!T)
&= \left((A_1 - \iota(\mu_{\scon})) - (A_0 + \iota(\mu_{\scon}))\right)
\cap (\bar{m}\Z_{> 0}\!+\!T)
           \\ \nonumber
&=\left( f_1(S[\tc(Y)] \mkern-2mu-\mkern-2mu S[\tau(Y)]) \mkern-2mu-\mkern-2mu f_2\iota_2(m_2)\Z_{>0} \right) \mkern-1.5mu\cap\mkern-1.5mu (\bar{m}\Z_{> 0}\!+\!T).
\end{align}
Proposition~\ref{Prop:RedBlueBlack} then implies $Y_{\!\scon}(\lambda_{\!\scon})$ is an L-space if and only if
\begin{equation}
\label{eq: test set for l-space}
\iota(\lambda_{\scon}) \Z  
\mkern2mu \cap \mkern1mu
(S_{\textsc{blue}}  - S_{\textsc{red}})
\mkern1mu \cap \mkern1mu
(\bar{m}\Z_{> 0}+T) = \emptyset.
\end{equation}

Suppose the above set is nonempty, hence contains some element
$b\iota(\lambda_{\scon})$ such that
\begin{equation}
\label{eq: intersection of l space test sets}
b\iota(\lambda_{\scon}) = f_1(h_{\mathrm{c}} - h)
- k f_2\iota_2(m_2) \;\in \bar{m}\Z_{> 0}+T
\end{equation}
with $b\in \Z_{\neq 0}$, $k \in \Z_{>0}$, $h_{\mathrm{c}} \in S[\tc(Y)]$,
and $h \in S[\tau(Y)]$.
Since $b\iota(\lambda_{\scon}),  k f_2\iota_2(m_2) \in \iota(H_1(Y_{\scon}))$,
we know that
$f_1(h_{\mathrm{c}} - h) \in \iota(H_1(Y_{\scon}))$, implying
$h_{\mathrm{c}} - h \in \iota(H_1(Y))$.
Moreover, since $b\iota(\lambda_{\scon}) \in \bar{m}\Z_{> 0}+T$,
we know that $h_{\mathrm{c}} - h \in \bar{m}_1\Z_{> 0}+T_1$.  In other words,
\begin{equation}
h_{\mathrm{c}} - h \in (S[\tc(Y)] - S[\tau(Y)]) \cap \iota_1(m_1 \Z_{>0} + l_1 \Z) 
=: \dtgz(Y). 
\end{equation}

Writing 
$h_{\mathrm{c}} - h = \delta \iota(m_1) + \gamma \iota(l_1) \in \dtgz(Y)$
and evaulating 
$f_1$, $f_2$, and $\iota(\lambda_{\scon})$ as expressed in
(\ref{eq: defs of f1 and f2}) and
(\ref{eq:  defs of mu connect and lambda connect}), we
transform (\ref{eq: intersection of l space test sets}) into
\begin{equation}
\label{eq: intersection of l space test sets, in terms of delta}
(bn)\iota(m) + (bn')\xi \iota(l) = (\beta \delta - kp ) \iota(m) + 
(\beta \gamma  - k q)\xi \iota(l),
\end{equation}
which, since $n m_1 + n' l_1 = \alpha \mu_1 + \beta \lambda_1=
(\alpha p + \beta q^*) m_1 + (\alpha q + \beta p^*)l_1$,
yields the two equations
\begin{align}
\label{eq: bn equation}
b(\alpha p + \beta q^*)
&= \beta \delta - kp > 0,
       \\
\label{eq: bn' equation}
b(\alpha q + \beta p^*)
&\equiv  \beta \gamma  - k q \;(\mod g).
\end{align}
One can use the identity $p p^* - q q^* = 1$ to solve the above two equations 
simultaneously for $b$,
obtaining $b \equiv p \gamma - q \delta\; (\mod g)$.
Moreover, taking the first equation modulo $p$ implies
$b \equiv -q\delta\; (\mod p)$.
Thus any solution in $b$ to
(\ref{eq: intersection of l space test sets, in terms of delta}) must satisfy
$b \equiv p \gamma - q \delta \; (\mod pg)$.

\subsection{Completing the proof of Theorem 
\ref{thm:  L-space interval in terms of beta/n}}

For each 
$\boldsymbol{\delta} = \delta \iota(m_1) + \gamma \iota(l_1)  \in \dtgz(Y)$,
set
$b_-^{\boldsymbol{\delta}} := [p\gamma - q\delta]_{pg}-pg$ and 
$b_+^{\boldsymbol{\delta}} := [p\gamma - q\delta]_{pg}$.
Note that our earlier assumption of $pg > \deg_{t_1}\tc(Y)$ ensures that
$b_+^{\boldsymbol{\delta}} \neq 0$ and $|b_-^{\boldsymbol{\delta}}| < pg$.
We want to show that
$Y_{\!\scon}(\lambda_{\mu_{\scon}})$ is an L-space---or in other words,
that (\ref{eq: bn equation}) and (\ref{eq: bn' equation}) have no solution
$(b,k) \in \Z \times \Z_{>0}$
for any $\delta \iota(m_1) + \gamma \iota(l_1)  \in \dtgz(Y)$---if and only if
\begin{equation}
\frac{b_-^{\boldsymbol{\delta}}}{\delta}
\le
\frac{\beta}{n}
\le
\frac{b_+^{\boldsymbol{\delta}}}{\delta}
\;\;\;
\text{for all }
\boldsymbol{\delta} = \delta \iota(m_1) + \gamma \iota(l_1)   \in \dtgz(Y).
\end{equation}

First, consider the case in which $n>0$.
Suppose there exists 
$\boldsymbol{\delta} = \delta \iota(m_1) + \gamma \iota(l_1)  \in \dtgz(Y)$ for which
$\beta/n > b_+^{\boldsymbol{\delta}} /\delta$.  Since $n, \delta >0$, this implies 
$0< b_+^{\boldsymbol{\delta}} n < \beta\delta$.
Thus, since $b_+^{\boldsymbol{\delta}} n \equiv 
(-q\delta)(\beta q^*) \equiv \beta \delta\; (\mod p)$,
 there exists $k_0 \in \Z_{>0}$ such that
$b_+^{\boldsymbol{\delta}} n = \beta \delta - k_0p >0$.
Thus $(b_+^{\boldsymbol{\delta}}, k_0)$ provides a solution for $(b, k)$ in
(\ref{eq: bn equation}), which, together with the relation 
$b_+^{\boldsymbol{\delta}} \equiv p\gamma - q\delta \,(\mod g)$, implies
(\ref{eq: bn' equation}) also holds for $(b,k) = (b_+^{\boldsymbol{\delta}}, k_0)$,
and so $Y_{\!\scon}(\lambda_{\scon})$ is not an L-space.

Conversely, suppose that $Y_{\!\scon}(\lambda_{\scon})$ is not an L-space.
Then there exist
$\boldsymbol{\delta} = \delta \iota(m_1) + \gamma \iota(l_1)  \in \dtgz(Y)$ and 
$(b,k) \in \Z \times \Z_{>0}$ for which 
(\ref{eq: bn equation}) and (\ref{eq: bn' equation}) hold.
In particular, (\ref{eq: bn equation}) implies $bn < \beta \delta$
and $b > 0$,
while (\ref{eq: bn equation}) and (\ref{eq: bn' equation}) together
imply $b \equiv b_+^{\boldsymbol{\delta}} (\mod pg)$,
requiring $b \ge  b_+^{\boldsymbol{\delta}}$.
Thus $\beta \delta > bn \ge  b_+^{\boldsymbol{\delta}} n$, implying 
$\beta/n >b_+^{\boldsymbol{\delta}}/\delta$.

The argument for the case of $n < 0$ is nearly identical,
but with a few signs and inequalities reversed,
and this completes the proof of the theorem.   \qed

\bigskip
\section{Seifert Fibered L-spaces}
\label{s: Seifert fibered L-spaces}

To illustrate the usage of our new L-space interval tool $\dt$,
in this section we exploit
Theorem \ref{thm:  L-space interval in terms of beta/n}
to offer a simple alternative proof of a known result: namely,
the classification of Seifert fibered spaces over $S^2$ which are L-spaces.
We restrict to the $S^2$ case because it is the most interesting one,
as no higher genus Seifert
fibered spaces are L-spaces, and all oriented Seifert fibered spaces
over $\R\P^2$ are L-spaces \cite{BGW}.

\subsection{Seifert fibered L-spaces, a history}
\label{ss: foliations on SFS, a history}
Up until now, the classification of Seifert fibered L-spaces
has relied, at least in one direction, 
on the classification of oriented Seifert fibered spaces $M$ over $S^2$
admitting transverse foliations,
a problem which dates back at least to 
1981, when Eisenbud, Hirsch, and Neumann \cite{EHN}
re-expressed this foliations problem in terms of
a criterion on representations of $\pi_1(M)$
in $\widetilde{\mathrm{Homeo}_+}S^1\!$, the universal cover of the
group of orientation-preserving homeomorphisms of $S^1\mkern-4.5mu$.

A few years later, Jankins and Neumann \cite{JankinsNeumann} 
reformulated the criterion of
\cite{EHN} in terms of Poincar{\'e}'s ``rotation number'' invariant
on $\widetilde{\mathrm{Homeo}_+}S^1\!$, a development which, along with
the correct conjecture that this criterion is met in 
$\widetilde{\mathrm{Homeo}_+}S^1\!$ if and only if it
is met in a smooth Lie subgroup thereof, allowed them to
write down an explicit characterization of Seifert fibered manifolds
over $S^2$ admitting transverse foliations.
With the exception of one special case, they also showed
that this list was complete.
It took more than a decade before Naimi \cite{Naimi}
resolved this outstanding case using dynamical methods,
and more than a decade after that before Calegari and Walker \cite{ziggurat}
generalized Naimi's methods to provide a proof of the
Jankins-Neumann classification that did not appeal
to smooth Lie subgroups.

In the late 1990's,
Eliashberg and Thurston \cite{EliashbergThurston} proved that
one can associate a weakly symplectically fillable contact structure
to any $C^2$ cooriented taut foliation on a closed three-manifold---a result
which Kazez and Roberts \cite{KazezRobertsCzero}, 
and independently Bowden \cite{Bowden}, 
have recently extended
to $C^0$ foliations.  Since Ozsv{\'a}th and Szab{\'o} have \cite{OSGen}
shown that this contact structure gives rise to a nontrivial class in
Heegaard Floer homology, this proves that L-spaces do not
admit co-oriented taut foliations.

In the converse direction,
Lisca and Mati{\'c} \cite{LiscaMatic}
proved that a Seifert fibered manifold $M$ over $S^2$
admits contact structures in each orientation
which are transverse to the fibration if and only if
$M$ belongs to the explicit set characterized by Jankins and Neumann.
Lisca and Stipsicz then showed \cite{LSIII} that
if there is an orientation on a Seifert fibered manifold $M$ over $S^2$
for which no positive contact structure is transverse to the fibration,
then $M$ is an L-space.

Since our own answer matches that of Jankins and Neumann,
one could take the non-L-space/transverse-foliation
equivalence for Seifert fibered manifolds over $S^2$ as a corollary
of Theorem \ref{thm: non L-space interval for seifert fibered spaces} below.
As for our L-space classification itself, however, the proof no longer requires
foliations, dynamical methods, or even
(after the proof of Theorem \ref{thm:  L-space interval in terms of beta/n})
contact or symplectic geometry.
It only uses
ordinary homology
and one computation
of Turaev torsion from a homology presentation.

\subsection{Conventions and bases}
\label{ss: sfs construction and conventions}

To construct a Seifert-fibered space with $n$ exceptional fibers over $S^2$,
we start with the trivial circle fibration $S^1 \times S^2$,
and remove $n+1$ solid tori,
$S^1 \times D^2_i$, $i \in \{0, \dots, n\}$,
yielding a trivial circle fibration over the $n+1$--punctured sphere,
\begin{equation}
\hat{Y} := S^1 \times (S^2 \setminus {\textstyle{\coprod_{i=0}^n D^2_i}}),\;\;\;\;\;
\partial \hat{Y} = {\textstyle{\coprod_{i=0}^n \partial_i \hat{Y}}},
\end{equation}
where $\partial_i \hat{Y}$ denotes the $i^{\mathrm{th}}$ toroidal boundary component,
$\partial_i \hat{Y} := - \partial(S^1 \times D^2_i)$.

Next, we choose presentations for $H_1(\hat{Y})$ and $H_1(\partial_i \hat{Y})$
in terms of the regular fiber class $f \in H_1(\hat{Y})$ and classes horizontal to this fiber.
For each $i \in \{0, \ldots, n\}$, we take
$(\tilde{f}_i, -\tilde{h}_i)$ as a reverse-oriented basis for $H_1(\partial_i \hat{Y})$.
Here, $\tilde{h}_i \in H_1(\partial_i \hat{Y})$ denotes
the meridian of the excised solid torus $S^1 \times D^2_i$,
and if we write
$\hat{\iota}_i : H_1(\partial_i \hat{Y}) \to H_1(\hat{Y})$
for the map induced by inclusion, 
then $\tilde{f}_i \in \hat{\iota}_i^{-1}(f)$ 
denotes the lift of $f$ satisfying 
$(\tilde{f}_i \cdot \tilde{h}_i)|_{\partial_i \hat{Y}} = 1$.
Setting each $h_i := \hat{\iota}_i(\tilde{h}_i) \in H_1(\hat{Y})$,
we note that there must be a relation among the $h_i$,
since the $n+1$-punctured sphere is the same as the $n$-punctured disk, with
first betti number $n$.  In fact, since any one of the $h_i$ can be regarded as
the class of minus the boundary of this disk, with the remaining $h_i$ summing
to a class equal to the boundary of the disk, we have $\sum_{i=0}^n h_i = 0$,
so that $H_1(\hat{Y})$ has presentation
\begin{equation}
H_1(\hat{Y}) = \left< f, h_0, \ldots, h_n \mkern-1mu
\left|\; {\textstyle{\sum_{i=0}^n h_i = 0}}\right. \right>.
\end{equation}

To specify a Seifert fibered space,
one simply lists the Dehn filling slopes,
in terms of the basis $(\tilde{f}_i, -\tilde{h}_i)$
for each $H_1(\partial_i \hat{Y})$,
of the $n+1$ toroidal boundary components of $\hat{Y}$,
conventionally filling $\partial_0 Y$ with an integer slope
and the remaining $\partial_i Y$ with noninteger slopes.
That is, for any $e_0, r_1, \ldots, r_n \in \Z$ and $s_1, \ldots, s_n \in \Z_{\neq 0}$
with each $\frac{r_i}{s_i} \notin \Z$,
the Seifert fibered space 
$M(e_0; {\textstyle{\frac{r_1}{s_{{\mkern-1mu}1}}, \ldots, \frac{r_n}{s_n}}})$
denotes the Dehn filling of $\hat{Y}$ along the slopes
\begin{align}
\mu_0
&:= e_0 \tilde{f}_0 - \tilde{h}_ 0,
     \\ \nonumber
\mu_i
&:= r_i \tilde{f}_i - s_i \tilde{h}_i, \;\; i\in \{1,\ldots, n\}.
\end{align}
The resulting manifold has first homology
\begin{equation}
\label{eq: presentation of H1 for seifert fibered}
H_1\!\left(M(e_0; {\textstyle{\frac{r_1}{s_{{\mkern-1mu}1}}, \ldots, \frac{r_n}{s_n}}}) \right)
=
\left<
f, h_0, \ldots, h_n \left|\;
\textstyle{\sum_{i=0}^n h_i} = \hat{\iota}_0(\mu_0) = \ldots = \hat{\iota}_n(\mu_n) = 0 \right.\right>.
\end{equation}
Note that for any $(z_0, \ldots, z_n) \in \Z^{n+1}$ satisfying $\sum_{i=0}^n z_i=0$, the
change of basis $h_i \mapsto h_i + z_i f$, $i \in \{0, \ldots, n\}$, yields the reparameterization
\begin{equation}
\label{eq: thm 2 trans 1}
M(e_0; {\textstyle{\frac{r_1}{s_{{\mkern-1mu}1}}, \ldots, \frac{r_n}{s_n}}}) \mapsto
M(e_0+ z_0; {\textstyle{\frac{r_1}{s_{{\mkern-1mu}1}}+z_1, \ldots, \frac{r_n}{s_n}+z_n}}).
\end{equation}
In addition, 
$M(e_0; {\textstyle{\frac{r_1}{s_{{\mkern-1mu}1}}, \ldots, \frac{r_n}{s_n}}})$
admits an orientation reversing homeomorphism,
\begin{equation}
\label{eq: thm 2 orientation reversing homeo}
-M(e_0; {\textstyle{\frac{r_1}{s_{{\mkern-1mu}1}}, \ldots, \frac{r_n}{s_n}}}) =
M(-e_0; {\textstyle{-\frac{r_1}{s_{{\mkern-1mu}1}}, \ldots, -\frac{r_n}{s_n}}}).
\end{equation}

\subsection{Statement of L-space classification}

We are now able to state our result.
\begin{theorem}
\label{thm: non L-space interval for seifert fibered spaces}
If  $M(e_0; {\textstyle{\frac{r_1}{s_{{\mkern-1mu}1}}, \ldots, \frac{r_n}{s_n}}})$
denotes a Seifert fibered space over $S^2$ with $n\!>\!0$ exceptional fibers,
then $M(e_0; {\textstyle{\frac{r_1}{s_{{\mkern-1mu}1}}, \ldots, \frac{r_n}{s_n}}})$
is {\em{not}} an L-space if and only if
$e_0 \mkern-1.2mu+\mkern-1.2mu 
\sum_{i=1}^n\mkern-4.5mu \frac{r_i}{s_i} \mkern-2mu=\mkern-1.2mu 0$ or
\begin{equation}
\label{eq: inequality for sfs l-space theorem}
-e_0 + \min_{0<x< s}\! 
-\frac{1}{x}\!\left(\mkern-.8mu -1 + \sum_{i=1}^n \!\left\lceil\! 
\frac{r_i x}{s_i}\! \right\rceil\right)
\;<\; 0 \;<\;
-e_0 
+
\max_{0<x< s}\!
-\frac{1}{x}\!\left(\mkern-.8mu 1+ \sum_{i=1}^n \!\left\lfloor\! 
\frac{r_i x}{s_i}\!\right\rfloor\right),
\end{equation}
where $s$ denotes the least common positive multiple of $s_1, \ldots, s_n$.
\end{theorem}

{\noindent{{\bf{Remark.}} If we take each $s_i > 0$,
then inequality (\ref{eq: inequality for sfs l-space theorem})
is equivalent to the condition that}}
\begin{equation}
\label{eq: orbifold Euler characteristic version of theorem}
\min_{0<x< s}\, \frac{1}{x} \!\left( 1 \,-\, \sum_{i = 1}^{n}\!\frac{[-r_i x]_{s_i}}{s_i} \right)
\;<\;
e_0 + \sum_{i = 1}^{n}\!\frac{r_i}{s_i}
\;<\;
\max_{0<x<s}\, \frac{1}{x} \!\left(-1 + \sum_{i = 1}^{n}\!\frac{[r_i x]_{s_i}}{s_i} \right).
\end{equation}
The middle expression,
$e_0 + \sum_{i=1}^n\!\frac{r_i}{s_i}$, is the orbifold Euler characteristic.
If $e_0 + \sum_{i=1}^n\!\frac{r_i}{s_i} = 0$, then
(\ref{eq: orbifold Euler characteristic version of theorem})
fails to hold when $n \le 2$, in which case all three expressions are equal.

\smallskip
Theorem
\ref{thm: non L-space interval for seifert fibered spaces}
makes it easy to deduce the L-space filling slope interval
for any regular-fiber complement in a Seifert fibered space.
That is, for any $j \in \{1, \ldots, n\}$,
the above theorem implies that 
$M(e_0; {\textstyle{\frac{r_1}{s_{{\mkern-1mu}1}}, \ldots, \frac{r_n}{s_n}}})$
is an L-space if and only if 
\begin{equation}
\label{eq: seifert fibered before removing ceilings and floors}
-e_0x - \left(\mkern-.8mu -1 + \sum_{i\neq j} \!\left\lceil\! 
\frac{r_i x}{s_i}\! \right\rceil \right)
\ge
\left\lceil\! \frac{r_j x}{s_j}\! \right\rceil
\;\;\;\;\mathrm{or}\;\;\;\;
-\mkern-3.1mu e_0 x 
-\left( \mkern-.8mu 1 + \sum_{i\neq j} \!\left\lfloor\! 
\frac{r_i x}{s_i}\!\right\rfloor \right)
\le 
\left\lfloor\! \frac{r_j x}{s_j}\! \right\rfloor
\end{equation}
for all $x \in \{1, \ldots, s-1\}$.
Since the above expressions are integers,
(\ref{eq: seifert fibered before removing ceilings and floors})
holds if and only if
\begin{equation}
-e_0x - \left(\mkern-.8mu -1 + \sum_{i\neq j} \!\left\lceil\! 
\frac{r_i x}{s_i}\! \right\rceil \right)
\ge
\frac{r_j x}{s_j}
\;\;\;\;\;\;\,\mathrm{or}\;\;\;\;\;\,
-\mkern-3.2mu e_0 x 
-\left(\mkern-.8mu 1 + \sum_{i\neq j} \!\left\lfloor\! 
\frac{r_i x}{s_i}\!\right\rfloor \right)
\le 
\frac{r_j x}{s_j}.
\end{equation}
Dividing both sides by $x$ then gives the following result.

\begin{cor}
If  $M(e_0; {\textstyle{\frac{r_1}{s_{{\mkern-1mu}1}}, \ldots, \frac{r_n}{s_n}}})$,
with each $s_i >0$,
denotes a Seifert fibered space over $S^2$ with $n\!>\!1$
exceptional fibers,
then for any $j \in \{1, \ldots, n\}$, 
$M(e_0; {\textstyle{\frac{r_1}{s_{{\mkern-1mu}1}}, \ldots, \frac{r_n}{s_n}}})$
is an L-space if and only if
$e_0 \mkern-1.2mu+\mkern-1.2mu 
\sum_{i=1}^n\mkern-4.5mu \frac{r_i}{s_i} \mkern-2mu \neq \mkern-1.2mu 0$ and
\begin{equation}
\frac{r_j}{s_j}
\le
-e_0 + \min_{0<x< s}\! 
-\frac{1}{x}\!\left(\mkern-1mu -1 + \sum_{i\neq j} \!\left\lceil\! 
\frac{r_i x}{s_i}\! \right\rceil \right)
\,\;\;\mathrm{or}\;\;\,
-\mkern-1.8mu e_0 
+
\max_{0<x< s}\!
-\frac{1}{x}\!\left(\mkern-1mu 1 + \sum_{i\neq j} \!\left\lfloor\! 
\frac{r_i x}{s_i}\!\right\rfloor \right)
\ge \frac{r_j}{s_j},
\end{equation}
where $s$ is the least common multiple of those $s_i$ with $i \neq j$.
\end{cor}

\subsection{Set-up for proof of Theorem
\ref{thm: non L-space interval for seifert fibered spaces}:
Dehn filling a Floer simple manifold}
We begin by expressing
$M(e_0; {\textstyle{\frac{r_1}{s_{{\mkern-1mu}1}}, \ldots, \frac{r_n}{s_n}}})$
as the Dehn filling of a Floer simple manifold $Y$.
For now, we demand that $0 < r_i < s_i$ and $\gcd(r_i, s_i) = 1$
for each $i \in \{1, \ldots, n\}$.
Let $Y$ denote the regular-fiber complement
\begin{equation}
Y := M(0; {\textstyle{\frac{r_1}{s_{{\mkern-1mu}1}}, \ldots, \frac{r_n}{s_n}}}) 
\setminus (S^1 \times D^2_0),
\end{equation}
so that
$Y\mkern-1.5mu(\mu_0) 
= M(e_0; {\textstyle{\frac{r_1}{s_{{\mkern-1mu}1}}, \ldots, \frac{r_n}{s_n}}}) $.  
Regarding $Y$ as a partial Dehn filling of $\hat{Y}$, we have
\begin{equation}
\label{eq: seif fibered minus Y0, presentation of H1(Y)}
H_1(Y)
= \left< f, h_0, \ldots, h_n \left|\;
\textstyle{\sum_{i=0}^n h_i} = \hat{\iota}_1(\mu_1) = \ldots = \hat{\iota}_n(\mu_n) = 0 \right.\right>.
\end{equation}
Writing $\iota_0 : H_1(\partial Y) \to H_1(Y)$ for the map
induced by inclusion,
and identifying
$\tilde{h}_0$ and $\tilde{f}_0$ with their respective images
under the canonical isomorphism
$H_1(\partial_0 \hat{Y}) \to H_1(\partial Y)$,
we again have $\iota_0(\tilde{h}_0) = h_0$
and $\iota_0(\tilde{f}_0) = f$, but in the sense of the above presentation for $H_1(Y)$.

Define
\begin{equation}
S_{\mathrm{gcd}} := 
\gcd\!\left(\frac{{\textstyle{\prod_{i=1}^n \!s_i}}}{s_1}, \ldots, 
\frac{{\textstyle{\prod_{i=1}^n \!s_i}}}{s_n} \right),\;\;\;
s := \frac{{\textstyle{\prod_{i=1}^n \!s_i}}}{S_{\gcd}},
\end{equation}
noting that this makes $s$ the least common multiple of $s_1, \ldots, s_n$.
Note that if we set
\begin{equation}
\label{eq: defs of l, p, q* for seifert fibered}
l: = p \tilde{f}_0 + q^* \tilde{h}_0,\;\;\;\mathrm{with}\;\;\;
p:= \sum_{i=1}^n\frac{r_i}{s_i}\frac{s}{g},\;\;\;\;
q^* := \frac{s}{g},\;\;\;\;
g:= \gcd \!\left( \sum_{i=1}^n \!\frac{r_i}{s_i} s,\, s\right),
\end{equation}
then $l$ is primitive in $H_1(\partial Y)$.  
In addition, since $h_0 = -\sum_{i=1}^n h_i$, we have
\begin{equation}
0
\;=\;
\sum_{i=1}^n\! \frac{s}{s_i}\hat{\iota}_i(\mu_i) 
\;=\;
\sum_{i=1}^n\!\frac{r_i}{s_i}s f \,+\,  s h_0
\;=\; g\iota_0(l).
\end{equation}
Thus $\iota_0(l) \in H_1(Y)$ is torsion, and so $l$ is also a canonical longitude.
Moreover, since $g\iota_0(l) = \sum_{i=1}^n \frac{s}{s_i}\hat{\iota}_i(\mu_i) = 0$
is a primitive linear combination of the relations in the presentation of $H_1(Y)$
in (\ref{eq: seif fibered minus Y0, presentation of H1(Y)}),
we have $g = |\!\left<\iota_0(l)\right>\!|$.
Choosing any $m \in H_1(\partial Y)$ satisfying $m \cdot l = 1$,
and writing $m = -q \tilde{f}_0 - p^* \tilde{h}_0$, allows one to solve for
$\tilde{f}_0$ and $\tilde{h}_0$
in terms of $m$ and $l$.

Now, since all $\frac{r_i}{s_i} > 0$ by assumption, we know from 
Ozsv{\'a}th and Szab{\'o} in ~\cite{OSSF} that
$Y(-h_0) = M(0; {\textstyle{\frac{r_1}{s_{{\mkern-1mu}1}}, \ldots, \frac{r_n}{s_n}}})$
is an L-space, so we may take $\mu_{\textsc{l}} := -\tilde{h}_0$ as our given
L-space filling slope, and choose $\lambda_{\textsc{l}}=  \tilde{f}_0$ for its longitude,
with $\mu_{\textsc{l}} \mkern-1mu \cdot \mkern-1mu \lambda_{\textsc{l}}
=  -\tilde{h}_0 \mkern-1mu \cdot \mkern-1mu \tilde{f}_0 = 1$.  We then have
\begin{equation}
\mu_{\textsc{l}} 
:= -\tilde{h}_0
= p m + q l,
\;\;\;\;\;\;
\lambda_{\textsc{l}} 
:= \tilde{f}_0
= q^* m + p^* l,
\end{equation}
with $p$ and $q^*$ as in
(\ref{eq: defs of l, p, q* for seifert fibered}),
and with $q$ and $p^*$ solving the diophantine equation $pp^* - q q^* = 1$.

\subsection{Computation of $\boldsymbol{\dt}(\boldsymbol{Y})$}
To compute $\dt(Y)$, we need the Turaev torsion, $\tau(Y)$.
Recall that $Y$ is a union along torus boundaries of trivial circle fibrations,
\begin{equation}
Y = S^1 \mkern-.7mu\times\mkern-.7mu (S^2 \setminus {\textstyle{\coprod_{i=0}^n D_i^2}}) 
\;\mkern1.5mu \cup\;\mkern1.5mu
S^1_1 \mkern-.7mu\times\mkern-.7mu D_1^2
\mkern1.5mu\;\cup\;\mkern1.5mu \ldots \mkern1.5mu\;\cup\;\mkern1.5mu
S^1_n \mkern-.7mu\times\mkern-.7mu D_n^2.
\end{equation}
The leftmost $S^1$ above, corresponding to the regular fiber in $\hat{Y}$, has class
$\hat{\iota}_0(\lambda_{\textsc{l}}) = f \in H_1(\hat{Y})$.
Similarly, for $i \in \{1, \ldots, n\}$,
each $S^1_i$ above has class $\hat{\iota}_i(\lambda_i) \in H_1(\hat{Y})$,
where $\lambda_i$ is any longitude satisfying 
$\left.(\mu_i \cdot \lambda_i)\right|_{\partial_i\mkern-1.5mu\hat{Y}} = 1$.
Since each $\hat{\iota}_i(\mu_i) = 0$, each class $\hat{\iota}_i(\lambda_i)$ is independent of the
choice of $\lambda_i$.  The Turaev torsion then obeys a product rule for
unions along torus boundaries \cite{Turaev}, yielding
\begin{align}
\tau(Y) 
&:= (1 - [\hat{\iota}_0(\lambda_{\textsc{l}})])^{-\chi(S^2 \setminus \coprod_{i=0}^n D_i^2)}
{\textstyle{\prod_{i=1}^n (1-[\hat{\iota}_i(\lambda_i)])^{-\chi(D_i^2)}}}
     \\ \nonumber
&:= (1 - [\hat{\iota}_0(\lambda_{\textsc{l}})])^{n-1}
{\textstyle{\prod_{i=1}^n (1-[\hat{\iota}_i(\lambda_i)])^{-1}}},
\end{align}
where $[\cdot]$ denotes inclusion of $H_1(\hat{Y})$
into the Laurent series group ring for $H_1(\hat{Y})$.

These $\hat{\iota}_i(\lambda_i)$ bear simple relationships to $\iota_0(\mu_{\textsc{l}})$ and
$\iota_0(\lambda_{\textsc{l}})$.  That is, we claim that
\begin{equation} 
\label{eq: mu and lambda in terms of lamba i}
\iota_0(\mu_{\textsc{l}})
= {\textstyle{\sum_{i=1}^n r_i \hat{\iota}_i(\lambda_i)}},
\;\;\mathrm{and}\;\;
\iota_0(\lambda_{\textsc{l}}) = s_i \hat{\iota}_i(\lambda_i)
\;\;\mathrm{for}\;\mathrm{each}\;\; i \in \{1, \ldots, n\}.
\end{equation}
To see this, note that since
each $\mu_i = r_i \tilde{f}_i - s_i \tilde{h}_i$,
with
$(\tilde{f}_i \cdot \tilde{h}_i)|_{\partial_i\mkern-1.5mu\hat{Y}}
= \left.(\mu_i \cdot \lambda_i)\right|_{\partial_i\mkern-1.5mu\hat{Y}} = 1$,
we know there exist $r^*_i, s^*_i \in \Z$ such that
\begin{equation}
\hat{\iota}_i(\lambda_i) = s_i^* f + r_i^* h_i,\;\;\; r_i r_i^* + s_i s_i^* = 1,
\end{equation}
implying that
\begin{align}
r_i \hat{\iota}_i(\lambda_i)
= s_i^* (r_i f) + r_i r_i^* h_i 
&= s_i^* (s_i h_i) + r_i r_i^* h_i 
= h_i,
        \\
s_i \hat{\iota}_i(\lambda_i)
= s_i s_i^* f + r_i^* (s_i h_i) 
&= s_i s_i^* f \,+\, r_i^* (r_i f) 
\;\mkern.6mu= f 
= \iota_0(\lambda_{\textsc{l}}).
\end{align}
Thus, since $\iota_0(\mu_{\textsc{l}}) = -h_0 = {\textstyle{\sum_{i=1}^n h_i}}$,
(\ref{eq: mu and lambda in terms of lamba i}) holds in $H_1(Y)$.

Since $\iota_0(\lambda_{\textsc{l}}) = s_i \hat{\iota}_i(\lambda_i)$ for each
$i \in \{1, \ldots, n\}$, we may rewrite $\tau(Y)$ as
\begin{equation}
\tau(Y) = 
\frac{1}{1 - [\iota_0(\lambda_{\textsc{l}})]}
\prod_{i=1}^n \frac{1 - [\hat{\iota}_i (\lambda_i)]^{s_i}\!}{1 - [\hat{\iota}_i(\lambda_i)]\hphantom{^{s_i}\!}},
\end{equation}
which has support
\begin{equation}
\label{eq: tt support for SFS}
S[\tau(Y)] = 
\{\iota_0(\lambda_{\textsc{l}}) \Z_{\ge 0} \} +
\left\{ \left.{\textstyle{\sum_{i=1}^ny_i\hat{\iota}_i(\lambda_i)}}
\mkern.7mu\right|\mkern.6mu
y_i \in \{0, \ldots, s_i - 1\}
\right\}.
\end{equation}
Since $Y$ has multiple L-space fillings, it is Floer simple, and
so each element of $H_1(Y)$ has coefficient 0 or 1 in $\tau(Y)$,
and the torsion complement $\tc(Y)$ has support
\begin{equation}
\label{eq: shows that stc - st = stc}
S[\tc(Y)] 
=
\left\{ \left.\mkern-.5mu- \mkern-1.5muj\iota_0(\lambda_{\textsc{l}}) \mkern-2.3mu
+\mkern-3.6mu  {\textstyle{\sum_{i=1}^n}} y_i\hat{\iota}_i(\lambda_i) 
\mkern.7mu\right|\mkern.6mu
j \!\in\! \{1, \ldots, n\!-\!1\},\mkern1mu
y_i \!\in \!\{0, \ldots, s_i \!-\! 1\}  \right\}
\cap
H_1\mkern-1.3mu(Y\mkern-.3mu)_{\mkern-1.7mu\ge \mkern.2mu0},
\end{equation}
where $H_1(Y)_{\ge 0} := \{w \in H_1(Y)| \phi(w) \ge 0\}$
for any homomorphism $\phi : H_1(Y) \to \Z$ satifying $\phi(\iota_0(m)) > 0$.

Since $s_i \hat{\iota}_i(\lambda_i) = \iota_0(\lambda_{\textsc{l}})$ for each
$i \in \{1, \ldots, n\}$, it follows from 
(\ref{eq: tt support for SFS}) that $S[\tau(Y)]$
is additively closed, which, in  turn, implies that
\begin{equation}
(S[\tc(Y)] - S[\tau(Y)]) \cap 
H_1(Y)_{\ge 0}
= S[\tc(Y)],
\end{equation}
so that $\dt(Y)$ is the intersection
$\dt(Y)
= S[\tc(Y)] \cap \iota_0(H_1(\partial Y))$.
By (\ref{eq: mu and lambda in terms of lamba i}), we know that
\begin{align}
\iota_0(H_1(\partial Y))
&= \Span \!\!\left\{ \iota_0(\mu_{\textsc{l}}) , \iota_0(\lambda_{\textsc{l}}) \right\}
        \\ \nonumber
&= \Span \!\!\left\{  \iota_0(\mu_{\textsc{l}})
={\textstyle{\sum_{i=1}^n r_i \hat{\iota}_i(\lambda_i)}},\;\,
\iota_0(\lambda_{\textsc{l}}) 
= s_1 \hat{\iota}_1(\lambda_1) 
= \ldots 
= s_n \hat{\iota}_n(\lambda_n)\right\}.
\end{align}
Now, for any $j \in \{1,\ldots, n-1\}$ and
$(y_1, \ldots, y_n) \in \prod_{i=1}^n \{0, \ldots, s_i - 1\}$, we have
\begin{equation}
- j\iota_0(\lambda_{\textsc{l}})
+ {\textstyle{\sum_{i=1}^n y_i\hat{\iota}_i(\lambda_i)}}
 = 
{\textstyle{\sum_{i=1}^n (y_i + z_is_i)\hat{\iota}_i(\lambda_i)}}
\mkern1mu-\mkern1mu
(j + {\textstyle{\sum_{i=1}^n \mkern-1.5mu z_i}})\iota_0(\lambda_{\textsc{l}})
\end{equation}
for any $(z_1, \ldots, z_n) \in \Z^n$.  Thus,
$- j\iota_0(\lambda_{\textsc{l}})  + \sum_{i=1}^n\mkern-1.5mu y_i\hat{\iota}_i(\lambda_i)
\in \iota_0(H_1(\partial Y))$ if and only if
there exist $(z_1, \ldots, z_n) \in \Z^n$ and $x \in \Z$ for which
\begin{equation}
(y_1 + z_1 s_1 ,\ldots, y_n + z_n s_n )
= (r_1 x, \ldots, r_n x).
\end{equation}
In such case, we have
$y_i = [r_i x]_{s_i}$
and
$z_i = \mkern-2.5mu\left\lfloor\! \frac{r_i x}{s_i}\! \right\rfloor$
for each $i \in \{1, \ldots, n\}$.

We can therefore parameterize
$\dt(Y) = S[\tc(Y)]\cap \iota_0(H_1(\partial Y))$ as
\begin{align}
\dt(Y)
= \left\{ \boldsymbol{\delta}^j_x\mkern1.5mu|\; 
j \in \{1, \ldots, n-1\},\; x \in \{1, \ldots, s-1\},\;
\delta^j_x  \ge 0 \right\},\;\,\mathrm{with}\;\;\;\;\;\;\;\;\;\;\;\;\;\;\;
    \\ \nonumber
\boldsymbol{\delta}^j_x := 
a^{j-}_x \iota_0(\mu_{\textsc{l}}) + b^{j-}_x \iota_0(\lambda_{\textsc{l}}),
\;\;\;\;\;\;\;
a^{j-}_x := x,
\;\;\;\;\;
b^{j-}_x := - j - \sum_{i=1}^n \!\left\lfloor\! \frac{r_i x}{s_i}\! \right\rfloor,
\;\;\;
    \\ \nonumber
\delta^j_x 
:= 
a^{j-}_x p + b^{j-}_x q^*
= \frac{s}{g}\!\left(\mkern-1mu -j + \sum_{i=1}^n\frac{[r_ix]_{s_i}}{s_i}\right),
\;\;\;\;\;\;\;\;\;\;\;\;\;\;\;\;\;\;\;\;\;\;\;\;\;\;
\;\;\;\;\;\;\;\;\;\;\;\;\;\;\mkern3.1mu
\end{align}
where $\delta^j_x := \tilde{\boldsymbol{\delta}}^j_x \!\cdot\mkern-1mu l$ for any
$\tilde{\boldsymbol{\delta}}^j_x \in \iota_0^{-1}\mkern-2.5mu(\boldsymbol{\delta}^j_x)$.
Since $\boldsymbol{\delta}^j_x$ is invariant under the action $x \mapsto x + s$,
it suffices to choose a fundamental domain of length $s$ for $x \in \Z$.
The above expression for $\dt(Y)$ uses the fundamental domain
$x \in \{0, \ldots, s-1\}$, but excludes $0$, since $\delta^j_0 < 0$
for all $j \in \{1, \ldots, n-1\}$.

\subsection{Application of
Theorem 
\ref{thm:  L-space interval in terms of beta/n}/Corollary
\ref{cor: L-space criterion in surgery coefficients}}
This particular choice of 
fundamental domain ensures that for all $\boldsymbol{\delta}_x^j \in \dtgz(Y)$,
we have
$b_x^{j-} = b_-^{\boldsymbol{\delta}^j_x}$ and
$a_x^{j-} = a_-^{\boldsymbol{\delta}^j_x}$ 
in the sense of 
Corollary \ref{cor: L-space criterion in surgery coefficients}.
That is,
for all $j \in \{1, \ldots, n-1\}$ and 
$x\in \{1, \ldots, s-1\}$ with $\delta^j_x > 0$,
we have
\begin{equation}
0 
\;\,<\;\,
-b^{j-}_x
\,\;=
\frac{a^{j-}_x p}{q^*} - \frac{\delta^j_x}{q^*} 
=
\sum_{i=1}^n \frac{r_ix}{s_i} - \frac{g}{s}\delta^j_x
\,<\, 
\sum_{i=1}^n \frac{r_is}{s_i} - 0
=\;\,\mkern.5mu
pg.
\end{equation}
This makes
$a_x^{j-} \mu_{\textsc{l}} + b_x^{j-} \lambda_{\textsc{l}} \in 
\iota_0^{-1}\mkern-1mu(\boldsymbol{\delta}^j_x)$
one of the two lifts of $\boldsymbol{\delta}^j_x$ closest to $\mu_{\textsc{l}}$
in $\P(H_1(\partial Y))$, and the closest lift of $\boldsymbol{\delta}^j_x$
on the other side of $\mu_{\textsc{l}}$ is
$a_x^{j+}\mkern-1.8mu \mu_{\textsc{l}} \mkern.9mu +
\mkern1.1mu b_x^{j+} \mkern-1.2mu\lambda_{\textsc{l}} \in 
\iota_0^{-1}\mkern-1mu(\boldsymbol{\delta}^j_x)$,
where
\begin{align}
a_x^{j+}
&:=\, a_+^{\boldsymbol{\delta}^j_x} 
\,=\, a_x^{j-} - q^{\mkern-1mu*}\mkern-2mug
\,=\, -(s-x),
   \\ \nonumber
b_x^{j+} 
&:=\, b_+^{\boldsymbol{\delta}^j_x} 
\mkern1.5mu\,=\, b_x^{j-} + pg
\mkern7.3mu\,=\,
-j + \sum_{i=1}^n\left\lceil \frac{r_i(s-x)}{s_i} \right\rceil.
\end{align}

To use Corollary \ref{cor: L-space criterion in surgery coefficients} on
$M(e_0; {\textstyle{\frac{r_1}{s_{{\mkern-1mu}1}}, \ldots, \frac{r_n}{s_n}}}) 
= Y\mkern-2mu(\mu_0)$, we shall also want the
$(\mu_{\textsc{l}}, \lambda_{\textsc{l}})$-surgery coefficients
for $\mu_0$, and the value of $\mu_0 \cdot l$.
Since $\mu_0 = e_0 \tilde{f}_0 - \tilde{h}_0$ and $l = p\tilde{f}_0 + q^*\tilde{h}_0$, 
with $\mu_{\textsc{l}} = -\tilde{h}_0$, $\lambda_{\textsc{l}} = \tilde{f}_0$,
$p = \frac{s}{g}\sum_{i=1}^n\!\frac{r_i}{s_i}$, and $q^* = \frac{s}{g}$, we have
\begin{align}
\mu_0 
&= \alpha \mu_{\textsc{l}} + \beta \lambda_{\textsc{l}},
\;\;\;\;\;\;\;\;\;\;\;\;
\;\;\;\;\;\;\;\;\;\;\;\;
\;\;\;\;\;\;\;\;\;\;\;\;
\alpha:= 1,\;\;\;\;
\beta:= e_0,
    \\  \nonumber
\mu_0 \cdot l
&= e_0 q^* + p = \frac{s}{g}\!\left(e_0 + \sum_{i=1}^n\frac{r_i}{s_i}\right).
\end{align}
Since $Y\mkern-2mu(\mu_0)$ is never an L-space when 
$\mu_0 \cdot l = 0$, and since the case of $e_0 + \sum_{i=1}^n\frac{r_ix}{s_i} =0$
is treated separately in the theorem statement, we henceforth restrict to the
case of $\mu_0 \cdot l \neq 0$.

Suppose that $\dtgz(Y) \neq \emptyset$.
In this case,
Corollary \ref{cor: L-space criterion in surgery coefficients}
tells us that
$M(e_0; {\textstyle{\frac{r_1}{s_{{\mkern-1mu}1}}, \ldots, \frac{r_n}{s_n}}}) 
= Y\mkern-2mu(\mu_0)$
is an L-space if and only if
\begin{equation}
\label{eq: L-space condition for all SF spaces with each 0 < r_i/s_i < 0}
\frac{\alpha}{\beta}:=
\frac{1}{e_0} 
\;\le\;
\frac{x}{b^{j-}_x}
=: \frac{a^{j-}_x}{b^{j-}_x}
\;\;\;\;\;\;\mathrm{or}\;\;\;\;\;\;
\frac{a^{j+}_x}{b^{j+}_x} :=
\frac{-(s-x)}{b^{j+}_x}
\;\le\;
\frac{1}{e_0}
=:\frac{\alpha}{\beta}
\end{equation}
for all $j\in \{1, \ldots, n-1\}$
and $x \in \{1, \ldots, s-1\}$ with $\delta^j_x > 0$,
and moreover
the left-hand (respectively right-hand) inequality
obtains only if $\beta/(\mu_0 \mkern-1.5mu\cdot\mkern-1.5mu l) <0$ 
(respectively $\beta/(\mu_0 \mkern-1.5mu\cdot\mkern-1.5mu l) > 0$).

Further suppose that $\beta = e_0 < 0$.
Then 
$M(e_0; {\textstyle{\frac{r_1}{s_{{\mkern-1mu}1}}, \ldots, \frac{r_n}{s_n}}})$
is an L-space if and only if
\begin{equation}
\label{eq: sf l-space ineqs but interdependent and with deltax > 0}
\begin{cases}
0 \ge -e_0 + b_x^{j-}/x 
\,\;\;\;\;\;\;\;\;\mkern3mu
\text{for all}\;j\;\text{and}\;x\;\text{with}\;\delta^j_x > 0
&\;\;\;\;\;
\mathrm{if}\;\mu_0 \mkern-1.5mu\cdot\mkern-1.5mu l > 0
            \\
0 \le
-e_0 - b^{j+}_x/(s\!-\!x)
\,\;
\text{for all}\;j\;\text{and}\;x\;\text{with}\;\delta^j_x > 0
&\;\;\;\;\;
\mathrm{if}\;\mu_0 \mkern-1.5mu\cdot\mkern-1.5mu l < 0
            \\
\text{never (case already excluded)}
&\;\;\;\;\;
\mathrm{if}\;\mu_0 \mkern-1.5mu\cdot\mkern-1.5mu l = 0
\end{cases}.
\end{equation}
Note that for all
$j \in \{1, \ldots, n-1\}$ and
$x \in \{1, \ldots, s-1\}$,
$\delta^j_x = a^{j\pm}_x p + b^{j\pm}_x q^*$ implies
\begin{equation}
b^{j-}_x/x = \delta^j_x/(q^*x) - p/q^*,
\;\;\;\;\;
-b^{j+}_x/(s\mkern-2mu -\mkern-2mu x) 
= -\delta^j_x/(q^*(s\mkern-2mu -\mkern-2mu x)) - p/q^*.
\end{equation}
Thus $b_x^{j-}/x$ is never maximized
and $-b_x^{j+}/(s-x)$ is never minimized when
$\delta_x \le 0$,
so we can remove the
$\delta^j_x > 0$ conditions from
(\ref{eq: sf l-space ineqs but interdependent and with deltax > 0}).
Moreover, $b_x^{j-}/x$ is never maximized
and $-b_x^{j+}/(s-x)$ is never minimized when
$j > 1$, so it suffices to fix $j=1$.
Reparameterizing the second case of  
(\ref{eq: sf l-space ineqs but interdependent and with deltax > 0})
by $s-x \mapsto x$ 
then transforms (\ref{eq: sf l-space ineqs but interdependent and with deltax > 0})
into the condition
\begin{equation}
\label{eq: primitive version of thm statement for sf L-spaces}
0 \le
-e_0 + \min_{0 < x < s} 
-\frac{1}{x}\!
\left(\mkern-.8mu-1+\sum_{i=1}^n \!\left\lceil\! \frac{r_i x}{s_i}\! \right\rceil \right)
\;\;\;\mathrm{or}\;\;\;
-e_0 + \max_{0 < x < s}
-\frac{1}{x}\!
\left(\mkern-.8mu 1+ \sum_{i=1}^n \!\left\lfloor\! \frac{r_i x}{s_i}\! \right\rfloor\right)
\le 0,
\end{equation}
which is the negation of the theorem statement's inequality for non-L-spaces.

When $e_0 \ge 0$,
$M(e_0; {\textstyle{\frac{r_1}{s_{{\mkern-1mu}1}}, \ldots, \frac{r_n}{s_n}}})$
is always an L-space, since $e_0 = 0$ corresponds to our initial L-space
$Y\mkern-1.5mu(\mu_{\textsc{l}})$, and since when $e_0 < 0$,
the right-hand inequality in 
(\ref{eq: L-space condition for all SF spaces with each 0 < r_i/s_i < 0})
holds for all $j \in \{1, \ldots, n-1\}$ and $x \in \{1, \ldots, s-1\}$.
Accordingly,
when $e \ge 0$,
(\ref{eq: primitive version of thm statement for sf L-spaces})
always holds (via its right-hand inequality).

Lastly, suppose that
$\dtgz(Y) = \emptyset$.
Since we have excluded the case of $\mu_0 \cdot l = 0$,
this implies that 
$Y(\mu_0) = M(e_0; {\textstyle{\frac{r_1}{s_{{\mkern-1mu}1}}, \ldots, \frac{r_n}{s_n}}})$
is an L-space, so we must show that
(\ref{eq: primitive version of thm statement for sf L-spaces}) holds.
To see this, first note that the negation of 
(\ref{eq: primitive version of thm statement for sf L-spaces})
is equivalent to the inequality
\begin{equation}
\label{eq: second use of orbifold euler characteristic version}
\min_{0<x< s}\, \frac{1}{x} \!\left( 1 \,-\, \sum_{i = 1}^{n}\!\frac{[-r_i x]_{s_i}}{s_i} \right)
\;<\;
e_0 + \sum_{i = 1}^{n}\!\frac{r_i}{s_i}
\;<\;
\max_{0<x<s}\, \frac{1}{x} \!\left(-1 + \sum_{i = 1}^{n}\!\frac{[r_i x]_{s_i}}{s_i} \right).
\end{equation}
Since $\dtgz(Y) = \emptyset$ implies $\delta^j_x \le 0$ for all 
$j \in \{1, \ldots, n-1\}$ and $x \in \{1, \ldots, s-1\}$, we have
\begin{equation}
1 - \sum_{i = 1}^{n}\!\frac{[-r_i x]_{s_i}}{s_i}
= -\delta^{j=1}_{s-x} \ge 0,
\;\;\;\;\;\;\;
-1 + \sum_{i = 1}^{n}\!\frac{[r_i x]_{s_i}}{s_i} = \delta^{j=1}_x \le 0
\end{equation}
for all $x \in \{1, \ldots, s-1\}$.
Thus
(\ref{eq: second use of orbifold euler characteristic version})
fails and
(\ref{eq: primitive version of thm statement for sf L-spaces})
holds.

We have finished showing that, when
$0 < r_i < s_i$ and $\gcd(r_i, s_i) = 1$ for each $i \in \{1, \ldots, n\}$,
$M(e_0; {\textstyle{\frac{r_1}{s_{{\mkern-1mu}1}}, \ldots, \frac{r_n}{s_n}}})$
is an L-space if and only if
$e_0 + \sum_{i=1}^n\frac{r_i}{s_i} \neq 0$ and
(\ref{eq: primitive version of thm statement for sf L-spaces})
holds.  Moreover, since
(\ref{eq: primitive version of thm statement for sf L-spaces})
is invariant under any map $\frac{r_i}{s_i} \mapsto \frac{d r_i}{d s_i}$
with $d \in \Z_{\neq 0}$,
or under any reparameterization of the type in
(\ref{eq: thm 2 trans 1}), we can remove our initial restrictions that
$0 < r_i < s_i$ and $\gcd(r_i, s_i) = 1$,
completing the proof of the theorem.


\smallskip
\section{Gluings along torus boundaries}
\label{s: L-space if and only if intervals cover.}

The introduction to Section \ref{s: Seifert fibered L-spaces}
discusses how, for Seifert fibered spaces over $S^2$ (although the same is true
for all Seifert fibered spaces \cite{BGW, Gabai}), the property of admitting a 
cooriented taut foliation is equivalent to the property of not being an L-space.

\subsection{Equivalent properties for Seifert fibered spaces}
In fact, this pair of equivalent properties belongs to a larger list.
\begin{theorem}[\cite{EHN, OSGen, LSIII, BRW}]
Suppose $M$ is a Seifert fibered space over $S^2$.
Then the following are equivalent:
\begin{enumerate}
\item[(1)]
$M$ admits a cooriented taut foliation.
\item[(2.$\rho$)]
There exists a homomorphism
$\rho : \pi_1(M) \to \mathrm{Homeo}_+\R$ with non-trivial image.
\item[(2.LO)]
The fundamental group $\pi_1(M)$ admits a left ordering.
\item[(3)]
$M$ is not an L-space.
\end{enumerate}
\end{theorem}
\begin{proof}[Summary of Proof]
Our idiosyncratic numbering owes to a result of
Boyer, Rolfsen, and Wiest \cite{BRW}, which implies that
$(2.\rho) = (2.\mathrm{LO})$ for (a superset of) all closed, prime, oriented three-manifolds.
We also have $(1) \Rightarrow (3)$ for all closed oriented three-manifolds, as shown by
Ozsv{\'a}th and Szab{\'o}  in the case of $C^2$ foliations \cite{OSGen},
a result recently extended to $C^0$ foliations 
by Kazez and Roberts \cite{KazezRobertsCzero}, 
and independently by Bowden \cite{Bowden}.

More is known for Seifert fibered spaces.
For Seifert fibrations over $S^2$, we have $(1) = (2)$ as a corollary of a result by
Eisenbud, Hirsh, and Neumann \cite{EHN}.
The result that $(3) \Rightarrow (1)$
is due to
Lisca, Mati{\'c}, and Stipsicz  for fibrations over $S^2$ \cite{LiscaMatic, LSIII},
Boyer, Gordon, and Watson  for fibrations over $\R\P^2$ \cite{BGW},
and Gabai  for fibrations with positive first betti number \cite{Gabai}.
One could also regard the classification
by Jankins, Neumann \cite{JankinsNeumann}, and Naimi \cite{Naimi}
of Seifert fibered spaces over $S^2$ satisfying (1),
together with the classification in
the present article's
Theorem 
\ref{thm: non L-space interval for seifert fibered spaces}
of Seifert fibered L-spaces over $S^2$, as an alternative proof that
$(1) = (3)$.
\end{proof}
{\noindent{The above result motivated
a conjecture of
Boyer, Gordon, and Watson \cite{BGW} that
properties (2) and (3) above are equivalent for all closed, prime, oriented
three-manifolds.}}

\subsection{Gluing results}

To further explore the relationship of the above properties,
Boyer and Clay \cite{BoyerClay} 
studied how each of these properties glue together
when one splices together Seifert fibered spaces along the toroidal boundaries of
fiber complements
to form a graph manifold.  In the process, Boyer and Clay observed that
properties (1) and (2) obey a similar criterion determining
when they admit compatible gluings.  The property (3) of being a non-L-space
proved less tractable for this exercise, but Boyer and Clay conjectured
that property (3) should follow a similar gluing pattern
to that of (1) and (2).

We are now able to confirm their conjecture in the case in which
two Floer simple manifolds glued along their torus boundaries
have the interiors of their L-space intervals overlap via the gluing map.
In fact, there is no requirement that these Floer simple manifolds be graph manifolds.

\begin{theorem}
\label{thm: torus gluing and complementary intervals for L-spaces}
Suppose that $Y_1$ and $Y_2$
are Floer simple manifolds glued together along their boundary tori.
Such gluing is specified by a linear map 
$\varphi : H_1 (\partial Y_1) \to H_1(\partial Y_2)$
with $\det \varphi = -1$, descending to a map
$\varphi_{\P} : \P(H_1 (\partial Y_1)) \to \P(H_1(\partial Y_2))$
on Dehn filling slopes.
Let $I_i \subset  \P(H_1 (\partial Y_i)) $ denote the
interval (with interior $\dot{I}_i$) of L-space filling slopes for $Y_i$, for 
each $i \in \{1,2\}$, and suppose that 
$\varphi_{\P}(\dot{I}_1) \cap \dot{I}_2$ is nonempty.
Then $Y_1\cup_{\varphi}\!Y_2$ is an L-space if and only if
$\varphi_{\P}(\dot{I}_1) \cup \dot{I}_2 =  \P(H_1(\partial Y_2))$ if
both $\dtge(Y_i)$ are nonempty, and if and only if
$\varphi_{\P}(I_1) \cup I_2 =  \P(H_1(\partial Y_2))$ otherwise.
\end{theorem}

\subsection{Set-up for proof: Conventions and simplifying assumptions}
We begin by choosing
bases $(m_i, l_i)$ for $H_1(\partial Y_i)$
and $\bar{m}_i$ for $H_1(Y_i)/\Tors(Y_i)$,
for each $i \in \{1,2\}$,
according to the conventions of Section \ref{ss: slope and basis conventions}.
Thus, if we write $\iota_i : H_1(\partial Y_i) \to H_1(Y_i)$ for the map induced on
homology by inclusion of the boundary, then
$l_i$ generates $\iota^{-1}(T_i)$, where $T_i := \Tors(H_1(Y_i))$,
$m_i$ satisfies $m_i \cdot l_i = 1$,
and $\bar{m}_i$ satisfies
$\iota_i(m_i) \in  g_i \bar{m}_i + T$,
where $g_i := |T_i^{\partial}|$, with
$T_i^{\partial} := \iota(\left<l_i\right>) = T_i \cap \iota(H_1(\partial Y_i))$.

We shall break the operation of  torus boundary gluing
into three steps more amenable to Heegaard Floer computation: those of
Dehn filling, connected sum, and Dehn surgery.
In preparation, assuming $\varphi_{\P}(\dot{I}_1) \cap \dot{I}_2$ nonempty,
choose $\mu_1\!  \in \P^{-1}\mkern-2mu(\dot{I}_1\mkern-2mu 
\cap \varphi^{-1}_{\P}\mkern-1.5mu(\dot{I}_2)) \subset H_1(\partial Y_1)$
and a longitude $\lambda_1 \in H_1(\partial Y_1)$
satisfying $\mu_1 \cdot \lambda_1 = 1$.
Set $\mu_2 := \varphi(\mu_1)$ and $\lambda_2:= -\varphi(\lambda_1) \in H_1(\partial Y_2)$,
noting that this makes $\lambda_2$ a longitude relative to $\mu_2$, since
$\mu_1 \cdot \lambda_1 = 1$ and $\det \varphi = -1$ imply $\mu_2 \cdot \lambda_2 = 1$.
Write $\mu_i = p_i m_i + q_i l_i$ and 
$\lambda_i = q^*_i m_i + p^*_i l_i$, with $q_i q^*_i - p_i p^*_i = 1$,
for each $i \in \{1,2\}$.
Note that the invariant $q^*:= q_1^* p_2 + q_2^* p_1$ is 
independent of choices of $\mu_1$ and $\lambda_1$.
That is, if we write $(\phi_{ij})$ for the entries of the matrix for $\varphi$
with respect to the bases $(m_1, l_1)$ and $(m_2, l_2)$, then
\begin{equation}
q^* \;= p_2 q^*_1 + q^*_2 p_1
= (\phi_{11}p_1 + \phi_{12}q_1)q_1^*
- (\phi_{11}q^*_1 + \phi_{12}p^*_1)p_1
=\; -\phi_{12}.
\end{equation}
Before using $\mu_i$ and $\lambda_i$ to splice together $Y_1$ and $Y_2$,
we first pause to make some simplifying assumptions, without
loss of generality.

\begin{prop}
\label{prop: judiciously chosen def}
Suppose $\varphi_{\P}(\dot{I}_1) \cap \dot{I}_2 \neq \emptyset$.
For purposes of proving Theorem 
\ref{thm: torus gluing and complementary intervals for L-spaces},
it is sufficient to take $q^* > 0$, and we may choose
$\mu_1\!  \in \P^{-1}\mkern-2mu(\dot{I}_1\mkern-2mu 
\cap \varphi^{-1}_{\P}\mkern-1.5mu(\dot{I}_2)) \subset H_1(\partial Y_1)$ to satisfy
$\gcd(p_i, q_i) = \gcd(p_1, p_2) = \gcd(p_1, g_2) = \gcd(p_2, g_1)= 1$,
$p_1, p_2 > q^*>0$, and 
$p_i > (1+\deg_{[\bar{m}_1]}\!\tc(Y_1))(1+\deg_{[\bar{m}_2]}\!\tc(Y_2))$
for $i \in \{1,2\}$,
where $p_i m_i + q_i l_i = \mu_i$, 
$q_i^* m_i + p_i^* l_i = \lambda_i$, $\mu_2 := \varphi(\mu_1)$,
$\lambda_2 := -\varphi(\lambda_1)$, and $q^* := q_1^* p_2 + q_2^* p_1$
for $i \in \{1,2\}$.  We call such $\mu_1$ {\em{``judiciously chosen.''}}
\end{prop}

\begin{proof}
We summarily dispense with the case in which 
$q^* = 0$, since then
$\varphi_{\P}(\dot{I}_1) \cup \dot{I}_2 \neq \P(H_1(\partial Y_2))$
and $Y_1 \cup_{\varphi} Y_2$ is not a rational homology sphere,
hence not an L-space.
If $q^* <0$, then we may send $q^*$ to $-q^*$
by making the changes of basis $(m_i, l_i) \mapsto (m_i, -l_i)$
while simultaneously reversing the orientations of both $Y_1$ and $Y_2$.
This preserves the positivity of $p_1$ and $p_2$, and
leaves invariant the questions of whether $Y_1 \cup_{\varphi} Y_2$
is an L-space and whether $\varphi_{\P}(\dot{I}_1) \cap \dot{I}_2 = \P(H_1(\partial Y_2))$,
or $\varphi_{\P}(I_1) \cap I_2 = \P(H_1(\partial Y_2))$.
Thus we henceforth take $q^* > 0$.

We can construct a judicious choice of $\mu_1$ as an approximation 
of a primitive representative
$P_1 m_1 + Q_1 l_1 \in \P^{-1}( \dot{I}_1 \cap \varphi^{-1}_{\P}\mkern-1.2mu(\dot{I}_2))$
with $P_1 > 0$.
Since $\dot{I}_1 \cap \varphi^{-1}_{\P}\mkern-1.2mu(\dot{I}_2)$ contains an open ball,
we can demand that $P_i$ and $Q_i$ are nonzero for $i \in \{1,2\}$,
where $P_2 m_2 + Q_2 l_2 = \varphi(P_1 m_1 + Q_1 l_1)$.
If $P_2 < 0$, 
we repair this sign with the change of basis $(m_2, l_2) \mapsto (-m_2, -l_2)$.
Writing $M_{\varphi} = (\phi_{ij})$ for the matrix for 
$\varphi$ with respect to the bases $(m_1, l_1)$ and $(m_2, l_2)$,
choose $s \in \Z$ such that 
$x := \phi_{22} + \phi_{12}s$ and $y:= -\phi_{21} -\phi_{11}s$
are nonzero, with $\gcd(x,g_2) = 1$, noting that we now have
$M_{\varphi}(x,y)^{\top} = (-1, s)^{\top}$.
Next, set
\begin{equation}
D: = \left|  g_1g_2xy(yP_1 - xQ_1)(P_1+xP_2) \right|,
\end{equation}
and define $\mu_1 := p_1 m_1 + q_1 l_1$ and $\mu_2 := p_2 m_2 + q_2 l_2 = \varphi(\mu_1)$, with
\begin{align}
\,
p_1 := P_1 D N + x,
\;\;\;\;\;\;\;\;\;\;\;\, 
p_2 := P_2 DN - 1,
        \\ \nonumber
q_1:= Q_1 DN + y,
\;\;\;\;\;\;\;\;\;\;\;\,
q_2 := Q_2 DN + s
\,
\end{align}
for some integer 
$N  > q^* (1+\deg_{[\bar{m}_1]}\!\tc(Y_1))(1+\deg_{[\bar{m}_2]}\!\tc(Y_2))$ chosen large enough
to make $\mu_1 := p_1 m_1 + q_1 l_1 $ lie in 
$\P^{-1}(\dot{I}_1 \cap \varphi^{-1}_{\P}\mkern-1.2mu(\dot{I}_2))$.
Then $\gcd(p_1,g_2) = \gcd(p_2, g_1) =1$, and one can use the facts that
$p_1/x - q_1/y = (yP_1 - xQ_1)(D/(xy))N$ is relatively prime
to $p_1/x$ and that 
$p_1/x + p_2 = (P_1 + xP_2)(D/x)N$ is relatively prime to $p_2$
to argue, respectively, that $\gcd(p_1, q_1) = 1$ and $\gcd(p_1, p_2) =1$,
the former of which statements implies $\gcd(p_2, q_2) = 1$.

\end{proof}

\subsection{Dehn filling a Floer simple manifold}

We are now ready to construct $Y_1 \cup_{\varphi} Y_2$
as the Dehn filling of a Floer simple manifold $Y$.
For each $i \in \{1,2\}$, perform the (L-space) Dehn filling $Y_i(\mu_i)$,
writing $K_{\mu_i}$ for the knot core of $Y_i(\mu_i) \setminus Y_i$.
Next, let $Y$ denote the (Floer simple) knot complement
\begin{equation}
Y:= Y_1(\mu_1) \# Y_2(\mu_2) \setminus K_{\mu_1} \# K_{\mu_2}
\end{equation}
of the connected sum 
$K_{\mu_1} \# K_{\mu_2} \subset Y_1(\mu_1) \# Y_2(\mu_2)= Y\mkern-2mu(\mu_{\textsc{l}})$,
where $\mu_{\textsc{l}}$ denotes the meridian of $K_{\mu_1} \# K_{\mu_2}$,
and as usual, write $\iota : H_1(\partial Y) \to H_1(Y)$
for the map induced on homology by inclusion of the boundary,
and set $T := \Tors(H_1(Y))$ and $T^{\partial} := \iota(H_1(\partial Y)) \cap T$.
The maps $f_i: H_1(Y_i) \longrightarrow H_1(Y)$ induced by inclusion
descend to an isomorphism
$f_1 \oplus f_2 : (H_1(Y_1) \oplus H_1(Y_2))/(\iota_1(\mu_1) \sim \iota_2(\mu_2)) 
\stackrel{\sim}{\longrightarrow} H_1(Y)$
that identifies meridians,
via $f_1\iota_1(\mu_1) = f_2\iota_2(\mu_2) = \iota(\mu_{\textsc{l}})$.
In addition, 
$K_{\mu_1} \# K_{\mu_2}$ has a longitude $\lambda_{\textsc{l}}$ satisfying
$f_1(\iota_1(\lambda_1)) + f_2(\iota_2(\lambda_2)) = \iota(\lambda_{\textsc{l}})$.

Consider the Dehn filling $Y(\lambda_{\textsc{l}})$,
which one could regard as 0-surgery with
respect to the basis $(\mu_{\textsc{l}}, \lambda_{\textsc{l}})$ along the knot
$K_{\mu_1} \# K_{\mu_2} \subset Y(\mu_{\textsc{l}}) = Y_1(\mu_1) \# Y_2(\mu_2)$,
with $Y(\mu_{\textsc{l}})$ an L-space.
Since $Y$ already identifies 
$\iota_1( \mu_1)$ with $\iota_2(\varphi(\mu_1))$,
and since setting $\iota(\lambda_{\textsc{l}}) = 0$
identifies $\iota_1( \lambda_1)$ with $\iota_2(\varphi(\lambda_1))$,
we have
\begin{equation}
Y(\lambda_{\textsc{l}}) = Y_1 \cup_{\varphi}\! Y_2.
\end{equation}

To describe $Y(\lambda_{\textsc{l}})$ more explicitly, one can deduce that
$f_1 \oplus f_2$ restricts to an isomorphism
\begin{equation}
(\iota_1(H_1(\partial Y_1)) \oplus
\iota_2(H_1(\partial Y_2)))/(\iota_1(\mu_1) \sim \iota_2(\mu_2)) 
\;\stackrel{\sim}{\longrightarrow}\; \iota(H_1(\partial Y)) \oplus \left<\sigma_0 \right>,
\end{equation}
for some $\sigma_0 \in T$ with $|\!\left<\sigma_0 \right>\!| =  \gcd(g_1, g_2)$.
That is, if we define
\begin{equation}
g_0 := \gcd(g_1, g_2),
\;\;\;\;\;
\hat{g}_1 := g_1/g_0,
\;\;\;\;\;
\hat{g}_2 := g_2/g_0,
\;\;\;\;\;
g:= g_1 g_2/ g_0 = \hat{g}_1\hat{g}_2 g_0,
\end{equation}
then for $l \in H_1(\partial Y)$ an appropriately signed generator of  
$\iota^{-1}(T)$ and any $m \in H_1(\partial Y)$ satisfying
$m \cdot l = 1$, there are $\sigma_0 \in T$ of order $g_0$ and
$\xi \in \Z/g$ such that
\begin{align}
\label{eq: defs of f1 and f2 for splicing}
f_1\!: \iota_1(m_1) \mapsto p_2 \iota(m) + q_2 \hat{g}_1 \xi \iota(l) - q_1 \sigma_0,
\;\;\;\;\;\;
&f_2\!: \iota_2(m_2) \mapsto p_1 \iota(m) + q_1 \hat{g}_2 \xi \iota(l) + q_2 \sigma_0,
\\ \nonumber
f_1\!: \iota_1(l_1) \mapsto p_2 \hat{g}_2 \xi \iota(l) + p_1 \sigma_0,
\;\;\;\;\;\;\;\;\;\;\;\;\;\;\;\;\;\;\;\;\;\;\;
&f_2\!: \iota_2(l_2) \mapsto p_1 \hat{g}_1 \xi \iota(l) - p_2 \sigma_0.
\end{align}
Thus,  $g = |T^{\partial}|$, and if we write
\begin{equation}
\mu_{\textsc{l}} = p m + q l,\;\;\;\;\;\;\; 
\lambda_{\textsc{l}} = q^* m + p^* l,
\end{equation}
then $p$, $q$, $q^*$, and $p^*$ satisfy
\begin{align}
\label{eq: equations for p,q q* and p*}
\;\;\;\;p = p_1 p_2,\;\;\;\;\;\;\;\;\;\;\;\;\;\;\;\;\;\;\;\;\;
q\, &\equiv (q_1 p_2 g_2  +   q_2 p_1 g_1) \xi \;(\mod g),\;\;\;\;\;
        \\ \nonumber
\;\;\;\;q^*\! = q_1^* p_2 + q_2^* p_1,\;\;\;\;\;\;\;\;\;
p^*\! &\equiv ((p_1 p_2^* +  q_1^* q_2 ) g_1 
+ (p_2 p_1^* + q_2^* q_1 )g_2 ) \xi \;(\mod g).
\end{align}
Again, the condition $\mu_{\textsc{l}} \cdot \lambda_{\textsc{l}} = 1$
determines the value of $\xi$, which we shall not need.
Of course, it will often be more convenient 
to express this restriction of $\iota_i(H_1(\partial Y_i))$ to 
$f_1 \oplus f_2$ in terms of the bases
$(\iota_i(\mu_i), \iota_i(\lambda_i))$ for $\iota_i(H_1(\partial Y_i))$
and $(\iota(\mu_{\textsc{l}}), \iota(\lambda_{\textsc{l}}))$
for $\iota(H_1(\partial Y))$,
as we shall describe explicitly in the proof of Proposition 
\ref{prop: L-space condition for Y1 cup Y2 in terms of b}.

In either case, we see that $q^* = q_1^* p_2 + q_2^* p_1$
makes its appearance as $\lambda_{\textsc{l}} \cdot l$.
Thus, $Y_1 \cup_{\varphi} Y_2 = Y(\lambda_{\textsc{l}})$
can be regarded as surgery with label
$(\mu_{\textsc{l}} \cdot \lambda_{\textsc{l}})/(\lambda_{\textsc{l}} \cdot l)
= 1/q^*$ along $K_{\mu_1} \# K_{\mu_2} \subset Y(\mu_{\textsc{l}})$.

\subsection{Computation of $\boldsymbol{\dt}(\boldsymbol{Y})$}

For the remainder of Section \ref{s: L-space if and only if intervals cover.},
we regard the entire preceding construction,
along with the hypotheses of Theorem 
\ref{thm: torus gluing and complementary intervals for L-spaces},
as fixed initial data.
We are now ready to compute $\dt(Y)$,
which we shall call $\dtge(Y)$ to emphasize that in this case we are not
excluding torsion elements.

\begin{prop}
\label{prop:computation of dtge}
Suppose that $\mu_1$ is ``judiciously chosen'' from
$\P^{-1}(\dot{I}_1 \cap \varphi^{-1}_{\P}(\dot{I}_2))$ nonempty,
and that $Y$ is constructed as above.
If we set $t_{\partial}:= [(\iota(m))]$, then
$\dtge(Y) = A_0 \amalg (A_1 \cup A_2) \amalg A_3$, with
\begin{align}
A_0
&:=
S\!\left[\frac{1}{1-t_{\partial}} - \frac{1 - t_{\partial}^{p_1 p_2}}
{(1 - t_{\partial}^{p_1})(1 - t_{\partial}^{p_2})} \right] + T^{\partial},
  \\ \nonumber
A_1
&:=
f_1(\dtge(Y_1)) \;+\;
f_2\!\left(\{0, \ldots, p_2 - 1\} \iota_2 (m_2)  + T^{\partial}_2\right)  ,
  \\ \nonumber
A_2
&:=
f_2(\dtge(Y_2)) \;+\; 
f_1\!\left(\{0, \ldots, p_1 - 1\} \iota_1 (m_1)  + T^{\partial}_1\right)  ,
  \\ \nonumber
A_3
&:=
\iota(\mu_{\textsc{l}}) 
+  f_1(\dtge(Y_1))
+  f_2(\dtge(Y_2)).
\end{align}
\end{prop}

\begin{proof}

To compute $\dtge(Y)$, we need
the Turaev torsion $\tau(Y)$ and torsion complement $\tc(Y)$.
In order to write these down,
we first choose generators $\bar{m}$ for $H_1(Y)/T$
and $\bar{m}_i$ and $H_1(Y_i)/T_i$ satisfying
\begin{equation}
\iota(m) \in g\bar{m} + T,
\;\;\;\;
\iota_i(m_i) \in g_i\bar{m}_i + T_i,
\;\;\;\; 
i \in \{1,2\}.
\end{equation}
Recall that the above condition only constrains the signs
of $\bar{m}$ and $\bar{m}_i$.  We shall write
\begin{equation}
t := [\bar{m}] \in \Z[H_1(Y)],
\;\;\;\;\;
t_i := [\bar{m}_i] \in \Z[H_1(Y_i)],
\;\;\;\; 
i \in \{1,2\},
\end{equation}
for the inclusions of $\bar{m}$ and $\bar{m}_i$ into their respective group rings.

Invoking the standard gluing rules for Turaev torsion yields
\begin{align}
\label{eq: torsion for union}
\tau(Y)
&= \left(1 - [\iota(\mu_{\textsc{l}})]\right) 
\tilde{f}_1(\tau(Y_1)) \, \tilde{f}_2(\tau(Y_2)),
\end{align}
where each $\tilde{f}_i$ denotes the lift of
$f_i$ to the Laurent series group ring $\Z[t_i^{-1}, t_i]][T_i] \supset \Z[H_1(Y_i)]$.
(One could also obtain this result by using Proposition~\ref{Prop:Chi}
and the fact that Heegaard Floer homology tensors on connected sums.)

For $i \in \{1, 2\}$, set $P_T := \sum_{h \in T}[h] \in \Z[H_1(Y)]$ 
and $P_{T_i} := \sum_{h_i \in T_i}[h_i] \in \Z[H_1(Y_i)]$,
and let $P$ and $P_i$ denote the Laurent series
$P := P_T/(1-t)$ and $P_i := P_{T_i}/(1-t_i)$, the latter with polynomial truncations
\begin{equation}
\bar{P}_i
\;:=\;
 (1 - [\iota_i(\mu_i)])P_i 
\;=\; 
\frac{1 - t_i^{p_i g_i}}{1-t_i} P_{T_i}.
\end{equation}
The torsion complements 
$\tc(Y) := P - \tau(Y)$ and
$\tc(Y_i) := P_i - \tau(Y_i)$
then satisfy
\begin{align}
&\;\;\;\;\; \tc(Y)
=
P -  \left(1 - [\iota(\mu_{\textsc{l}})]\right)\!
\tilde{f}_1(P_1 - \tc(Y_1))
\tilde{f}_2(P_2 - \tc(Y_2))
       \\ \nonumber
&\phantom{\;\;\;\;\; \tc(Y)}
=
A^{\mathrm{c}}_0 + A^{\mathrm{c}}_{12} +  A^{\mathrm{c}}_3,
         \\ \nonumber
\!\!\!\!\!\!\!\!\text{with}\;\;
A^{\mathrm{c}}_0
&:= P -  \left(1 - [\iota(\mu_{\textsc{l}})]\right)\! \tilde{f}_1(P_1) \tilde{f}_2(P_2),
         \\ \nonumber
A^{\mathrm{c}}_{12}
&:=
\tilde{f}_1 (\tc(Y_1)) \tilde{f}_2(\bar{P}_2) 
+\tilde{f}_1(\bar{P}_1) \tilde{f}_2 (\tc(Y_2))
- \tilde{f}_1( \tc(Y_1))\tilde{f}_2( \tc(Y_2)) ,
         \\ \nonumber
A^{\mathrm{c}}_3
&:=[\iota(\mu_{\textsc{l}})]\tilde{f}_1 (\tc(Y_1))\tilde{f}_2 (\tc(Y_2)).
\end{align}

It is straightforward to show that each of $A^{\mathrm{c}}_0$,
$A^{\mathrm{c}}_{12}$, and $A^{\mathrm{c}}_3$
is an element of $\Z[H_1(Y)]$ with coefficients in $\{0, 1\}$,
and that the three sets
$S[A^{\mathrm{c}}_0]$,
$S[A^{\mathrm{c}}_{12}]$, and $S[A^{\mathrm{c}}_3]$ are disjoint.
In particular, $A^{\mathrm{c}}_0$ satisfies the property
\begin{equation} 
\label{eq: trap for extra stuff}
\left(S[A^{\mathrm{c}}_0 ] - S[\tilde{f}_1(P_1)\tilde{f}_2 (P_2)]\right)
\cap S[P] = S[A^{\mathrm{c}}_0],
\end{equation}
while $A^{\mathrm{c}}_{12}$ satisfies
\begin{equation}
S[A^{\mathrm{c}}_{12}] 
= S[\tilde{f}_1 (\tc(Y_1)) \tilde{f}_2(\bar{P}_2) 
+\tilde{f}_1(\bar{P}_1) \tilde{f}_2 (\tc(Y_2))].
\end{equation}
On the other hand, since each $(1-[\iota_i(\mu_i)])\tau(Y_i)$ has no negative
coefficients,  it follows from 
(\ref{eq: torsion for union}) that
$\tau(Y)$ has support 
\vspace{-.11cm}
\begin{equation}
S[\tau(Y)] 
= S[ \tilde{f}_1(\tau(Y_1)) \, \tilde{f}_2(\tau(Y_2))] \subset H_1(Y).
\end{equation}

Lastly, we compute $\dtge(Y):= (S[\tc(Y)] - S[\tau(Y)]) \cap \iota(m\Z_{\ge 0} + l\Z)$.
Using the facts that $0 \in S[\tau(Y_i)]$ for each $i \in \{1,2\}$
(as per the convention stated in 
(\ref{eq: torsion support 0 convention}) in
Section \ref{ss: torsion conventions}) and that
$\iota(H_1(\partial Y)) \subset 
f_1\iota_1(H_1(\partial Y_1)) \oplus f_2\iota_2(H_1(\partial Y_2))$,
we obtain 
$\dtge(Y) = A_0 \amalg (A_1 \cup A_2) \amalg A_3$, with
\begin{align}
A_0
&= S[A^{\mathrm{c}}_0 ] \cap \iota(m \Z_{\ge0} + lZ),
  \\ \nonumber
A_1
&=
f_1(\dtge(Y_1)) +  f_2(S[\bar{P}_2]\cap \iota_2(H_1(Y_2))),
  \\ \nonumber
A_2
&=
f_2(\dtge(Y_2)) +  f_1(S[\bar{P}_1]\cap \iota_1(H_1(Y_1))) ,
  \\ \nonumber
A_3
&=
\iota(\mu_{\textsc{l}}) 
+  f_1(\dtge(Y_1))
+  f_2(\dtge(Y_2)),
\end{align}
where property 
(\ref{eq: trap for extra stuff})
has made any remaining subsets
of $S[\tc(Y)] - S[\tau(Y)]$---such as, for example,
$f_1(S[\tc(Y_1)] - S[\tau(Y_1)]) \cap (m \Z_{<0} + T)$---land in $S[A^{\mathrm{c}}_0 ]$.
It is straightforward to show that the above $A_i$ are equal to those
enumerated in the statement of the proposition.
\end{proof}

\subsection{Computation of L-space interval for $\boldsymbol{Y}$}
Having determined $\dt(Y)$, we can apply
Theorem 
\ref{thm:  L-space interval in terms of beta/n}
to compute the L-space interval for $Y$.

\begin{prop}
\label{prop: L-space condition for Y1 cup Y2 in terms of b}
Suppose that $\mu_1$ is ``judiciously chosen'' from
$\P^{-1}(\dot{I}_1 \cap \varphi^{-1}_{\P}(\dot{I}_2))$ nonempty,
and that $Y$ is constructed as above.
For each $i \in \{1,2\}$, set $\bar{q}_i := \left[q_i^*\right]_{p_i}$ and let
$B_i$ denote the set
$B_i := \left\{ \left. [p_i \gamma_i - q_i \delta_i ]_{p_ig_i}\!\right| 
{\boldsymbol{\delta}_i} = \delta_i \iota_i(m_i) + \gamma_i \iota_i(l_i) 
\in \dtge(Y_i)  \right\}$.
Then $Y_1 \cup_{\varphi} Y_2$ is an L-space if and only if
condition $({\textsc{l}}.i)$ holds for each $b_1 \in B_1$,
$({\textsc{l}}.ii)$ holds for each $b_2 \in B_2$, and
$({\textsc{l}}.iii)$ holds for each $(b_1, b_2) \in B_1 \times B_2$ with
$b_1 \equiv b_2 \, (\mod g_0)$:
\begin{align*}
({\textsc{l}}.i)\;\;\;\;\;
 \frac{1}{b}\!\left\lfloor \! \frac{b \bar{q}_1}{p_1}\! \right\rfloor
+\frac{1}{b}\!\left\lfloor \! \frac{b \bar{q}_2}{p_2}\! \right\rfloor \ge 1
\;\;\forall\;b \equiv b_1\,(\mod p_1g_1),\; 0<b< pg,
\;\;\;\;\;\;\;\;\;\;\;\;\;\;\;\;\;\;\;\;\;\;\;\;\;\;\;\;\;\;
          \\
({\textsc{l}}.ii)\;\;\;\;
 \frac{1}{b}\!\left\lfloor \! \frac{b \bar{q}_1}{p_1}\! \right\rfloor
+\frac{1}{b}\!\left\lfloor \! \frac{b \bar{q}_2}{p_2}\! \right\rfloor \ge 1
\;\;\forall\;b \equiv b_2\,(\mod p_2g_2),\; 0<b< pg,
\;\;\;\;\;\;\;\;\;\;\;\;\;\;\;\;\;\;\;\;\;\;\;\;\;\;\;\;\;\;
          \\
({\textsc{l}}.iii)\;\;\;
 \frac{1}{b}\!\left\lfloor \! \frac{b \bar{q}_1}{p_1}\! \right\rfloor
+\frac{1}{b}\!\left\lfloor \! \frac{b \bar{q}_2}{p_2}\! \right\rfloor > 1
\;\;\forall\; 
b \equiv b_1\,(\mod p_1g_1),\;b \equiv b_2\,(\mod p_2g_2),
\; 0<b< pg,
\end{align*}
where $p:= p_1 p_2$ and $g := g_1 g_2 / g_0$, with $g_0 = \gcd(g_1, g_2)$.
\end{prop}
\begin{proof}

We begin by ensuring that $\dtge(Y)$ meets the conditions of 
Theorem 
\ref{thm:  L-space interval in terms of beta/n}
Since $A_0 \not\subset T$ implies $\dtgz(Y) \neq \emptyset$,
it remains to verify, for each 
$\boldsymbol{\delta} = \delta \iota(m) + \gamma \iota(l) \in \dtge(Y)$,
that 
$b_{\boldsymbol{\delta}} := [p\gamma - q\delta]_{pg} 
(\equiv \mu_{\textsc{l}} \cdot \iota^{-1}({\boldsymbol{\delta}})\, (\mod pg))$ is nonzero,
or equivalently, that $\boldsymbol{\delta} \notin \left<\iota(\mu_{\textsc{l}})\right>$.
Now, the definition of $\dtge$ already  implies $0 \notin \dtge(Y)$.
Recalling the result of
Proposition \ref{prop:computation of dtge},
and that
$\iota(\mu_{\textsc{l}}) = p\iota(m) + q\iota(l)$ with $p := p_1 p_2$, 
we know that the inclusions 
\begin{align}
A_0 
&\subset \{1, \ldots, p_1p_2 - p_1 - p_2\}\iota(m) + T^{\partial},
      \\
A_1 \cup A_2 
&\subset
f_1\!\left(\{0, \ldots, p_1 - 1\} \iota_1 (m_1) + T_1^{\partial}\right)
+ f_2\!\left(\{0, \ldots, p_2 - 1\} \iota_2 (m_2) + T_2^{\partial}\right)
        \\ \nonumber
&=(\{0, \ldots, p_1 - 1\} p_2 + \{0, \ldots p_2 - 1\} p_1 ) \iota(m) + T^{\partial}
\end{align}
imply that $\left<\iota(\mu_{\textsc{l}})\right> \cap ( A_0 \cup A_1 \cup A_2) = \emptyset$.
Lastly, since 
our ``judiciously chosen'' hypothesis makes 
$\deg_{[\bar{m}_i]} \tc(Y_i) <  p_ig_i = \deg_{[\bar{m}_i]}[\iota_i(\mu_i)]$,
and since
the kernel of $f_1 \oplus f_2$ is generated by
$(\iota_1(\mu_1), -\iota_2(\mu_2))$, we know that
$\left<\iota(\mu_{\textsc{l}})\right> \cap A_3 = \emptyset$.
Thus, 
Theorem 
\ref{thm:  L-space interval in terms of beta/n}
applies.

Since we can regard $Y_1 \cup_{\varphi} Y_2 = Y(\lambda_{\textsc{l}})$
as surgery with label 
$1/q^*$
along $K_{\mu_1} \# K_{\mu_2} \subset Y(\mu_{\textsc{l}})$,
Theorem 
\ref{thm:  L-space interval in terms of beta/n}
tells us that $Y_1 \cup_{\varphi} Y_2$ is an L-space if and only if
\begin{equation}
\label{eq: splicing constraint inequality for 1/q*}
\frac{b_{\boldsymbol{\delta}}-p}{\delta}
\le
\frac{1}{q^*}
\le
\frac{b_{\boldsymbol{\delta}}}{\delta}
\end{equation}
for all ${\boldsymbol{\delta}}  = \delta \iota(m) + \gamma \iota(l) 
\in \dtgz(Y)\; (= \dtge(Y) \setminus T)$.
Now, since
$b_{\boldsymbol{\delta}} \equiv \mu_{\textsc{l}} \cdot \tilde{\boldsymbol{\delta}} \, (\mod pg)$
for any lift $\tilde{\boldsymbol{\delta}} \in \iota^{-1}({\boldsymbol{\delta}})$, 
there always exists a unique $a_{\boldsymbol{\delta}} \in \Z$ for which
${\boldsymbol{\delta}} = \iota(
a_{\boldsymbol{\delta}} \mu_{\textsc{l}}
+ b_{\boldsymbol{\delta}}  \lambda_{\textsc{l}})$.
Such 
$a_{\boldsymbol{\delta}} \in \Z$
satisfies $\delta = a_{\boldsymbol{\delta}} p + b_{\boldsymbol{\delta}} q^*$.
Taking this as a definition for $a_{\boldsymbol{\delta}} \in \Z$,
we note that,
since $b_{\boldsymbol{\delta}}-p < 0$ and $q^* >0$, the left-hand inequality
in (\ref{eq: splicing constraint inequality for 1/q*}) is vacuous, whereas
the right-hand inequality is equivalent to the condition
$a_{\boldsymbol{\delta}}  \le 0$.

Since $b_{\boldsymbol{\delta}} q^* >0$ for all 
${\boldsymbol{\delta}} = \delta \iota(m) + \gamma \iota(l)  \in \dtge(Y)$,
we obtain $a_{\boldsymbol{\delta}}  \le 0$ automatically whenever
$\delta < p$.  In particular, $a_{\boldsymbol{\delta}}  \le 0$ for all
${\boldsymbol{\delta}} \in A_0$ and for any
${\boldsymbol{\delta}} \in \dtge(Y) \cap (0\iota(m) + T^{\partial})$.
Now, the latter case is, strictly speaking, irrelevant to the question of
whether $Y_1 \cup_{\varphi} Y_2$ is an L-space,
but the fact that the condition $a_{\boldsymbol{\delta}}  \le 0$
is vacuous on torsion elements of $\dtge(Y)$ allows us to apply the condition to
all of $\dtge(Y)$, thereby simplifying our bookkeeping.

It remains to apply the condition $a_{\boldsymbol{\delta}}  \le 0$
to each of  $A_1$, $A_2$, and $A_3$, from which we shall obtain the
respective conditions $({\textsc{l}}.i)$, $({\textsc{l}}.ii)$, and $({\textsc{l}}.iii)$.
To do this, we first, for each $i \in \{1,2\}$, consider the bijection,
\begin{align}
\{0, \ldots, p_i g_i-1\} 
&\longrightarrow
\{0, \ldots, p_i - 1\} \iota_i(m_i) + T^{\partial}_i \;\;\subset\; \iota_i(H_1(\partial Y_i)),
         \\ \nonumber
b_i
&\longmapsto
\iota_i(-\!\left\lfloor\!\frac{b_iq_i^*\!}{p_i}\!\right\rfloor\! \mu_i 
\,+\, b_i \lambda_i) 
\;\;\in\;
[b_i \bar{q}_i]_{p_i} \iota_i(m_i) + T^{\partial}_i  \mkern-4mu,
\end{align}
recalling that $\mu_i = p_i m_i + q_i l_i$, $\lambda_i = q^*_i m_i + p^*_i l_i$,
and $g_i := |T^{\partial}_i|$ with $T^{\partial}_i = \!\left<\iota_i(l_i)\right>\!$.
The inverse map sends
\begin{equation}
\mathbf{x}_i := x_i \iota_i(m_i) + y_i \iota_i(l_i)
\longmapsto
b_i^{\mathbf{x}_i} := [\mu_i \cdot (x_i m_i + y_i l_i)]_{p_i g_i} = [p_i y_i - q_i x_i]_{p_i g_i}.
\end{equation}
Thus, if we define 
$a_i^{\mathbf{x}_i} := -(b_i^{\mathbf{x}_i} q^*_i - \left[b_i^{\mathbf{x}_i} q^*_i\right]_{p_i})/p_i$,
then for any 
$\mathbf{x}_i := x_i \iota_i(m_i) + y_i \iota_i(l_i)$ with 
$x_i \in \{0, \ldots, p_i - 1\}$, and for any $s_i \in \Z$, we have
\begin{equation}
\label{eq: s lifts of xi in terms of a and b}
\mathbf{x}_i = \iota_i(a_i^{\mathbf{x}_i} \mu_i + b_i^{\mathbf{x}_i} \lambda_i)
= \iota_i( (a_i^{\mathbf{x}_i} - q^*_i g_i s_i) \mu_i + 
(b_i^{\mathbf{x}_i} + p_i g_i s_i) \lambda_i),
\end{equation}
with $s_i \in \Z$ parametrizing the lifts $\iota_i^{-1}(\mathbf{x}_i )$
of $\mathbf{x}_i $.

Since $f_1\iota_1(\mu_1) =  f_2\iota_2(\mu_2) = \iota(\mu_{\textsc{l}})$
and $f_1\iota_1(\lambda_1) + f_2\iota_2(\lambda_2) = \iota(\lambda_{\textsc{l}})$,
we deduce from (\ref{eq: s lifts of xi in terms of a and b}) that
$f_1(\mathbf{x}_1) + f_2(\mathbf{x}_2) \in \iota(H_1(\partial Y))$ if and only if
there exist $s_1, s_2 \in \Z$ such that
$b_1^{\mathbf{x}_1} + p_1 g_1 s_1 = b_2^{\mathbf{x}_2} + p_2 g_2 s_2$,
which, in turn, occurs if and only if
$b_1^{\mathbf{x}_1}  \equiv b_2^{\mathbf{x}_2} (\mod g_0)$, since
$g_0 = \gcd(p_1 g_1, p_2 g_2) = \gcd(g_1, g_2)$.
In such case, if we write
$f_1(\mathbf{x}_1) + f_2(\mathbf{x}_2) = \iota(a \mu_{\textsc{l}} + b \lambda_{\textsc{l}} )$
with $b\in \{0, \ldots, pg - 1\}$,
then $b$ is
the unique solution in $\{0, \ldots, pg - 1\}$ to the equivalences
$b \equiv b_1^{\mathbf{x}_1} \,(\mod p_1 g_1)$,
$b \equiv b_2^{\mathbf{x}_2} \,(\mod p_2 g_2)$.
Setting $b = b_i^{\mathbf{x}_i} + p_i g_i s_i$ makes $g_i s_i = (b - b_i^{\mathbf{x}_i})/p_i$
for each $i \in \{1,2\}$, so that we obtain
\begin{align}
\label{eq: computing a for spliced Dt}
a
&= {\textstyle{\sum_{i\in\{1,2\}}}} (a_i^{\mathbf{x}_i} - q^*_i g_i s_i)
           \\ \nonumber
&= {\textstyle{\sum_{i\in\{1,2\}}}}
\!\left(
-(b_i^{\mathbf{x}_i}  q^*_i - \left[b_i^{\mathbf{x}_i}  q^*_i\right]_{p_i})/p_i
\;-\; q_i^*(b - b_i^{\mathbf{x}_i})/p_i
\right)
           \\ \nonumber
&= 
-b( \bar{q}_1 p_2 + \bar{q}_2 p_1 - p_1 p_2)/p_1p_2
\,+\, \left[b_1^{\mathbf{x}_1}  q^*_1\right]_{p_1}\!/{p_1}
\,+\, \left[b_2^{\mathbf{x}_2}  q^*_2\right]_{p_2}\!/{p_2}
           \\ \nonumber
&= -b 
- \left\lfloor\frac{b\bar{q}_1}{p_1}\right\rfloor
- \left\lfloor\frac{b\bar{q}_2}{p_2}\right\rfloor,
\end{align}
where the third line uses the identity 
$q^* := q_1^* p_2 + q_2^* p_1 =  \bar{q}_1 p_2 + \bar{q}_2 p_1 - p_1 p_2$,
implied by the hypothesis $p_1p_2 > q^* >0$ and the definitions
$\bar{q}_i := [q_i^*]_{p_i}$.

Since we may write any ${\boldsymbol{\delta}} \in A_1$ as
\begin{equation}
{\boldsymbol{\delta}} 
\;=\; f_1({\boldsymbol{\delta}_1}) + f_2(\mathbf{x}_2)
\;=\; \iota(a \mu_{\textsc{l}} + b \lambda_{\textsc{l}})
\end{equation}
with ${\boldsymbol{\delta}_1} \in \dtge(Y_1)$,
$\mathbf{x}_2 \in \{0, \ldots, p_2-1\}\iota_2(m_2) + \{0, \ldots, g_2-1\}\iota_2(l_2)$
satisfying $b_2^{\mathbf{x}_2} \equiv b_1^{\boldsymbol{\delta}_1}\, (\mod g_0)$,
$0 < b < pg$, and $a$ as determined in (\ref{eq: computing a for spliced Dt}),
we have $a_{\boldsymbol{\delta}} = a$, and demanding $a_{\boldsymbol{\delta}}  \le 0$
yields condition $(i)_{\textsc{l}}$.
Likewise, applying $a_{\boldsymbol{\delta}} \leq 0$ for all
${\boldsymbol{\delta}} \in A_2$ yields condition $(ii)_{\textsc{l}}$.
The case of $A_3$ is similar, except that since
${\boldsymbol{\delta}} = 
\iota(\mu_{\textsc{l}}) + f_1({\boldsymbol{\delta}_1}) + f_2({\boldsymbol{\delta}_2})$,
we need 
$a_{\boldsymbol{\delta}} = 1 + a \le 0$
and $b_1^{\boldsymbol{\delta}_1} \equiv b_2^{\boldsymbol{\delta}_2}\, (\mod g_0)$,
yielding condition ($\mathrm{iii}_{\textsc{l}}$).

\end{proof}

\subsection{Determining when gluing hypothesis is met}

We next turn our attention to the L-space filling slope intervals
$I_i \subset \P(H_1(\partial Y_i))$, to determine when they
combine according to the hypotheses of the theorem.
\begin{prop}
\label{prop: interval covering condition in terms of b1 and b2}
Suppose that $\mu_1$ is ``judiciously chosen'' from
$\P^{-1}(\dot{I}_1 \cap \varphi^{-1}_{\P}(\dot{I}_2))$ nonempty,
and that $Y$ is constructed as above.
For each $i \in \{1,2\}$, set $\bar{q}_i := \left[q_i^*\right]_{p_i}$ and let
$B_i$ denote the set
$B_i := \left\{ \left. [p_i \gamma_i - q_i \delta_i ]_{p_ig_i}\!\right| 
{\boldsymbol{\delta}_i} = \delta_i \iota_i(m_i) + \gamma_i \iota_i(l_i) 
\in \dtge(Y_i)  \right\}$.
Then $\varphi_{\P}(\dot{I}_1) \cup \dot{I}_2 = \P(H_1(\partial Y_2))$
when both $\dtge(Y_i)$ are nonempty---and 
$\varphi_{\P}(I_1) \cup I_2 = \P(H_1(\partial Y_2))$ when
one or both $\dtge(Y_i)$ are empty---if and only if
the following three conditions hold:
\begin{align*}
(\textsc{i}.i)\;\;\;\;
 \frac{1}{b_1}\!\left\lfloor \! \frac{b_1 \bar{q}_1}{p_1}\! \right\rfloor
+\frac{1}{b_1}\!\left\lfloor \! \frac{b_1 \bar{q}_2}{p_2}\! \right\rfloor \ge 1
\;\;\mathrm{for}\;\mathrm{all}\; b_1 \in B_1,
\;\;\;\;\;\;\;\;\;\;\;\;\;\;\;\;\,
          \\
(\textsc{i}.ii)\;\;\;
 \frac{1}{b_2}\!\left\lfloor \! \frac{b_2 \bar{q}_1}{p_1}\! \right\rfloor
+\frac{1}{b_2}\!\left\lfloor \! \frac{b_2 \bar{q}_2}{p_2}\! \right\rfloor \ge 1
\;\;\mathrm{for}\;\mathrm{all}\; b_2 \in B_2,
\;\;\;\;\;\;\;\;\;\;\;\;\;\;\;\;\,
          \\
\;\;\;\;\,(\textsc{i}.iii)\;\;\,
 \frac{1}{b_1}\!\left\lfloor \! \frac{b_1 \bar{q}_1}{p_1}\! \right\rfloor
+\frac{1}{b_2}\!\left\lfloor \! \frac{b_2 \bar{q}_2}{p_2}\! \right\rfloor > 1
\;\;\mathrm{for}\;\mathrm{all}\; (b_1, b_2) \in B_1 \times B_2.
\end{align*}
\end{prop}
\begin{proof}

For $i \in \{1,2\}$, let $\pi_i$ denote the ``surgery label'' map,
$\pi_i :
H_1(\partial Y_i) \setminus \{0\} 
\longrightarrow \Q \cup \infty$,
\begin{equation}
\label{eq: def of pi i surgery label}
\pi_i
\, : \,
\alpha_i \mu_i + \beta_i \lambda_i
\,\longmapsto\,
\frac{\mu_i \cdot(\alpha_i \mu_i + \beta_i \lambda_i)}
{(\alpha_i \mu_i + \beta_i \lambda_i) \cdot l_i}
= \frac{\beta_i}{\alpha_i p_i + \beta_i q^*_i},
\end{equation}
and for each ${\boldsymbol{\delta}_i} = \delta_i \iota_i(m_i) + \gamma_i \iota_i(l_i)
\in \dtge(Y_i)$, let 
$\tilde{\boldsymbol{{\delta}}}_{i+},
\tilde{\boldsymbol{{\delta}}}_{i-} \in \iota_i^{-1}({\boldsymbol{\delta}_i})$
denote the two lifts of ${\boldsymbol{\delta}_i}$ closest to $\mu_i$ with
respect to surgery label, {\it{i.e.}}, 
\begin{align}
\tilde{\boldsymbol{{\delta}}}_{i+}
&= a_{i+}^{\boldsymbol{\delta}_i} \mu_i +  b_{i+}^{\boldsymbol{\delta}_i} \lambda_i,
\;\;\;b_{i+}^{\boldsymbol{\delta}_i} := [p_i\gamma_i -q_i \delta_i]_{p_ig_i}
   \\
\tilde{\boldsymbol{{\delta}}}_{i-}
&= a_{i-}^{\boldsymbol{\delta}_i} \mu_i +   b_{i-}^{\boldsymbol{\delta}_i} \lambda_i
:= \left(a_{i+}^{\boldsymbol{\delta}_i}\!+q^*_ig_i\right) \mu_i +  
\left(b_{i+}^{\boldsymbol{\delta}_i}\!-p_ig_i\right) \lambda_i.
\end{align}
Note that since $p_i > \deg_{[\bar{m}_i]}\!\tc(Y_i)$ implies $\delta_i < p_i$, we have
\begin{equation}
\delta_i 
= a_{i+}^{\boldsymbol{\delta}_i} p_i +  b_{i+}^{\boldsymbol{\delta}_i} q^*_i
= [b_{i+}^{\boldsymbol{\delta}_i} q^*_i]_{p_i}
= a_{i-}^{\boldsymbol{\delta}_i} p_i +  b_{i-}^{\boldsymbol{\delta}_i} q^*_i
= [b_{i-}^{\boldsymbol{\delta}_i} q^*_i]_{p_i}
\ge 0.
\end{equation}
Note also that
$\pi_i(\tilde{\boldsymbol{{\delta}}}_{i-}) < 0 <  
\pi_i(\tilde{\boldsymbol{{\delta}}}_{i+})$ unless $\delta_i = 0$,
in which case $\pi_i(\tilde{\boldsymbol{{\delta}}}_{i-}) 
= \pi_i(\tilde{\boldsymbol{{\delta}}}_{i+}) = \infty$.

Corollary 
\ref{cor: L-space criterion in Dehn filling basis}
then implies that, for $\dtge(Y_i)$ nonempty,
$\tilde{I}_i := \P^{-1}(I_i) \subset H_1(\partial Y_i)$ takes the form
$\tilde{I}_i = \bigcap_{{\boldsymbol{\delta}}_{\boldsymbol{i}} \in \dtge(Y_i)} 
\tilde{I}_i^{\boldsymbol{\delta}_i}$, where
\begin{equation}
\tilde{I}_i^{\boldsymbol{\delta}_i}
:=
\left\{\mu \in H_1(\partial Y_i) \setminus \{0\} \left|
\begin{array}{lr}
\pi_i(\tilde{\boldsymbol{{\delta}}}_{i-}) \le
\pi_i(\mu) \le
\pi_i(\tilde{\boldsymbol{{\delta}}}_{i+})
&   \mathrm{if}\; \delta_i >0,
      \\
\pi_i(\mu) \neq \infty \;(=\pi_i(\tilde{\boldsymbol{{\delta}}}_{i-}\mkern-.6mu)
\!=\! \pi_i(\tilde{\boldsymbol{{\delta}}}_{i+}))
&    \mathrm{if}\; \delta_i = 0\hphantom{,}
\end{array}\!
\right.\right\}.
\end{equation}
If $\dtge(Y_i) = \emptyset$, then, similarly to the case in which $\delta_i = 0$,
$\tilde{I}_i$ is the 
complement of $\pi_i^{-1}(\infty)$.

Note that we always have $\infty \notin \pi_i(\tilde{I}_i)$.
Thus, a necessary condition to achieve $\varphi_{\P}(I_1) \cup I_2 = \P(H_1(\partial Y_2))$
or $\varphi_{\P}(\dot{I}_1) \cup \dot{I}_2 = \P(H_1(\partial Y_2))$ is to have
\begin{equation}
\label{eq: infty included}
(\si.i)\;\;\; \infty \in \pi_2\!\circ\!\varphi(\tilde{I}_1),
\;\;\;\;\;\;\;\;
(\si.ii)\;\;\; \infty \in \pi_1\!\circ\!\varphi^{-1}(\tilde{I}_2).
\end{equation}
We claim that conditions $(\si.i)$ and $(\si.ii)$ are respectively equivalent to
$(\textsc{i}.i)$ and $(\textsc{i}.ii)$.
First note that it is sufficient to prove the equivalence of 
$(\si.i)$ and $(\textsc{i}.i)$, since the maps
\begin{equation}
\varphi :
\alpha \mu_1 + \beta \lambda_1 \mapsto
\alpha \mu_2 - \beta \lambda_2,\;\;\;
\varphi^{-1} :
\alpha \mu_2 + \beta \lambda_2 \mapsto
\alpha \mu_1 - \beta \lambda_1
\end{equation}
are exchanged by swapping $i=1$ with $i=2$.
Also, when $\dtge(Y_1) = \emptyset$, in which case
$(\textsc{i}.i)$ holds vacuously, our hypothesis that $q^* \neq 0$,
ensuring that $\pi_2\varphi\pi_1^{-1}(\infty) \neq \infty$,
implies ($\si.i$) holds automatically.
Thus, we henceforth assume that $\dtge(Y_1)$ is nonempty.

For any $a_1 \mu_1 + b_1 \lambda_1 \in H_1(\partial Y_1) \setminus \{0\}$,
it is straightforward to show that the map
\begin{equation}
\label{eq: pi2 phi map}
\pi_2 \!\circ\! \varphi:  a_1 \mu_1 + b_1 \lambda_1
\mapsto  \frac{-b_1}{a_1p_2 - b_1q_2^*}
\end{equation}
has denominator satisfying
\begin{equation}
a_1p_2 - b_1q_2^*
=\frac{p_2}{p_1} (a_1 p_1 + b_1 q_1^*) - b_1 \frac{q^*}{p_1}.
\end{equation}
In particular, since $q^* > 0$, and since
$\delta_1 = a_{1-}^{\boldsymbol{\delta}_1} p_1 +  b_{1-}^{\boldsymbol{\delta}_1} q^*_1
\ge 0$
and
$b_{1-}^{\boldsymbol{\delta}_1} < 0$
for any ${\boldsymbol{\delta}_1} \in \dtge(Y_1)$,
we have
\begin{equation}
\label{eq: -lift pos, with pos denom}
a_{1-}^{\boldsymbol{\delta}_1} p_2 - b_{1-}^{\boldsymbol{\delta}_1} q^*_2 >0,\;\;
 \pi_2 \!\circ\!\varphi(\tilde{\boldsymbol{{\delta}}}_{1-})  > 0
\;\;\;\;\mathrm{for}\;\mathrm{all}\;
{\boldsymbol{\delta}_1} \in \dtge(Y_1).
\end{equation}

Now, there are two ways in which 
$\pi_2\circ\varphi(\tilde{I}^{\boldsymbol{\delta}_1}_1)$
could contain $\infty$.  One is if $\infty$ is  contained as an endpoint
of $\pi_2\circ\varphi(\tilde{I}^{\boldsymbol{\delta}_1}_1)$,
in which case, since $\pi_2\circ\varphi(\tilde{\boldsymbol{{\delta}}}_{1-}) \neq \infty$,
we must have $\pi_2\circ\varphi(\tilde{\boldsymbol{{\delta}}}_{1+}) = \infty$,
or equivalently,
$a_{1+}^{\boldsymbol{\delta}_1} p_2 - b_{1+}^{\boldsymbol{\delta}_1} q^*_2 = 0$.
Conveniently, the condition
$\pi_2\circ\varphi(\tilde{\boldsymbol{{\delta}}}_{1-})
 \neq
 \pi_2\circ\varphi(\tilde{\boldsymbol{{\delta}}}_{1+})$
also implies that 
$\pi_2\!\circ\!\varphi(\tilde{I}^{\boldsymbol{\delta}_1}_1)$ is closed in this case.
The other possibility is that $\infty$ lies in the interior of
$\pi_2\!\circ\!\varphi(\tilde{I}^{\boldsymbol{\delta}_1}_1)$.
Since $\pi_1^{-1}$, $\varphi$, and $\pi_2$ are each orientation
reversing, this is equivalent to the condition that
$\pi_2\circ\varphi(\tilde{\boldsymbol{{\delta}}}_{1-})
 \le
 \pi_2\circ\varphi(\tilde{\boldsymbol{{\delta}}}_{1+})$,
 which, since $\pi_2\circ\varphi(\tilde{\boldsymbol{{\delta}}}_{1-}) >0$,
 implies $\pi_2\circ\varphi(\tilde{\boldsymbol{{\delta}}}_{1+}) > 0$
 and hence 
$a_{1+}^{\boldsymbol{\delta}_1} p_2 - b_{1+}^{\boldsymbol{\delta}_1} q^*_2 < 0$.
In fact, the converse is also true:
using the substitutions 
$a_{1-}^{\boldsymbol{\delta}_1} = a_{1+}^{\boldsymbol{\delta}_1}\!+q^*_1g_1$
and
$b_{1-}^{\boldsymbol{\delta}_1} = b_{1+}^{\boldsymbol{\delta}_1}\!-p_1g_1$,
and the fact that
$a_{1+}^{\boldsymbol{\delta}_1} p_1 + b_{1+}^{\boldsymbol{\delta}_1} q^*_1 \ge 0$,
it is straightforward to show that the inequalities
$a_{1-}^{\boldsymbol{\delta}_1} p_2 - b_{1-}^{\boldsymbol{\delta}_1} q^*_2 >0$
(from (\ref{eq: -lift pos, with pos denom}))
and $a_{1+}^{\boldsymbol{\delta}_1} p_2 - b_{1+}^{\boldsymbol{\delta}_1} q^*_2 < 0$
imply that 
$\pi_2\circ\varphi(\tilde{\boldsymbol{{\delta}}}_{1-})
 \le
 \pi_2\circ\varphi(\tilde{\boldsymbol{{\delta}}}_{1+})$,
Thus, in summary, ($\si.i$) holds if and only if
$a_{1+}^{\boldsymbol{\delta}_1} p_2 - b_{1+}^{\boldsymbol{\delta}_1} q^*_2 \le 0$
for all ${\boldsymbol{\delta}_1} \in \dtge(Y_1)$,
or equivalently, if and only if 
\begin{equation}
\label{eq: denom+ is less than bq*}
          a_{1+}^{\boldsymbol{{\delta}}_1} p_2 - b_{1+}^{\boldsymbol{{\delta}}_1} q^*_2 
\leq -[b_{1+}^{\boldsymbol{{\delta}}_1} q^*_2]_{p_2}
\;\;\mathrm{for}\;\mathrm{all}\;
{\boldsymbol{\delta}_1} \in \dtge(Y_1),
\end{equation}
which, after substituting
$a_{1+}^{\boldsymbol{{\delta}}_1} 
= ([b_{1+}^{\boldsymbol{{\delta}}_1}q_1^*]_{p_1} 
- b_{1+}^{\boldsymbol{{\delta}}_1}q_1^*)/p_1$
and $q_1^*p_2 + q_2^* p_1 = \bar{q}_1 p_2 + \bar{q}_2 p_1 - p_1 p_2$,
becomes condition ($\textsc{i}.i$).

Thus, conditions ($\si.i$) and ($\si.ii$) are respectively
equivalent to conditions  ($\textsc{i}.i$) and ($\textsc{i}.ii$).  When one or both of
$\dtge(Y_i)$ are empty, ($\textsc{i}.iii$) holds vacuously, and
($\si.i$) and ($\si.ii$) are jointly equivalent to the condition
that  $\varphi_{\P}(I_1) \cup I_2 = \P(H_1(\partial Y_2))$.
We henceforth assume that each $\dtge(Y_i) \neq \emptyset$,
and that conditions ($\textsc{i}.i$) and ($\textsc{i}.ii$),
hence ($\si.i$) and ($\si.ii$), hold.

For each
$({\boldsymbol{\delta}_1}, {\boldsymbol{\delta}_1}) \in \dtge(Y_1) \times \dtge(Y_2)$,
the substitutions
$a_{1+}^{\boldsymbol{\delta}_1} 
= ([b_{1+}^{\boldsymbol{\delta}_1}q_1^*]_{p_1} - b_{1+}^{\boldsymbol{\delta}_1}q_1^*)/p_1$,
$a_{2+}^{\boldsymbol{\delta}_2} = 
([b_{2+}^{\boldsymbol{\delta}_2}q_2^*]_{p_2} - b_{2+}^{\boldsymbol{\delta}_2}q_2^*)/p_2$,
and $q_1^*p_2 + q_2^* p_1 = \bar{q}_1 p_2 + \bar{q}_2 p_1 - p_1 p_2$
make the condition
\begin{equation}
\label{eq: (iii)I in + notation}
\frac{1}{b_{1+}^{\boldsymbol{\delta}_1}}
\!\!\left\lfloor \! 
\frac{b_{1+}^{\boldsymbol{\delta}_1} \bar{q}_1}{p_1}\! \right\rfloor
+ 
\frac{1}{b_{2+}^{\boldsymbol{\delta}_2}}
\!\!\left\lfloor \! 
\frac{b_{2+}^{\boldsymbol{\delta}_2} \bar{q}_2}{p_2}\! \right\rfloor 
\,>\;
1
\end{equation}
equivalent to the inequality
\begin{equation}
a_{1+}^{\boldsymbol{\delta}_1}  b_{2+}^{\boldsymbol{\delta}_2}
+
a_{2+}^{\boldsymbol{\delta}_2} b_{1+}^{\boldsymbol{\delta}_1} 
< 0,
\end{equation}
which, after we multiply by $-p_2$ and add
$b_{2+}^{\boldsymbol{{\delta}}_2}
 (a_{1+}^{\boldsymbol{{\delta}}_1} p_2 - b_{1+}^{\boldsymbol{{\delta}}_1} q^*_2)$
to both sides, becomes
\begin{equation}
\label{eq: almost the inequality between endpoints}
-b_{1+}^{\boldsymbol{{\delta}}_1}
(a_{2+}^{\boldsymbol{{\delta}}_2} p_2 + b_{2+}^{\boldsymbol{{\delta}}_2} q^*_2)
 >
b_{2+}^{\boldsymbol{{\delta}}_2}
 (a_{1+}^{\boldsymbol{{\delta}}_1} p_2 - b_{1+}^{\boldsymbol{{\delta}}_1} q^*_2),
\end{equation}
which, since
$-b_{1+}^{\boldsymbol{{\delta}}_1}
(a_{2+}^{\boldsymbol{{\delta}}_2} p_2 + b_{2+}^{\boldsymbol{{\delta}}_2} q^*_2)
\le 0$,
implies 
$a_{1+}^{\boldsymbol{{\delta}}_1} p_2 - b_{1+}^{\boldsymbol{{\delta}}_1} q^*_2 \neq 0$,
and hence
$\pi_2\circ\varphi(\tilde{\boldsymbol{{\delta}}}_{1+}) \neq \infty$.
Note that when
$\pi_2\circ\varphi(\tilde{\boldsymbol{{\delta}}}_{1+}) \neq \infty$,
condition ($\si.i$) is equivalent to the condition
\begin{equation}
\label{eq: infty in I1, but infty not an endpoint}
(0\, <)\;\; 
\pi_2\circ\varphi(\tilde{\boldsymbol{{\delta}}}_{1-})
\le
\pi_2\circ\varphi(\tilde{\boldsymbol{{\delta}}}_{1+}).
\end{equation}

If $\delta_2 = 0$, then $\tilde{I}^{\boldsymbol{\delta}_2}_2$ is
the complement of $\pi_2^{-1}(\infty)$, and so
(\ref{eq: infty in I1, but infty not an endpoint})
is equivalent to the condition that
$\varphi_{\P}(\dot{I}^{\boldsymbol{\delta}_1}_1) \cup
\dot{I}^{\boldsymbol{\delta}_2}_2 = \P(H_1(\partial Y_2))$,
where 
$\dot{I}^{\boldsymbol{\delta}_i}_i$
denotes the interior of 
$\P(\tilde{I}^{\boldsymbol{\delta}_i}_i)$ for each $i \in \{1,2\}$.
If $\delta_2 > 0$, so that
$\pi_2(\tilde{\boldsymbol{{\delta}}}_{2-})
<
0
<
\pi_2(\tilde{\boldsymbol{{\delta}}}_{2+})$,
then dividing 
(\ref{eq: almost the inequality between endpoints})
by $\delta_2 (a_{1+}^{\boldsymbol{{\delta}}_1} p_2 - b_{1+}^{\boldsymbol{{\delta}}_1} q^*_2)$
makes 
(\ref{eq: almost the inequality between endpoints})
equivalent to the inequality
$\pi_2\circ\varphi(\tilde{\boldsymbol{{\delta}}}_{1+})
<\pi_2(\tilde{\boldsymbol{{\delta}}}_{2+})$, which, combined with
(\ref{eq: infty in I1, but infty not an endpoint}), becomes
\begin{equation}
\pi_2(\tilde{\boldsymbol{{\delta}}}_{2-})< 0 < 
\pi_2\circ\varphi(\tilde{\boldsymbol{{\delta}}}_{1-})
\le
\pi_2\circ\varphi(\tilde{\boldsymbol{{\delta}}}_{1+})
<
\pi_2(\tilde{\boldsymbol{{\delta}}}_{2+}),
\end{equation}
which again is equivalent to the condition that
$\varphi_{\P}(\dot{I}^{\boldsymbol{\delta}_1}_1) \cup
\dot{I}^{\boldsymbol{\delta}_2}_2 = \P(H_1(\partial Y_2))$.
Thus condition ($\textsc{i}.iii$), which takes
(\ref{eq: (iii)I in + notation}) over all
$({\boldsymbol{\delta}_1}, {\boldsymbol{\delta}_1}) \in \dtge(Y_1) \times \dtge(Y_2)$,
is equivalent to the condition that
$\varphi_{\P}(\dot{I}_1) \cup
\dot{I}_2 = \P(H_1(\partial Y_2))$.

\end{proof}

\subsection{Comparison of L-space classification with gluing hypothesis}
Now that we have both classified when $Y_1 \cup_{\varphi} Y_2$
is an L-space, and classified when it satisfies the gluing hypothesis
in terms of the union of the L-space intervals of $Y_1$ and $Y_2$,
it remains to show that these two classifications are equivalent.

\begin{prop}
\label{prop: equivalence of conditions iiL and iiI}
Suppose that $\mu_1$ is ``judiciously chosen'' from
$\P^{-1}(\dot{I}_1 \cap \varphi^{-1}_{\P}(\dot{I}_2))$ nonempty,
and that $Y$ is constructed as above.
For each $i \in \{1,2\}$, set $\bar{q}_i := \left[q_i^*\right]_{p_i}$ and let
$B_i$ denote the set
$B_i := \left\{ \left. [p_i \gamma_i - q_i \delta_i ]_{p_ig_i}\!\right| 
{\boldsymbol{\delta}_i} = \delta_i \iota_i(m_i) + \gamma_i \iota_i(l_i) 
\in \dtge(Y_i)  \right\}$.
Then condition $(\textsc{i}.i)$ (respectively $(\textsc{i}.ii)$) 
from Proposition
\ref{prop: interval covering condition in terms of b1 and b2}
holds if and only if 
condition $(\textsc{l}.i)$ (respectively $(\textsc{l}.ii)$)
from Proposition \ref{prop: L-space condition for Y1 cup Y2 in terms of b}
holds for all $b_1 \in B_1$ (respectively $b_2 \in B_2$).
\end{prop}

\begin{proof}

If $B_1 = \emptyset$, then conditions $(\textsc{i}.i)$ and $(\textsc{i}.ii)$
hold vacuously, hence are equivalent.
We therefore assume $B_1$ is nonempty and fix some $b_1 \in B_1$.
Clearly $(\textsc{l}.i)$ implies the statement of $(\textsc{i}.i)$
for that particular $b_1$, since
$b_1 \in \{b \in \Z \left|\; b \equiv b_1\, (\mod p_1), 
0 < b < p_1g_1 p_2 g_2/\mkern-2mug_0 \right.\}$.

Conversely, suppose $(\textsc{i}.i)$ holds for that $b_1$.
Substituting $q^* = \bar{q}_1 p_2 + \bar{q}_2 p_1 - p_1 p_2$ gives
\begin{equation}
\frac{b_1 q^*}{p_1 p_2}
\ge \frac{[b_1 \bar{q}_1]_{p_1}}{p_1} + \frac{[b_1 \bar{q}_2]_{p_2}}{ p_2}.
\end{equation}
Thus, for any $b := b_1 + y\mkern.5mu p_1\mkern-1mu g_1$
with $y \in \{0, \dots, p_2g_2/\mkern-2mug_0 - 1\}$, we have
\begin{align}
\frac{b q^*}{p_1 p_2}
&\ge
\left(\frac{[b_1 \bar{q}_1]_{p_1}}{p_1} + \frac{[b_1 \bar{q}_2]_{p_2}}{ p_2}\right)
+ \frac{y\mkern.5mu p_1\mkern-1mu g_1 (p_1 \bar{q}_2  -  (p_1 - \bar{q}_1)p_2)}{p_1p_2}
           \\ \nonumber
&\ge
\frac{[b \bar{q}_1]_{p_1}}{p_1} + \left(\frac{[b_1 \bar{q}_2]_{p_2}}{ p_2}
+ \frac{y g_1[ p_1 \bar{q}_2 ]_{p_2} }{p_2} \right)
           \\ \nonumber
&\ge
\frac{[b \bar{q}_1]_{p_1}}{p_1} + \frac{[b \bar{q}_2]_{p_2}}{ p_2},
\end{align}
which is equivalent to the inequality in condition $(\textsc{l}.i)$.
An analogous argument proves the equivalence of
conditions $(\textsc{l}.ii)$ and $(\textsc{i}.ii)$ for any $b_2 \in B_2$.

\end{proof}

\begin{prop}
\label{prop: comparison of property iii}
Suppose that $\mu_1$ is ``judiciously chosen'' from
$\P^{-1}(\dot{I}_1 \cap \varphi^{-1}_{\P}(\dot{I}_2))$ nonempty,
and that $Y$ is constructed as above.
For each $i \in \{1,2\}$, set $\bar{q}_i := \left[q_i^*\right]_{p_i}$ and let
$B_i$ denote the set
$B_i := \left\{ \left. [p_i \gamma_i - q_i \delta_i ]_{p_ig_i}\!\right| 
{\boldsymbol{\delta}_i} = \delta_i \iota_i(m_i) + \gamma_i \iota_i(l_i) 
\in \dtge(Y_i)  \right\}$.
Suppose conditions $(\textsc{i}.i)$ and $(\textsc{i}.ii)$ 
from Proposition
\ref{prop: interval covering condition in terms of b1 and b2}
hold.
Then
condition $(\textsc{i}.iii)$
from Proposition
\ref{prop: interval covering condition in terms of b1 and b2}
holds if and only if
condition $(\textsc{l}.iii)$
from Proposition \ref{prop: L-space condition for Y1 cup Y2 in terms of b}
holds for all $(b_1, b_2) \in B_1 \times B_2$ with $b_1 \equiv b_2\;(\mod g_0)$.
\end{prop}

\begin{proof}

We henceforth assume that $\dtge(Y_1)$ and $\dtge(Y_2)$ are nonempty,
since otherwise conditions $(\textsc{i}.iii)$
and $(\textsc{l}.iii)$ hold vacuously in all cases.

If condition $(\textsc{i}.iii)$ holds, then it holds for any
$(b_1, b_2) \in B_1 \times B_2$ with $b_1 \equiv b_2\;(\mod g_0)$.
In this case, the
unique $b \in \{0, \ldots, p_1p_2 g - 1\}$ satisfying $b \equiv b_1 \,(\mod p_1g_1)$ and $b \equiv b_2 \,(\mod p_2g_2)$
also satisfies $[b\bar{q}_1]_{p_1} = [b_1\bar{q}_1]_{p_1}$ and 
$[b\bar{q}_2]_{p_2} = [b_2\bar{q}_2]_{p_2}$, so that we have
\begin{align}
   \frac{1}{b}\!\left\lfloor \! \frac{b \bar{q}_1}{p_1}\! \right\rfloor
  +\frac{1}{b}\!\left\lfloor \! \frac{b \bar{q}_2}{p_2}\! \right\rfloor
&=
   \frac{1}{b_1\!}\!\left\lfloor \! \frac{b_1 \bar{q}_1}{p_1}\! \right\rfloor
  +\frac{1}{b_2\!}\!\left\lfloor \! \frac{b_2 \bar{q}_2}{p_2}\! \right\rfloor
+ \left(\frac{1}{b_1\!} - \frac{1}{b} \right)\!
\frac{[b_1\bar{q}_1]_{p_1}\!}{p_1}
+ \left(\frac{1}{b_2\!} - \frac{1}{b} \right)\!
\frac{[b_2\bar{q}_2]_{p_2}\!}{p_2}
          \\ \nonumber
&> 1.
\end{align}
Thus $(\textsc{i}.iii)$ implies $(\textsc{l}.iii)$
for all $(b_1, b_2) \in B_1 \times B_2$ with $b_1 \equiv b_2\;(\mod g_0)$,
and it remains to prove the converse.
\smallskip

{{{\bf{Claim.}}  {\it{Suppose that conditions $(\textsc{i}.i)$ and $(\textsc{i}.ii)$
hold, and that there exists some $(b_1, b_2) \in B_1 \times B_2$
for which the statement of $(\textsc{i}.iii)$ fails, or equivalently
(using the substitution $q^* = \bar{q}_1 p_2 + \bar{q}_2 p_1 - p_1 p_2$,
for which
\begin{equation}
\label{eq: counter example for iii}
\frac{q^*}{p_1p_2}
\le 
\frac{[b_1\bar{q}_1]_{p_1\!}}{b_1p_1}
+\frac{[b_2\bar{q}_2]_{p_2\!}}{b_2p_2}.
\end{equation}
Then we have the inequalities
\begin{equation}
\label{eq: claim, b1q1 /b1 > b2q1/b2}
(i)\;\; \;
\frac{[b_1\bar{q}_1]_{p_1\!}}{b_1}
\,\ge\,
\frac{[b_2\bar{q}_1]_{p_1\!}}{b_2},
\;\;\;\;\;\;\;\;\;\;\;\;\;\;\;\;\;\;\;
(ii)\;\; \;
\frac{[b_2\bar{q}_2]_{p_2\!}}{b_2}
\,\ge\,
\frac{[b_1\bar{q}_2]_{p_2\!}}{b_1},\;\;\;\;\;\;\;\;\;\;\;
\end{equation}
and conditions  $(\textsc{i}.i)$ and $(\textsc{i}.ii)$ for this particular 
$(b_1, b_2) \in B_1 \times B_2$
become the equalities
\begin{equation}
\label{eq: (i) and (ii) are equalities}
(i)\;\;\;
\frac{q^*}{p_1p_2}
\,=\,
\frac{[b_1\bar{q}_1]_{p_1\!}}{b_1p_1}
+\frac{[b_1\bar{q}_2]_{p_2\!}}{b_1p_2},
\;\;\;\;\;\;\;\;\;
(ii)\;\;\;
\frac{q^*}{p_1p_2}
\,=\,
\frac{[b_2\bar{q}_1]_{p_1\!}}{b_2p_1}
+\frac{[b_2\bar{q}_2]_{p_2\!}}{b_2p_2}.
\end{equation}
}}}}

{\it{Proof of Claim.}}  Using the substitution
$q^* = \bar{q}_1 p_2 + \bar{q}_2 p_1 - p_1 p_2$,
we can re-express conditions $(\textsc{i}.i)$ and $(\textsc{i}.ii)$ as
\begin{equation}
\label{eq: q* form of conditions i and ii}
(i)\;\; \;
\frac{q^*}{p_1p_2}
\,\ge\,
\frac{[b_1\bar{q}_1]_{p_1\!}}{b_1p_1}
+\frac{[b_1\bar{q}_2]_{p_2\!}}{b_1p_2},
\;\;\;\;\;\;\;\;
(ii)\;\;\;
\frac{q^*}{p_1p_2}
\,\ge\,
\frac{[b_2\bar{q}_1]_{p_1\!}}{b_2p_1}
+\frac{[b_2\bar{q}_2]_{p_2\!}}{b_2p_2}.
\end{equation}
Concatenating
(\ref{eq: counter example for iii})
with (\ref{eq: q* form of conditions i and ii}$.i$)
(respectively, (\ref{eq: q* form of conditions i and ii}$.ii$))
then yields inequality
(\ref{eq: claim, b1q1 /b1 > b2q1/b2}$.ii$)
(respectively, (\ref{eq: claim, b1q1 /b1 > b2q1/b2}$.i$)).
Setting $\delta_1 := [b_1 \bar{q}_1]_{p_1\!} \in \dtge(Y_1)$ and
$\delta_2 := [b_2 \bar{q}_2]_{p_2\!} \in \dtge(Y_2)$, we note that
(\ref{eq: claim, b1q1 /b1 > b2q1/b2}$.i$)
implies
\begin{equation}
\label{eq: lower bound on 1/b1}
\frac{\delta_2}{b_2 p_2} < \frac{1}{b_1},
\end{equation}
since otherwise, applying
(\ref{eq: claim, b1q1 /b1 > b2q1/b2}$.i$) and 
$1/ b_1 \le \delta_2/(b_2 p_2)$ in succession would yield
\begin{equation}
\nonumber
\frac{[b_2 \bar{q}_1]_{p_1\!}}{b_2p_1} 
\,\le\,
\frac{\delta_1}{p_1} \!\cdot\!
\frac{1}{b_1}
\,\le\,
\frac{\delta_1}{p_1} \!\cdot\!
\frac{\delta_2 }{b_2 p_2}
\,=\,
\frac{\delta_1 \delta_2 /p_2}{b_2 p_1}
\,<\,
\frac{(1\mkern-1mu+\mkern-1mu \deg_{t_1} \mkern-5mu\tc(Y_1))
(1 \mkern-1mu+\mkern-1mu \deg_{t_2} \mkern-5mu\tc(Y_2)/p_2}{b_2 p_1}
\,<\, 
\frac{1}{b_2 p_1},
\end{equation}
making $[b_2 \bar{q}_1]_{p_1\!} < 1$, a contradiction.
Thus (\ref{eq: lower bound on 1/b1}) must hold.

Applying 
(\ref{eq: q* form of conditions i and ii}$.i$),
(\ref{eq: counter example for iii}), and
(\ref{eq: lower bound on 1/b1}) in succession, we obtain
\begin{align}
\label{eq: first line squeeze}
\frac{[b_1\bar{q}_1]_{p_1\!}}{b_1p_1}
+\frac{[b_1\bar{q}_2]_{p_2\!}}{b_1p_2}
&\le
\frac{q^*}{p_1 p_2}
    \\ \nonumber
&\le
\frac{[b_1\bar{q}_1]_{p_1\!}}{b_1p_1}
+\frac{[b_2\bar{q}_2]_{p_2\!}}{b_2p_2}
    \\
&<
\frac{[b_1\bar{q}_1]_{p_1\!}}{b_1p_1}
+\frac{1}{b_1}.
\label{eq: last line squeeze}
\end{align}
Subtracting $\frac{[b_1\bar{q}_1]_{p_1\!}}{b_1p_1}
+\frac{[b_1\bar{q}_2]_{p_2\!}}{b_1p_2}$ from
lines
(\ref{eq: first line squeeze}) and
(\ref{eq: last line squeeze}) then yields
\begin{equation}
0
\;\le\;
\frac{q^*}{p_1 p_2}
-\left(\frac{[b_1\bar{q}_1]_{p_1\!}}{b_1p_1}
+\frac{[b_1\bar{q}_2]_{p_2\!}}{b_1p_2}\right)
    \\
\;<\;
\frac{1}{b_1} - \frac{[b_1\bar{q}_2]_{p_2\!}}{b_1p_2},
\end{equation}
but we also know that
\begin{equation}
\frac{q^*}{p_1 p_2}
-\left(\frac{[b_1\bar{q}_1]_{p_1\!}}{b_1p_1}
+\frac{[b_1\bar{q}_2]_{p_2\!}}{b_1p_2}\right)
\; \in\; \frac{1}{b_1}\Z.
\end{equation}
Thus, (\ref{eq: (i) and (ii) are equalities}$.i$) must hold,
and (\ref{eq: (i) and (ii) are equalities}$.ii$)
follows from symmetry, proving our Claim.

Having proven our Claim, we pause to introduce the notation 
$b_i \mapsto \boldsymbol{\delta}_{\mkern-1.5mu i}^{b_i}$ for the bijection
\begin{align}
\{0, \ldots, p_ig_i-1\} 
&\;\to\;
\{0, \ldots, p_i - 1\} \iota_i(m_i) + T^{\partial}_i,
    \\ \nonumber
b_i 
&\;\mapsto\; \boldsymbol{\delta}_{\mkern-1.5mu i}^{b_i}
:= 
\iota_i(-\!\left\lfloor\!\frac{b_iq_i^*\!}{p_i}\!\right\rfloor\! \mu_i 
\,+\, b_i \lambda_i) 
\;\;\in\;
[b_i \bar{q}_i]_{p_i} \iota_i(m_i) + T^{\partial}_i,
\end{align}
whose inverse we used to define each $B_i$ as a set of integers indexing the elements of
$\dtge(Y_i)$.
\smallskip

We now proceed with an inductive argument.
Suppose that
$(\textsc{l}.iii)$ holds for all $(b_1, b_2) \in B_1 \times B_2$ satisfying
$b_1 \equiv b_2\, (\mod g_0)$,
and that
$(\textsc{i}.i)$ and $(\textsc{i}.ii)$ hold, but that there exist
$b_i \in B_i$ and $b_I \in B_I$, with $\{i, I\} = \{1,2\}$ and $b_i \le b_I$,
for which $(\textsc{i}.iii)$ fails, {\it{i.e.}}, for which
\begin{equation}
\label{eq: counter example for iii, but with i and I}
\frac{q^*}{p_1p_2}
\le 
\frac{[b_i\bar{q}_i]_{p_i\!}}{b_ip_i}
+\frac{[b_I\bar{q}_I]_{p_I\!}}{b_Ip_I}.
\end{equation}
Equation
(\ref{eq: (i) and (ii) are equalities}) from our Claim then tells us that
\begin{equation}
\label{eq: cond iI= as counter example to cond iiiL}
 \frac{1}{b_i}\!\left\lfloor \! \frac{b_i \bar{q}_i}{p_i}\! \right\rfloor
+\frac{1}{b_i}\!\left\lfloor \! \frac{b_i \bar{q}_I}{p_I}\! \right\rfloor = 1.
\end{equation}
This means that $b_i \notin B_I$, since otherwise,
setting $b:= b_i \in B_i \cap B_I =  B_1 \cap B_2$ would make
(\ref{eq: cond iI= as counter example to cond iiiL}) contradict
condition (${\textsc{l}}.iii$).  
Thus, 
$\boldsymbol{\delta}_{\mkern-2.5mu I}^{b_i}
\notin 
\dtge(Y_I)$
and $b_i < b_I$.

We next apply 
(\ref{eq: claim, b1q1 /b1 > b2q1/b2}) from our Claim, to obtain
\begin{equation}
\label{eq: deltaI > deltai}
[b_i \bar{q}_I]_{p_I} \le \frac{b_i}{b_I}[b_I \bar{q}_I]_{p_I} < [b_I \bar{q}_I]_{p_I}.
\end{equation}
Since 
$\boldsymbol{\delta}_{\mkern-1.5mu I}^{b_I} - \boldsymbol{\delta}_{\mkern-1.5mu I}^{b_i}
\in \left([b_I \bar{q}_I]_{p_I} - [b_i \bar{q}_I]_{p_I}\right) \iota_I(m_I)
+ T^{\partial}_I$,
the above inequality implies
$\boldsymbol{\delta}_{\mkern-1.5mu I}^{b_I} - \boldsymbol{\delta}_{\mkern-1.5mu I}^{b_i}
\in \iota_I(m_I \Z_{\ge 0} + l_I\Z)$. 
Thus, since 
$\boldsymbol{\delta}_{\mkern-1.5mu I}^{b_i}
\notin 
\dtge(Y_I)$ and
$\boldsymbol{\delta}_{\mkern-1.5mu I}^{b_I}
\in 
\dtge(Y_I)$,
the additive closure of 
$\iota_I(m_I \Z_{\ge 0} + l_I\Z) \setminus \dtge(Y_I)$ from
Proposition \ref{prop: The complement of D is additively closed.}
tells us that
$\boldsymbol{\delta}_{\mkern-1.5mu I}^{b_I} - \boldsymbol{\delta}_{\mkern-1.5mu I}^{b_i}
\in \dtge(Y_I)$.
Since (\ref{eq: deltaI > deltai}) implies
$\left([b_I \bar{q}_I]_{p_I} - [b_i \bar{q}_I]_{p_I}\right) = [(b_I - b_i) \bar{q}_I]_{p_I}$,
we actually have 
$\boldsymbol{\delta}_{\mkern-1.5mu I}^{b_I} - \boldsymbol{\delta}_{\mkern-1.5mu I}^{b_i}
= \boldsymbol{\delta}_{\mkern-1.5mu I}^{b_I - b_i} \in \dtge(Y_I)$,
implying $b_I - b_i \in B_i$.
We furthermore have
\begin{equation}
\frac{[b_i \bar{q}_I]_{p_I}}{b_i}
\le \frac{[b_I \bar{q}_I]_{p_I}}{b_I}
 \;\;\implies\;\;
\frac{[b_I \bar{q}_I]_{p_I}}{b_I}
\le
\frac{[(b_I - b_i) \bar{q}_I]_{p_I}}{b_I-b_i},
 \end{equation}
so that (\ref{eq: counter example for iii, but with i and I}) implies
\begin{equation}
\frac{q^*}{p_1p_2}
\le 
\frac{[b_i\bar{q}_i]_{p_i\!}}{b_ip_i}
+\frac{[(b_I-b_i)\bar{q}_I]_{p_I\!}}{(b_I-b_i)p_I},
\end{equation}
with $b_i \in B_i$ and $b_I - b_i \in B_I$, mimicking our initial conditions.

We then iterate the process, at each iteration redefining $i, I \in \{1,2\}$, $b_i$, and $b_I$ so that
\begin{equation}
b^{\mathrm{new}}_i := \min\{b_i^{\mathrm{old}},\, b_I^{\mathrm{old}}-b_i^{\mathrm{old}}\},
\;\;\;\;\;\;\;
b^{\mathrm{new}}_I := \max\{b_i^{\mathrm{old}},\, b_I^{\mathrm{old}}-b_i^{\mathrm{old}}\}.
\end{equation}
Like any Euclidean Algorithm, this strictly decreasing sequence bounded 
by zero must terminate at zero, with its last two nonzero entries equal to
\begin{equation}
b^{\mathrm{final}}_i = b^{\mathrm{final}}_I = \gcd(b_i^{\mathrm{original}}, b_I^{\mathrm{original}}).
\end{equation}
Setting $b := b^{\mathrm{final}}_i = b^{\mathrm{final}}_I  \in B_1 \cap B_2$ then
makes
(\ref{eq: cond iI= as counter example to cond iiiL})
contradict condition $({\textsc{l}}.iii)$,

This completes the proof of
the proposition,
thereby
completing the proof of
Theorem~\ref{thm: torus gluing and complementary intervals for L-spaces}

\end{proof}

\section{Generalized solid tori and NLS Detection}
\label{Sec:dtgz}
\
In this section, we study manifolds with \(\dtgz= \emptyset\). Unless otherwise specified, we assume that \(Y\) is a rational homology \(S^1 \times D^2\) with \(H_1(Y) =\Z\oplus T\), and that \(\phi:H_1(Y)\to H_1(Y)/T\simeq \Z\) is the projection. We define \(g_Y>0\) by the relation \(\im \phi = g_Y \Z \subset \Z\). 
 The number \(g_Y\) is the minimal intersection number of a curve on \(\partial Y\) with a surface generating \(H_2(Y, \partial Y)\). Equivalently, it is the minimal number of boundary components of such a surface, or the order of the homological longitude \(l\) in \(H_1(Y)\). Finally, we define 
\(k_Y\) to be the order of the group \(T/(T\cap \im \iota)\), so that \(|T|=k_Yg_Y\).
 
 \subsection{Generalized solid tori}
The Seifert fibred spaces \(N_g=M(\emptyset;1/g,-1/g)\)  provide a motivating example of a class of manifolds with \(\lL(N_g) = Sl(N_g)/[l]\). They were studied in \cite{BGW} (for \(g=2\)) and subsequently by Watson \cite{WatsonFST} for arbitrary values of \(g\). We briefly describe them  here. 
First, we have 
 $$H_1(N_g) = \langle f, h_1, h_2 \, | \, f+ gh_1 = f-gh_2 = 0 \rangle \simeq \Z \oplus \Z/g.$$
 The \(\Z\) summand is generated by \(h_1\), and the \(\Z/g\) summand is generated by \(\sigma = h_1+h_2\). 
 \(H_1(\partial N_g) = \langle f,\sigma \rangle\), so \(\iota (H_1(\partial N_g)) = g \Z \oplus \Z/g \subset H_1(Y_g)\). 
 The Turaev torsion is 
 $$ \tau(N_g) \sim \frac{1-[f]}{(1-[h_1])(1-[h_2])} = \frac{1-t^g}{(1-t)(1-t\sigma)}$$
 so the Milnor torsion is 
$$\taubar(Y) = \tau(N_g)|_{\sigma=1} = \frac{1-t^g}{(1-t)^2} = 1 + 2t + 3t^2+\ldots +(g-1)t^{g-1}+gt^g+gt^{g+1}+\ldots$$ 
 It is easy to see that if \(x \not \in S[\tau(N_g)]\), \(y \in S[\tau(N_g)]\) with \(\phi(x) > \phi(y)\), then 
 \(\phi(x-y) <g\). If \(x-y \in \im \iota\), we must have \(\phi(x-y)=0\), so \(\dtgz(N_g) = \emptyset\). 
 More generally, the same argument shows that 
 
 \begin{prop} 
 \label{Prop:dtgz}
 If \(Y\) is a Floer simple and \(\deg \Deltas(Y) < g_Y\), then \(\dtgz(Y) = \emptyset\). 
 \end{prop}
 Motivated by this, we make the following 
 \begin{definition}
 A {\em generalized solid torus} is a Floer simple manifold \(Y\) with \(\deg \Deltas(Y) < g_Y\). 
  \end{definition}
 
 If \(Y\) is such a manifold, Corollary~\ref{Cor:TNorm} implies that \(\| Y\| \leq g_Y-2\). On the other hand, an embedded surface which generates \(H_2(Y,\partial Y)\) has at least \(g_Y\) boundary components, so a norm-minimizing surface must have genus \(0\). 
 
 The Milnor torsion of a generalized solid torus is  determined by \(g_Y\) and \(k_Y\). 
 \begin{lemma}
 \label{Lem:RedAlex}
 Suppose that \(Y\) is a rational homology \(S^1\times D^2\). If \(p:\Z[t]\to \Z[t]/(t^{g_Y}-1)\) is the projection, then
 $p(\Deltas(Y)) = k_Y(1-t^{g_Y})/(1-t).$
 \end{lemma}
 
\begin{proof}
The usual product formula for the torsion implies that 
$$ \tau(Y(l)) = j_{1*}(\tau(S^1\times D^2)j_{2*}(\tau(Y))$$
where \(j_1:S^1 \times D^2 \to Y(l)\) and \(j_2:Y \to Y(l)\) are the inclusions. 
It follows that 
$$\taubar(Y(l)) = \frac{\taubar(Y)}{1-t^{g_Y}}.$$
By \cite{Turaev}, Lemma 3.2, we have 
$$\taubar(Y(l)) = \frac{t^c|H_1(Y(l)|}{(1-t)^2} + P(t)$$
where \(c \in \Z\) and \( p(t) \in \Z[t^{\pm1}]\). \(|H_1(Y(l))| = k_Y\). 
Combining the two formulas, we see that 
$$ \Deltas(Y) = \frac{k_Y t^c(1-t^{g_Y})}{1-t} + (1-t^{g_Y})(1-t)P(t).$$
\end{proof}
 
 \noindent Combining the lemma with the requirement that  \(\deg \Deltas(Y) < g_Y\) gives
 
\begin{cor}
  If \(Y\) is a generalized solid torus, \(\Deltas(Y) \sim k_Y(1-t^{g_Y})/(1-t)\).
  \end{cor}
  
 In contrast, \(\tau(Y)\) is not determined by the fact that \(Y\) is a generalized solid torus, as can be seen by considering the Seifert-fibred spaces \(M(\emptyset; a/g,-a/g)\).

 \begin{prop}
 \label{Prop:FakeFloerFull}
 A generalized solid torus is a  Floer homology solid torus in the sense of Watson \cite{WatsonFST}. 
\end{prop}

\begin{proof}
Let \(g=g_Y\). 
Recall that \(Y\) is a Floer homology solid torus if 
\(\cfd(Y,m,l) \simeq \cfd(Y,m+l,l)\), where \(l\) is the canonical longitude and 
\(m \cdot l = 1\). By composing with an appropriate change of basis bimodule, we see that this is equivalent to saying that for some \(\mu, \lambda\) with \(\mu \cdot \lambda = 1\), we have 
\(\cfd(Y,\mu, \lambda) \simeq \cfd(Y,\tau_{l}(\mu), \tau_{l}(\lambda))\), where
\(\tau_l\) is the Dehn twist along \(l\). 

Suppose that \(Y\) is a generalized solid torus.  By Proposition~\ref{Prop:FloerSimpleCFD}, we can explicitly compute \(\cfd(Y, \mu,\lambda)\) for an appropriate choice of \(\mu\) and \(\lambda\). In fact, \(\cfd(\mu,\lambda)\) is determined by the polynomials \(\chi(\hfk(K_\mu))\) and \(\chi(\hfk(K_\lambda))\), which are in turn determined by \(\Delta(Y)\), \(\iota(\mu)\), and \(\iota(\lambda)\). Since \(\|Y\| = g-2\), the criteria of Proposition~\ref{Prop:FloerSimpleCFD} will be satisfied if we take  
\(\mu = m\) and \(\lambda = l - N m\), where \(N\gg 0\). 

Let \(S_\mu \subset H_1(Y)\) be the support of \(\hfk(K_\mu)\), normalized so that  if \(x \in S_\mu\), then 
\(0 \leq \phi(\mu)\leq 2g-2\). \(S_\mu\) is determined by the conditions that for
\(0 \leq \phi(\mu)\leq g-1\), \(x \in S_\mu\) if and only if \(x \in S[\tau(Y)]\), and for \(g-1 \leq \phi(\mu) \leq 2g-2\), \(x \in S[\mu]\) if and only if \(x- \mu \not \in S[\tau(Y)]\). 

Similarly, let \(S_\lambda \subset H_1(Y)\) be the support of \(\hfk(K_\lambda)\), normalized so that  if \(x \in S_\lambda\), then 
\(0 \leq \phi(\lambda)\leq (N+1)g-2\). \(S_\lambda\) is determined by the conditions that for
\(0 \leq \phi(x)\leq g-1\), \(x \in S_\lambda\) if and only if \(x \in S[\tau(Y)]\), and for \(g-1 \leq \phi(x) \leq (N+1)g-2\), \(x \in S_\lambda\) if and only if \(x+ \lambda \not \in S[\tau(Y)]\). (Note that \(\phi(\lambda)<0\), so we need \(x+ \lambda\) here rather than \(x-\lambda\)).

Now let \(\mu' = \tau_l (m) = \mu+l\) and \(\lambda' = \tau_l(\lambda) = \lambda - N l\).
The supports \(S_{\mu'}\) and \(S_{\lambda'}\) can be described similarly.
 
 We define an isomorphism  
 \(f:\cfd(Y\mkern-5mu,\mu,\lambda) 
 \to \cfd(Y\mkern-5mu,\mu'\!, \lambda')\). The map 
 \(f:\hfk(K_\mu) \to \hfk(K_{\mu'})\) is given as follows. If \(x \in S_\mu\), then \(f\) takes the unique nonzero element of \(\hfk(K_\mu)\)  supported at \(x\) to the unique nonzero element of \(\hfk(K_{\mu'})\) supported at \(x + \lfloor\phi(x)/g\rfloor l\). Using the description of the sets \(S_\mu\) and \(S_\mu'\) given above, together with the fact that \(\phi(\mu) = g\), it is easy to see that \(f\) is a bijection. Similarly, if \(x \in S_\lambda\), we define \(f\) to take the unique nonzero element supported at \(x\) to the unique nonzero element of \(\hfk(K_{\lambda'})\) supported at \(x + \lfloor\phi(x)/g\rfloor l\).
 
 It remains to check that \(f\) carries the arrows in the diagram for \(C=\cfd(Y,\mu, \lambda)\) to the arrows in the diagram for \(C'=\cfd(Y,\mu',\lambda')\). Suppose \(x\) and \(y\) are the initial and terminal ends of an arrow of type \(D_{23}\) in \(C\), so that \(y-x = \mu\). Then \(\phi(y) - \phi(x) = g\), so \(f(y)-f(x) = \mu + l = \mu'\), so \(f(y)\) and \(f(x)\) are the endpoints of an arrow of type \(D_{23}\) in \(C'\). A very similar argument shows that arrows of types \(D_1\) and \(D_3\) are preserved as well. 
\end{proof}

We can prove a partial converse to Proposition~\ref{Prop:dtgz}. Recall that \(Y\) is said to be semi-primitive if \(T \subset \im \iota\). Equivalently, \(Y\) is semi-primitive if \(k_Y=1\).  
 
\begin{prop}
\label{Prop:LSemi}
Suppose that \(Y\) is semi-primitive and Floer simple.
If \(\dtgz(Y)=\emptyset\), then  \(Y\) is a generalized solid torus. 
\end{prop}

\begin{proof}
Let \(g=g_Y\). 
Since \(Y\) is semiprimitive, we  have \(H_1(Y) = \Z \oplus (\Z/g)\) and also
 \(\im \iota =  g\Z  \oplus \Z/g \subset H_1(Y)\). 
Let \(t,\sigma\) be generators of the \(\Z\) and \(\Z/g\) summands respectively, so that 
\(\tau(Y) = \sum_{i=0}^\infty q_i(\sigma) t^i\), where \(q_i(\sigma)\) is a sum of powers of \(\sigma\). 
Suppose that for some value of \(i\), \(q_i(1)<g\) and \(q_{i-g}(1)>0\). Then we can find \(x \not \in S[\tau(Y)]\) with \(\phi(x)=i\) and  \(y \in S[\tau(Y)]\) with \(\phi(y)=i-g\).  It follows that \(x-y \in  \im \iota\), which contradicts  \(\dtgz(Y) = 0\). We conclude that for a fixed value of \(k\) there is at most one value of \(n\) for which \(q_{k+ng}(1) \neq 0,g\). 

The Milnor torsion of \(Y\) is 
$\taubar(Y)=\displaystyle \Deltas(Y)/(1-t) = \sum_{i=0}^\infty a_i t^i$,
where \(a_i\)=\(q_i(1)\).
\begin{lemma}
There is a constant \(c\) so that \(\displaystyle \sum_{i \equiv k \, (g)} a_i \equiv k + c \, (g)\). 
\end{lemma}
Note that all but finitely many of the \(a_i\) are equal to either \(0\) or \(g\), so the sum is well defined. 
\begin{proof}
We say that \(f(t) \in \Z[t]\) has property (*) if the statement of the corollary holds for \(a_i\) given by 
\(f(t)/(1-t) = \sum_{i=0}^\infty a_i t^i. \)
It is easy to see that \(f(t) = 1+t+\ldots +t^{g-1}\) has property (*), and that if \(f(t)\) has property (*), then so do \(f(t)+t^i - t^{g+i}\) and \(t^c f(t)\). Lemma~\ref{Lem:RedAlex} implies that  \(\Delta(Y)\) can be obtained from \(1+t+\ldots + t^{g-1}\) by a sequence of operations of the first type plus a single operation of the second type, so \(\Delta(Y)\) has property (*). 
\end{proof}

The lemma implies that after an appropriate shift in the indexing of the \(a_i\)'s (so that \(\taubar(Y)\) is no longer constrained to to have \(t^0\) as its lowest order term) the subsequence \((a_{k+ng})\) has the form 
\(\ldots, 0, 0,0, k, g,g,g \ldots\), where \(0 \leq k \leq g\). In other words, each subsequence is determined up to a global shift, and it remains to see how these shifts fit together. 

We claim that the sequence \((a_i)\) has the form \(\ldots 0,0,0,1,2,\ldots,g-1,g,g,g \ldots\).
Equivalently, 
 \begin{equation*}
 \taubar(Y) \sim \taubar_0 = t+2t^2 + \ldots +(g-1)t^{g-1}+gt^{g}+gt^{g+1}+\ldots 
  = \frac{t(1-t^g)}{(1-t)^2}
 \end{equation*}
To see this, let us say that \(Q(t) \in \Z[t^{-1},t]]\) is obtained from \(P(t)\) by an elementary shift if \(Q(t) - P(t) = a t^i + (g-a)t^{i+g}\) for some \(a,i \in \Z\). We have shown above that \(\taubar(Y)\) is obtained from \(\taubar_0\) by a sequence of elementary shifts. 

Next, we consider the effect of an elementary shift on the Alexander polynomial. If \(Q(t) \in \Z[t^{-1},t]]\), let 
\(F(Q(t)) = p((1-t)Q(t))\), where \(p:\Z[t] \to \Z[t]/(t^g-1)\) is the projection, so that \(F(\taubar_0) = 1+\ldots +t^{g-1}\). An easy calculation shows that if 
\(Q(t) - P(t) = a t^i + (g-a)t^{i+g}\), then \(F(Q(t)) - F(P(t)) = gt^i-gt^{i+1}\). It follows that if 
\(Q(t)\) is obtained from \(\taubar_0\) by a sequence of elementary shifts and \(F(Q(t)) = F(\taubar_0)\), then \(Q(t)\) is obtained from \(\taubar_0\) by a global shift; that is, each residue class is shifted by the same number of elementary shifts. To sum up, we have proved that \(\taubar(Y) \sim \taubar_0\), so \(Y\) is a generalized solid torus. \end{proof}

As we observed above, if \(Y\) is a generalized solid torus,  \(H_2(Y,\partial Y)\) is generated by a surface of genus \(0\).  It follows that \(Y(l) = Z \# (S^1 \times S^2)\), where \(Z\) is a rational homology sphere. Conversely, we have 

\begin{prop}
\label{Prop:LS1S2}
Suppose that \(K \subset Z \# (S^1\times S^2)\) has an L-space surgery. Then the complement of \(K\) is a generalized solid torus. 
\end{prop}

\begin{proof}

 We use the exact triangle with twisted coefficients, as formulated by Ai and Peters in \cite{AiPeters}. We briefly recall their statement. Given a class \(\eta \in H_1(Y)\) and \(\mu \in Sl(Y)\),  we can form \(\omega_\mu = PD(j_*(\eta)) \in H^2(Y(\mu))\), where \(j:Y \to Y(\mu)\) is the inclusion. The twisted Floer homology  \(\hfhat(Y(\mu); \Lambda_{\omega_\mu})\) is a module over the universal Novikov ring
 \begin{equation*}
 \Lambda = \left\{\sum a_r t^r \, | \,  r\in \R, a_r \in \Z, \#\{r <C\, | \, a_r \neq 0\} < \infty \ \text{for all} \ C \in \R\right\}.
 \end{equation*}
 If the image of \(\omega_\mu\) in \(H^2(Y(\mu), \R)\) is \(0\), then \(\hfhat(Y(\mu); \Lambda_{\omega_\mu}) = \hfhat(Y(\mu)) \otimes \Lambda\).
 Ai and Peters show that if \(\mu \cdot \lambda = 1\), there is a long exact sequence 
 \begin{equation*}
\to  \hfhat(Y(\mu); \Lambda_{\omega_\mu}) \to  \hfhat(Y(\lambda); \Lambda_{\omega_\lambda}) \to  \hfhat(Y(\mu+\lambda); \Lambda_{\omega_{\mu+ \lambda}}) \to \hfhat(Y(\mu); \Lambda_{\omega_\mu}) \to .
 \end{equation*}
 
 Let \(Y\) be the complement of \(K\), so \(Y(l) = Z \# (S^1 \times S^2)\).
 Choose \(\eta \in H_1(Y)\) with \(\phi(\eta) = 1\), so that \(\omega_l\) generates \(H_2(Y(l)) = \Z\). By \cite{AiPeters} Proposition 2.2,  \(\hfhat(Y(l); \Lambda_{\omega_\lambda})=0\). 
 
 Now suppose there is some 
 \(m\) with \(m \cdot l = 1\) and \(m \in \lL(Y)\). In this case  \(H^2(Y(m); \R) \simeq H^2(Y(m+l);\R) = 0\). 
 The exact triangle shows that
 \(\hfhat(Y(m))\otimes \Lambda \simeq \hfhat(Y(m+l))\otimes \Lambda\), which implies that 
  \(\hfhat(Y(m)) \simeq \hfhat(Y(m+l))\). Since \(H_1(Y(m)) \simeq H_1(Y(m+l))\), it follows that 
  \(m+l \in \lL(Y)\). Repeating, we find that \(m+nl \in \lL(Y)\) for all \(n>0\), and thus that \(l\) is a limit point of \(\lL(Y)\). It follows that \(Y\) is Floer simple and \(\dtgz(Y) =  \emptyset\). 
  
  For the general case, suppose that \(\mu \in \lL(Y)\). Then  \(Y(l)\) is obtained by integer surgery on 
  \(K_\mu \#K_{-q/p} \subset Y(\mu)\# L(q,-p)\) for an appropriate choice of \(p\) and \(q\). Let \(Y'\) be the complement of this knot. The argument above shows that every non-longitudinal filling of \(Y'\) is an L-space. An infinite family of these fillings are also obtained by Dehn filling on \(Y\), so \(Y\) is Floer simple. 
  
  To conclude the argument, we compute \(\taubar(Y)\). 
 Let  \(j_1:Y \to Y(l)\) and \(j_2:S^1 \times D^2 \to Y(l)\) be the inclusions. The usual product formula for the  torsion says that
  \begin{equation*}
  \taubar(Y(l))  \sim j_{1*}(\taubar(Y)) j_{2*}(\taubar(S^1 \times D^2)) .
  \end{equation*}
  Here
  \begin{equation*}
  \taubar(Y(l)) = \taubar(Z \# (S^1\times S^2)) \sim \frac{|H_1(Z)|}{(1-t)^2}. 
  \end{equation*}
  It is easy to see that the map \(j_{1*}:H_1(Y)/Tors \to H_1(Y(l))/Tors\) is an isomorphism, while the map 
 \(j_{2*}:H_1(S^1 \times D^2) \to H^1(Y(l))/Tors\) is multiplication by \(g\),
 so 
   \begin{equation*}
  \frac{|H_1(Z)|}{(1-t)^2} \sim \frac{ \taubar(Y)} {1-t^g} .
  \end{equation*}
  Equivalently
  \begin{equation*}
  \taubar(Y)  \sim |H_1(Z)|\frac{1-t^g}{(1-t)^2}.
  \end{equation*}
  It follows that \(Y\) is a generalized solid torus. 
 \end{proof}
 
  Proposition~\ref{Prop:FST} from the introduction is an immediate consequence of Propositions~\ref{Prop:LSemi} and \ref{Prop:LS1S2}, and Proposition~\ref{Prop:Fake=Floer} follows from Proposition~\ref{Prop:FakeFloerFull}.

\subsection{NLS Detection}
\label{SubSec:NLS}

Next, we study  the notion of NLS detection introduced by Boyer and Clay in \cite{BoyerClay}. Suppose that \(Y_1\) is a rational homology solid torus and that \(Y_2\) is a semi-primitive generalized solid torus. Given a primitive class \(\alpha \in H_1(Y_1)\), choose an orientation reversing homeomorphism \(\varphi:\partial Y_1 \to \partial Y_2\) with \(\varphi_*(\alpha) = l\), where \(l \in H_1(\partial Y_2)\) is the homological longitude. Since \(Y_2\) is a Floer homology solid torus, \(\hfhat(Y_\varphi)\) is well defined, in the sense that any \(\phi\) satisfying \(\varphi_*(\alpha) = l\) will give the same result. We say that \(\alpha\)  is  NLS detected by \(Y_2\) if \(Y_\varphi\) is not an L-space. 

If \(Y_1\) is Floer simple, it follows from  Theorem~\ref{Thm:Splice}  that \(\alpha\) is NLS detected by \(Y_2\) if and only if \(\alpha\) is not in the interior of \(\lL(Y)\). In fact, there is a direct proof of this fact for any \(Y_1\). 

\begin{prop}
The slope  \(\alpha\) is  NLS detected by \(Y_2\) if and only if \(\alpha\) 
is not in the interior of \(\lL(Y)\). 
\end{prop}

\begin{proof}
Suppose that \(\alpha\) is not NLS detected by \(Y_2\). Then \(Y_{\varphi_i}\) is an L-space for every \(\varphi_i\) with \(\varphi_{i*}(\alpha) = l\). The manifolds \(Y_{\varphi_i}\) are all obtained by Dehn filling  a manifold \(Y'\) which is constructed by identifying \(\nu(\alpha) \subset \partial Y_1\) with \(\nu(l) \subset \partial Y_2\), as in the proof of Lemma~\ref{Lem:SpliceSurgery}.  It follows that \(Y'\) is Floer simple. 

Let \(\mu \in H_1(\partial Y')\) be the class which represents the common image of \(\alpha \in H_1(\partial Y_1)\) and \(l \in H_1(\partial Y_2)\). The sutured manifold  \((Y',\gamma_{\mu})\) contains an essential annulus \(A\) which separates \(Y_1\)  from \(Y_2\). The boundary of \(A\) is a pair of curves parallel to \(\mu\). We choose the position of the sutures so that one component of \(\partial A\) lies in \(R_+(\gamma_\mu)\)  and the other component is in \(R_-(\gamma_\mu)\). Decomposing \((Y',\gamma_{\mu})\) along \(A\) gives a new sutured manifold which is the disjoint union of \((Y_1, \gamma_\alpha)\) and \((Y_2,\gamma_l)\). \(A\) is a product annulus, so it follows from Lemma 8.9 of \cite{Juhasz} that 
\begin{equation*}
SFH(Y',\gamma_\mu) = SFH(Y_1, \gamma_\alpha) \otimes SFH(Y_2,\gamma_l).
\end{equation*}

Since \(\partial Y' = T^2\), there is a natural injection \(c:{\mathrm{Spin}^c}(Y',\gamma_\mu) \to H_1(Y')\) given by the formula
 \(j(\spi)=P D(c_1(\spi))\), and similarly for \(Y_1\) and \(Y_2\).
The tensor product respects the decomposition into \({\mathrm{Spin}^c}\) structures in the sense that 
\(c(x \otimes y) = j_{1*}(c(x)) + j_{2*}(c(y))\), where \(j_i:Y_i \to Y'\) is the inclusion.

In the case at hand, $
H_1(Y') = H_1(Y_1)\oplus H_1(Y_2)/\langle \alpha = l\rangle $,
and \(H_1(Y_2) \simeq \Z \oplus \Z/g_{Y_2}\), where the \(\Z/g_{Y_2}\) summand is generated by \(l\).  
Thus \(H_1(Y') \simeq \Z \oplus (H_1(Y_1)/\langle g_{Y_2} \alpha \rangle)\). Now
\(\alpha\) is a nontorsion element of \(H_1(Y_1)\) (otherwise \(Y_\varphi\) is not a rational homology sphere), so the image of \(j_{1*}\) is contained in the torsion subgroup of \(H_1(Y')\). 

If \(Y\) is a rational homology \(S^1\times D^2\) and \(\beta \in Sl(Y)\), then 
\(SFH(Y,\gamma_\beta, \spi) = 0 \) whenever \(\phi(c(\spi)) > \| Y\| + |\phi(\beta)|\). The set
\(O_{(Y,\gamma_\beta)} = \{\spi \in {\mathrm{Spin}^c}(Y, \gamma_\beta) \, | \, \phi(c(\spi)) = \| Y\| + |\phi(\beta)| \}\) is the set of {\em outer} \({\mathrm{Spin}^c}\) structures for \((Y, \gamma_\beta)\) \cite{Juhasz}. We write
$$SFH(Y, \gamma_\beta,O) = \bigoplus_{\spi \in O_{(Y,\gamma_\beta)}} SFH(Y,\gamma_\beta,\spi).$$
Since the image of \(j_{1*}\) is contained in the torsion subgroup, we have 
\begin{equation}
\label{Eq:OuterSpinc}
SFH(Y', \gamma_\mu, O) \simeq SFH(Y_1,\gamma_\alpha) \otimes SFH(Y_2,\gamma_l,O).
\end{equation}
In particular, \(\|Y'\| = \|Y_2\| = g_{Y_2}-2=g_{Y'}-2\), so \(Y'\) is a generalized solid torus.

To conclude the proof we  use the following two lemmas. The first is probably well-known, but we give a proof just in case. 
\begin{lemma}
Suppose \(Y\) is an incompressible rational homology \(S^1 \times D^2\), that \(l \in H_1(\partial Y)\) is the homological longitude, and that \(m \cdot l = 1\). Then $$SFH(Y,\gamma_l,O) \simeq SFH(Y, \gamma_m,O) \otimes H_*(S^1).$$
\end{lemma}

\begin{proof}
Let \(S \subset Y\) be a properly embedded surface generating \(H_2(Y, \partial Y)\). If we decompose \((Y, \gamma_m)\) along \(S\), we get a sutured manifold \((Z, \gamma_Z)\), where \(\partial Z\) is a union of two copies of \(S\) glued together their boundaries, and there is  one suture for each component of \(\partial S\). 
Decomposing \((Y, \gamma_l)\) along \(S\) gives \((Z, \gamma_Z')\), where the suture \(\gamma_Z'\) is the same as \(\gamma_Z\) except that there are three parallel sutures along one component of \(\partial S\) instead of one. By Proposition 9.2 of \cite{Juhasz}, \(SFH(Z, \gamma_Z') \simeq SFH(Z, \gamma_Z)\otimes H_*(S^1)\).
\end{proof}

\begin{lemma} If \(Y\) is a generalized solid torus and \(m \in H_1(\partial Y)\) satisfies \(\phi(m) = g_Y\), 
 then \(SFH(Y,\gamma_m,O) \simeq \Z^{k_Y}\).
\end{lemma}

\begin{proof}
\(SFH(Y,\gamma_m,O) = \hfk(K_m,O)\), where \(K_m \subset Y(m)\) is the dual knot. Since \(Y\) is a generalized solid torus, the latter group is Floer simple, hence determined by its Euler characteristic. By Lemma~\ref{Lem:RedAlex}, 
$$ \phi(\chi(\hfk(K_m))) = \frac{k_Y(1-t^{g_Y})^2}{(1-t)^2}.$$
It follows that \(\hfk(K_m,O) \simeq \Z^{k_Y}\). 
\end{proof}

Applying the lemmas to \(Y_2\), which has \(k_{Y_2}=1\), we see that \(SFH(Y_2,\gamma_l,O) \simeq H_*(S^1)\). For \(Y'\), suppose that 
 \(H_1(Y_1(\alpha)) = H_1(Y_1)/\langle \alpha \rangle\) has order \(d\). The torsion subgroup of \(H_1(Y')\)
 is \( H_1(Y_1)/\langle g_{Y_2}\alpha \rangle\), so it has   order \(g_{Y_2}d\), Since \(g_{Y'}=g_{Y_2}\), we see that \(k_{Y'}=d\). Since \(\mu\) is the homological longitude of \(Y'\), \(SFH(Y',\gamma_\mu,O) \simeq H_*(S^1)\otimes \Z^d\). Comparing with equation~\eqref{Eq:OuterSpinc}, we see that \(SFH(Y_1,\gamma_\alpha) \simeq \Z^d\). Now if \(K_\alpha \subset Y_1(\alpha)\) is the dual knot, then \(\hfk(K_\alpha) = SFH(Y_1,\gamma_\alpha) \simeq \Z^d\), where \(d=|H_1(Y_1(\alpha))|\). So \(K_\alpha\) is Floer simple,  which implies that \(Y_1\) is Floer simple and that \(\alpha \) is in the interior of \(\lL(Y_1)\). 
 
 Conversely, if \(\alpha\) is in the interior of \(\lL(Y)\),  Theorem~\ref{Thm:Splice} implies that \(Y_\varphi\) is an L-space, so \(\alpha\) is not NLS detected by \(Y_2\). 
\end{proof}

 Boyer and Clay define \(\alpha\) to be NLS detected if it is NLS detected by some \(N_g\), where \(N_g = M(1/g,1-/g)\) is the original family of Floer homology solid tori discussed above. The proposition shows that \(\alpha\) is NLS detected by one \(N_g\) if and only if it is NLS detected by all \(N_g\) if and only if \(\alpha\) is not the interior of \(\lL(Y)\). This proves Corollary~\ref{Cor:NLS}.

\subsection{Examples}
\label{Subsec:FSTExamples}
We conclude by constructing some examples of generalized solid tori. Some of these  were previously known to Hanselman and Watson \cite{WatsonFST} and Vafaee \cite{VafaeePC}. We start with the following observation.
\begin{cor}
If \(Y\) is an irreducible, semi-primitive generalized solid torus, then 
 \(Y\) is the complement of a closed \(g_Y\)-strand braid in \(S^1 \times S^2\). 
\end{cor}

\begin{proof}
The hypotheses imply that \(\Delta(Y) \sim (1-t^g)/(1-t)\) and that \(H_2(Y, \partial Y)\) is generated by a \(g_Y\)-times punctured sphere. By Corollary~\ref{Cor:TNorm}, it follows that \(Y\) fibres over \(S^1\) with fibre of genus \(0\). 
\end{proof}

By Proposition \ref{Prop:LS1S2}, any knot in \(S^1 \times S^2\) with a lens space surgery is a generalized solid torus.
Cebanu \cite{Cebanu} showed that a knot of this form is a closed braid in \(S^2\). 
 Examples of such knots were studied by Buck, Baker, and Leucona in \cite{BBL}. Many (but not all) of them are derived from knots in the solid torus which have solid torus surgeries. These knots were completely classified by Gabai \cite{GabaiST} and Berge \cite{BergeTorus}. 

To find other examples, we look for braids in \(S^1 \times S^2\) which have L-space surgeries. 
One criterion for finding such examples is given here. Suppose \(\sigma\) is an ordinary \(g\) strand braid in \(D^2\times I\). We can close \(\sigma\) to get a closed braid in \(S^1 \times D^2\). Dehn filling \(S^1 \times D^2\) along \(S^1\times p\) gives the ordinary braid closure \(\overline{\sigma} \subset  S^3\). We can also fill \(S^1 \times D^2\) along \(\partial D^2\) to get a closed braid in \(S^1\times S^2\), which we denote by 
\(\widetilde{\sigma}\). Let \(\Delta \in Br_g\) be the full twist on \(g\)-strands. 

\begin{prop}Suppose that \(\sigma\) is a braid with the property that \(K_n=\overline{\Delta^n\sigma}\) is an  L-space knot in \(S^3\) for all \(n\geq 0\). Then the complement of  \(\widetilde{\sigma}\) is a semi-primitive generalized solid torus. 
\end{prop}

\begin{proof}
Let \(L\subset S^3 \) be the link which is the union of \(K=\overline{\sigma}\) and the braid axis \(B\). The braid \(\tilde{\sigma}\) is the image of \(K\) in the \(S^1\times S^2\) obtained by doing \(0\)-surgery on \(B\). 

 Let \(L(a,c)\) be the manifold obtained by doing \(a\) surgery on \(K\) and \(c\) surgery on \(A\), where \(a \in \Z\) and \(c \in \Q\).   Then \(L(a,-1/n)\) is the result of \(a+ng^2\) surgery on \(K_n\). Using Seifert's algorithm, it is easy to see that there is a constant \(C(\sigma)\) with the property  that \(g(K_n) \leq C(\sigma) + ng(g-1)/2\). Thus  if \(a>2C(\sigma)\), then 
\(a+ng^2 \geq 2g(K_n)-1\) for all \(n\geq 0\). By hypothesis, \(K_n\) is a positive L-space knot, so \(L(a,-1/n)\) is an L-space for all \(n>0\). 

Now let \(\Ybar\) be the manifold obtained by doing \(a\) surgery on \(K\), and let \(Y = \Ybar- \nu(B)\). There is a slope \(\alpha_0 \in Sl(Y)\) so that \(Y(\alpha_0)=L(a,0)\), and a sequence of slopes \(\alpha_{-1/n} \in Sl(Y)\) which converge to \(\alpha_0\) such that 
\( Y(\alpha_{-1/n})=L(a,-1/n)\). It follows that \(Y\) is Floer simple and that \(\alpha_0\) is in the closure of \(\lL(Y)\). Since \(\alpha_0\) is not the homological longitude of \(Y\),  \(\alpha_0 \in \lL(Y)\), so \(L(a,0)\) is an L-space. By Proposition~\ref{Prop:LS1S2}, \(Y\) is a generalized solid torus. 
\end{proof}

We call a closed braid in the solid torus which satisfies the criterion a {\em L-space braid}. Examples include:
\begin{itemize}
\item Knots in  the solid torus with solid torus surgeries ({\em aka} Berge-Gabai knots)
\item The twisted torus knots \(T(p,kp\pm1;2,1)\) studied by Vafaee \cite{Vafaee}
\item  Cables of L-space braids \cite{HomHedden}
\item Satellites where the  pattern knot is a  Berge-Gabai knot and the companion is an L-space braid \cite{HLV}
\end{itemize}

We conclude with two remarks. First, 
we conjecture that every positive one-bridge braid (not just the Berge-Gabai knots) is an L-space braid. Since the knot obtained by applying a full twist to a one-bridge braid is again a one-bridge braid, this is equivalent to showing that the closure of any positive one-bridge braid is an L-space knot in \(S^3\). 
Second, in light of the last two items, it would be interesting to know if a satellite where both the pattern and the companion are L-space braids is also an L-space braid.

\bibliography{Lslope}

\begin{thebibliography}{10}

\bibitem{AiPeters}
Yinghua Ai and Thomas~D. Peters.
\newblock The twisted {F}loer homology of torus bundles.
\newblock {\em Algebr. Geom. Topol.}, 10(2):679--695, 2010.

\bibitem{BBL}
Kenneth Baker, Dorothy Buck, and Ana Lecuona.
\newblock {Some knots in $S^1 \times S^2$ with lens space surgeries}.
\newblock { arXiv:1302.7011}, 2013.

\bibitem{BergeTorus}
John Berge.
\newblock The knots in {$D\sp 2\times S\sp 1$} which have nontrivial {D}ehn
  surgeries that yield {$D\sp 2\times S\sp 1$}.
\newblock {\em Topology Appl.}, 38(1):1--19, 1991.

\bibitem{BoyerCebanu}
Michel Boileau, Steven Boyer, Radu Cebanu, and Genevieve~S. Walsh.
\newblock Knot commensurability and the {B}erge conjecture.
\newblock {\em Geom. Topol.}, 16(2):625--664, 2012.

\bibitem{Bowden}
Jonathan Bowden.
\newblock {Approximating {$C^0$-foliations} by contact structures}.
\newblock arXiv:1509.07709, 2015.

\bibitem{BoyerClay}
Steven Boyer and Adam Clay.
\newblock {Foliations, orders, representations, L-spaces and graph manifolds}.
\newblock arXiv:1401.7726, 2014.

\bibitem{BGW}
Steven Boyer, Cameron~McA. Gordon, and Liam Watson.
\newblock On {L}-spaces and left-orderable fundamental groups.
\newblock {\em Math. Ann.}, 356(4):1213--1245, 2013.

\bibitem{BRW}
Steven Boyer, Dale Rolfsen, and Bert Wiest.
\newblock Orderable 3-manifold groups.
\newblock {\em Ann. Inst. Fourier (Grenoble)}, 55(1):243--288, 2005.

\bibitem{ziggurat}
Danny Calegari and Alden Walker.
\newblock Ziggurats and rotation numbers.
\newblock {\em J. Mod. Dyn.}, 5(4):711--746, 2011.

\bibitem{Cebanu}
Radu Cebanu.
\newblock {\em A Generalization of Property R}.
\newblock 2013.
\newblock Thesis (Ph.D.)--UQAM.

\bibitem{EHN}
David Eisenbud, Ulrich Hirsch, and Walter Neumann.
\newblock Transverse foliations of {S}eifert bundles and self-homeomorphism of
  the circle.
\newblock {\em Comment. Math. Helv.}, 56(4):638--660, 1981.

\bibitem{EliashbergThurston}
Yakov~M. Eliashberg and William~P. Thurston.
\newblock {\em Confoliations}, volume~13 of {\em University Lecture Series}.
\newblock American Mathematical Society, Providence, RI, 1998.

\bibitem{FJR}
Stefan Friedl, Andr{\'a}s Juh{\'a}sz, and Jacob Rasmussen.
\newblock The decategorification of sutured {F}loer homology.
\newblock {\em J. Topol.}, 4(2):431--478, 2011.

\bibitem{Gabai}
David Gabai.
\newblock Foliations and the topology of {$3$}-manifolds.
\newblock {\em J. Differential Geom.}, 18(3):445--503, 1983.

\bibitem{GabaiST}
David Gabai.
\newblock Surgery on knots in solid tori.
\newblock {\em Topology}, 28(1):1--6, 1989.

\bibitem{Hanselman}
Jonathan Hanselman.
\newblock {Splicing integer framed knot complements and bordered Heegaard Floer
  homology}.
\newblock arXiv:1409.1912, 2014.

\bibitem{HRRW}
Jonathan Hanselman, Jacob Rasmussen, Sarah~Dean Rasmussen, and Liam Watson.
\newblock {Taut foliations on graph manifolds}.
\newblock arXiv:1508.05911, 2015.

\bibitem{WatsonFST}
Jonathan Hanselman and Liam Watson.
\newblock {Floer homology solid tori}.
\newblock {In preparation}.

\bibitem{HanWat}
Jonathan Hanselman and Liam Watson.
\newblock {A calculus for bordered Floer homology}.
\newblock {Preprint}, 2015.

\bibitem{HeddenLens}
Matthew Hedden.
\newblock {On Floer homology and the Berge conjecture on knots admitting lens
  space surgeries}.
\newblock {\em Trans. Amer. Math. Soc.}, 363(2):949--968, 2011.
\newblock math/0710.0357.

\bibitem{HomHedden}
Jennifer Hom.
\newblock A note on cabling and {$L$}-space surgeries.
\newblock {\em Algebr. Geom. Topol.}, 11(1):219--223, 2011.

\bibitem{HLV}
Jennifer Hom, Tye Lidman, and Faramarz Vafaee.
\newblock Berge-{G}abai knots and {L}-space satellite operations.
\newblock {\em Algebr. Geom. Topol.}, 14(6):3745--3763, 2014.

\bibitem{HuangRamos}
Yang Huang and Vinicius Ramos.
\newblock {A topological grading on bordered Heegaard Floer homology}.
\newblock { arXiv:1211.7367}, 2012.

\bibitem{JankinsNeumann}
Mark Jankins and Walter~D. Neumann.
\newblock Rotation numbers of products of circle homeomorphisms.
\newblock {\em Math. Ann.}, 271(3):381--400, 1985.

\bibitem{Juhasz1}
Andr{\'a}s Juh{\'a}sz.
\newblock Holomorphic discs and sutured manifolds.
\newblock {\em Algebr. Geom. Topol.}, 6:1429--1457, 2006.

\bibitem{Juhasz}
Andr{\'a}s Juh{\'a}sz.
\newblock Floer homology and surface decompositions.
\newblock {\em Geom. Topol.}, 12(1):299--350, 2008.

\bibitem{JT}
Andras Juhasz and Dylan Thurston.
\newblock {Naturality and mapping class groups in Heegaard Floer homology}.
\newblock {arXiv:1210.4996}, 2012.

\bibitem{KazezRobertsCzero}
William~H. Kazez and Rachel Roberts.
\newblock {Approximating $C^{1,0}$ foliations}.
\newblock 2015.
\newblock arXiv:1404.5919.

\bibitem{LOT}
Robert Lipshitz, Peter Ozsvath, and Dylan Thurston.
\newblock {Bordered Heegaard Floer homology: Invariance and pairing}.
\newblock { arXiv:0810.0687}, 2008.

\bibitem{LOT2}
Robert Lipshitz, Peter~S. Ozsv{\'a}th, and Dylan~P. Thurston.
\newblock Bimodules in bordered {H}eegaard {F}loer homology.
\newblock {\em Geom. Topol.}, 19(2):525--724, 2015.

\bibitem{LiscaMatic}
Paolo Lisca and Gordana Mati{\'c}.
\newblock Transverse contact structures on {S}eifert 3-manifolds.
\newblock {\em Algebr. Geom. Topol.}, 4:1125--1144 (electronic), 2004.

\bibitem{LSIII}
Paolo Lisca and Andr{\'a}s~I. Stipsicz.
\newblock Ozsv\'ath-{S}zab\'o invariants and tight contact 3-manifolds. {III}.
\newblock {\em J. Symplectic Geom.}, 5(4):357--384, 2007.

\bibitem{Naimi}
Ramin Naimi.
\newblock Foliations transverse to fibers of {S}eifert manifolds.
\newblock {\em Comment. Math. Helv.}, 69(1):155--162, 1994.

\bibitem{YiNiFibred}
Yi~Ni.
\newblock Knot {F}loer homology detects fibred knots.
\newblock {\em Invent. Math.}, 170(3):577--608, 2007.

\bibitem{YiNiThurstonNorm}
Yi~Ni.
\newblock Link {F}loer homology detects the {T}hurston norm.
\newblock {\em Geom. Topol.}, 13(5):2991--3019, 2009.

\bibitem{OSHFK1}
Peter Ozsv{\'a}th and Zolt{\'a} Szab{\'o}.
\newblock Holomorphic disks and knot invariants.
\newblock {\em Adv. Math.}, 186:58--116, 2004.
\newblock math.GT/0209056.

\bibitem{OSSF}
Peter Ozsv{\'a}th and Zolt{\'a}n Szab{\'o}.
\newblock On the {F}loer homology of plumbed three-manifolds.
\newblock {\em Geom. Topol.}, 7:185--224 (electronic), 2003.

\bibitem{OSGen}
Peter. Ozsv{\'a}th and Zolt{\'a}n Szab{\'o}.
\newblock Holomorphic disks and genus bounds.
\newblock {\em Geom. Topol.}, 8:311--334, 2004.
\newblock math.GT/0311496.

\bibitem{OSLens}
Peter Ozsv{\'a}th and Zolt{\'a}n Szab{\'o}.
\newblock On knot {F}loer homology and lens space surgeries.
\newblock {\em Topology}, 44:1281--1300, 2005.
\newblock math.GT/0303017.

\bibitem{OSInt}
Peter Ozsv{\'a}th and Zolt{\'a}n Szab{\'o}.
\newblock {Knot Floer homology and integer surgeries}.
\newblock {\em Algebraic \& Geometric Top.}, 8:101--153, 2008.
\newblock math/0410300.

\bibitem{OSRat}
Peter~S. Ozsv{\'a}th and Zolt{\'a}n Szab{\'o}.
\newblock Knot {F}loer homology and rational surgeries.
\newblock {\em Algebr. Geom. Topol.}, 11(1):1--68, 2011.

\bibitem{Ina}
Ina Petkova.
\newblock {An absolute $\mathbb{Z}/2$ grading on bordered Heegaard Floer
  homology}.
\newblock { arXiv:1401.2670}, 2014.

\bibitem{jakethesis}
Jacob Rasmussen.
\newblock Floer homology and knot complements.
\newblock Harvard University thesis. math.GT/0306378, 2003.

\bibitem{LSpaceSurgeries}
Jacob Rasmussen.
\newblock {Lens space surgeries and {L}-space homology spheres.}
\newblock arXiv:0710.2531, 2007.

\bibitem{Turaev}
Vladimir Turaev.
\newblock {\em Torsions of {$3$}-dimensional manifolds}, volume 208 of {\em
  Progress in Mathematics}.
\newblock Birkh\"auser Verlag, Basel, 2002.

\bibitem{VafaeePC}
Faramarz Vafaee.
\newblock {private communication}.

\bibitem{Vafaee}
Faramarz Vafaee.
\newblock {On the knot Floer homology of twisted torus knots}.
\newblock { arXiv:1311.3711}, 2013.

\end{thebibliography}
\bibliographystyle{plain}

\end{document}